\documentclass[11pt,reqno,a4paper]{amsart}
\usepackage{amsmath,amscd,amsbsy,amssymb,url,amsthm,mathtools}
\usepackage{kantlipsum} 
\calclayout
\usepackage{fancyhdr}
\usepackage{extramarks}
\usepackage[plain]{algorithm}
\usepackage{algpseudocode}
\usepackage{epsfig,graphicx,subfigure,arydshln}
\usepackage{balance,mathtools}
\usepackage{mathrsfs, euscript}
\usepackage{stmaryrd}
\usepackage{multirow}
\usepackage{bigdelim}
\usepackage{enumerate}
\usepackage{mathrsfs}
\usepackage{xcolor}
\usepackage{tikz-cd}
\usepackage{tensor}
\usepackage{indentfirst}
\usepackage{csquotes}
\usetikzlibrary{arrows,chains,matrix,positioning,scopes}
\usepackage[title]{appendix}
\usepackage{hyperref}
\usepackage[normalem]{ulem}
\usepackage{url}

\newtheorem{thm}{Theorem}[section]
\newtheorem{defn}[thm]{Definition}
\newtheorem{lem}[thm]{Lemma}

\newtheorem{pro}[thm]{Proposition}
\newtheorem{cor}[thm]{Corollary}
\newtheorem{rmk}[thm]{Remark}

\usetikzlibrary{automata,positioning}
\numberwithin{equation}{section}

\newcommand{\Hom}{\mathrm{Hom}}
\newcommand{\PHom}{\mathrm{PHom}}
\newcommand{\IHom}{\mathrm{IHom}}
\newcommand{\rank}{\mathrm{rank}}

\newcommand{\s}{\mathcal{S}}

\DeclareMathOperator{\Rep}{Rep}
\DeclareMathOperator{\rep}{rep}
\DeclareMathOperator{\Img}{Im}
\DeclareMathOperator{\Ext}{Ext}
\DeclareMathOperator{\Coker}{coker}
\DeclareMathOperator{\add}{add}
\DeclareMathOperator{\red}{red}

\DeclareMathOperator{\In}{in}
\DeclareMathOperator{\out}{out}
\DeclareMathOperator{\E}{E}
\DeclareMathOperator{\e}{e}

\allowdisplaybreaks

\begin{document}
\title{$H$-BASED QUIVERS WITH POTENTIALS AND THEIR REPRESENTATIONS }
\author{XIAOYUE LIN }
\address{School of Mathematical Sciences, Shanghai Jiao Tong University, China}
\email{xiao\_yue\_lin@sjtu.edu.cn}

\thanks{The author was supported in part by National Natural Science Foundation of China (No. 12131015 and No. 12571038).}
\date{}
\dedicatory{}

\begin{abstract}
We generalize Derksen-Weyman-Zelevinsky's theory of quivers with potentials (QPs) \cite{derksen2008quivers} to an $H$-based setting by considering quivers with exactly one loop at each vertex, asking the loops to be nilpotent and so attaching a truncated polynomial ring $H_i$ to each vertex.  The algebra is then defined by taking the quotient of the complete path algebra by relations arising from analogs of the Jacobian ideals of a given potential. We develop the mutation theory for such $H$-based QPs and their decorated representations in general position.  As an application, we consider generalized cluster algebras introduced by Chekhov-Shapiro \cite{Chekhov2013}. For those algebras corresponding to $H$-based quivers $(Q,\mathbf{d})$ that have mutation degree $d_k\leq 2$ at each vertex $k$ and admit nondegenerate potentials $\s$ making $(Q,\mathbf{d},\mathcal{S})$ locally free, we provide a representation-theoretic interpretation of $\mathbf{g}$-vectors and $F$-polynomials. When the exchange matrix $B(Q)$ has full rank, we further construct  generic character for upper generalized cluster algebras.
\end{abstract}
\maketitle\setcounter{tocdepth}{2}
\tableofcontents
\setlength{\parindent}{2em}
 \section{Introduction}
    The study of quivers with potentials (QPs) was systematically developed by Derksen, Weyman, and Zelevinsky in their foundational work \cite{derksen2008quivers,derksen2010quivers}. They introduced the notions of quivers with potential and their decorated representations as a framework to categorify Fomin and Zelevinsky's \cite{fomin2001,fomin2006cluster} cluster algebras. DWZ's construction focused on the skew-symmetric case of cluster algebras, which is encoded by quivers without loops and oriented 2-cycles. 
    
    In this paper, we introduce and study $H$-based quivers with potentials ($H$-based QPs for short), a generalization of the classical QPs, that incorporates controlled loop structures. Specifically, an $H$-based QP is a triple $(Q,\mathbf{d},\s)$, where $Q$ is a quiver with exactly one loop $\varepsilon_i$ at each vertex $i$ and no oriented 2-cycles,  $\mathbf{d}=(d_i)$ is an $|Q_0|$-tuple of positive integers, and $\s$ is a potential containing no loop-only cycles. Each loop $\varepsilon_i$ is nilpotent of degree $d_i$. In this setup, each vertex $i$ has an associated truncated polynomial ring $H_i=K[\varepsilon_i]/\langle\varepsilon_i^{d_i}\rangle$. We associate to $\mathbf{d}$ the algebra $H=\bigoplus_{i\in Q_0}H_i$, called the \textbf{base} of $(Q,\mathbf{d})$. For a potential $\s$, the Jacobian ideal $J(\s)$ is generated by the partial derivatives of $\s$ with respect to the non-loop arrows of $Q$ (unlike classical QPs, where derivatives are taken with respect to all arrows). The Jacobian algebra $\mathcal{P}(Q,\mathbf{d},\s)$ is an $H$-algebra and is determined by non-loop partial derivatives.

The original motivation for our study comes  from the theory of generalized cluster algebras  introduced by Chekhov and
Shapiro \cite{Chekhov2013}. In contrast to cluster algebras where cluster variables satisfy binomial exchange relations, generalized cluster algebras are characterized by reciprocal polynomial exchange relations. This generalized structure appears naturally in several important contexts including  the Teichm{\"u}ller spaces of Riemann surfaces with holes and orbifold points, WKB analysis \cite{Iwaki2015}, representations of quantum affine algebras \cite{gleitz2014}, the Drinfeld double of
$\mathrm{GL}_n$ \cite{Gekhtman2017}.

Remarkably, generalized cluster algebras preserve several fundamental properties of cluster algebras, including
the Laurent phenomenon, finite type classification \cite{Chekhov2013}, tropical dualities phenomenon between $C$-matrices
and $G$-matrices \cite{Nakanishi_2015}, the existence of greedy bases in rank 2 cases \cite{rupel2013},  as well as various combinatorial results \cite{cao2019,cheung2022,Mou2024}.

In this paper, we focus on the skew-symmetric case of generalized cluster algebras. In this case,  the exchange matrix is  a skew-symmetric integer $n\times n$ matrix $B=(b_{i,j})$ and the mutation degree is an $n$-tuple of positive integers $\mathbf{d}=(d_i)$. They correspond to a quiver $Q(B,\mathbf{d})$ with vertices $1,\ldots,n$, where for each pair $(i,j)$ with $b_{i,j}>0$, there are $b_{i,j}$ arrows from $j$ to $i$, and each vertex $i$ has exactly one nilpotent loop  $\varepsilon_i$ of degree $d_i$. For every vertex $k$, the mutation at $k$ transforms $B$ into another skew-symmetric integer $n\times n$ matrix $\mu_k(B)=\overline{B}=(\overline{b}_{i,j})$, whose entries $\overline{b}_{i,j}$ are computed by equation (\ref{muation of B(A)}). The corresponding quiver $Q(\overline{B},\mathbf{d})$ can be obtained from $Q(B,\mathbf{d})$ through the following three-step procedure (see Proposition \ref{mutation of Q}):

\begin{enumerate}
    \item[Step 1.] For each pair $i \xrightarrow{a} k \xrightarrow{b} j$, create $d_k$ \enquote{composite} arrows $[b \varepsilon_k^l a] \colon i \to j$ for $0 \leq l \leq d_k - 1$; thus, whenever $b_{j,k} b_{k,i} > 0$, we create $d_k b_{j,k} b_{k,i}$ new arrows from $i$ to $j$.
    
    \item[Step 2.] Reverse the direction of all arrows incident to $k$; that is, replace each arrow $a \colon i \to k$ with $a^{\star} \colon k \to i$, and each arrow $b \colon j \to k$ with $b^{\star} \colon k \to j$.
    
    \item[Step 3.] Remove a maximal disjoint collection of oriented $2$-cycles (which may appear after creating new arrows in Step 1).
\end{enumerate}

The main results of this paper extend the work of Derksen, Weyman, and Zelevinsky by constructing the mutations of $H$-based QPs and their (decorated) representations. Crucially, just as in the classical case, we develop the Splitting Theorem  for $H$-based QPs.
\begin{thm}[Theorem \ref{splitting thm}]\label{splitting thm in introduction}
 For every $H$-based QP $(H,A,\mathcal{S})$ with the trivial 
 arrow span $A_{\text{triv }}$ and the reduced  arrow span $A_{\text{red }}$, there exist a trivial $H$-based QP $\left(H,A_{\text{triv}}, \mathcal{S}_{\text{triv}}\right)$ and a reduced $H$-based QP $\left(H,A_{\text{red}}, \mathcal{S}_{\text{red}}\right)$ such that $(H,A,\mathcal{S})$ is right-equivalent to the direct sum
 \begin{align*}
    \left(H, A_{\text{triv}}, S_{\text{triv}}\right) \oplus\left(H, A_{\text{red}}, S_{\text {red}}\right).
\end{align*}
 Furthermore, the right-equivalence class of each of the $H$-based QPs $\left(H,A_{\text{triv}}, \s_{\text{triv}}\right)$ and $\left(H,A_{\text{red}}, \s_{\text{red}}\right)$ is determined by the right-equivalence class of $(H,A,\s)$.
\end{thm}
While the first two steps of the DWZ's mutation procedure extend straightforwardly to $H$-based QPs, Theorem \ref{splitting thm in introduction} ensures the  validity of Step 3, that is, the removal of oriented 2-cycles can be accompanied by a modification of the potential that preserves the Jacobian algebra. Furthermore, to ensure that arbitrary sequences of mutations can be applied without creating oriented 2-cycles, we prove in Theorem \ref{eixstence of nondenerate GQP}, as in the classical QPs, that there exists a nondegenerate $H$-based QP for every underlying quiver. 

To establish a framework for mutations of representations, we require a specific class of $H$-based QPs, namely the locally free $H$-based QPs, a notion inspired by the work of Geiss–Leclerc–Schröer \cite{Geiss_2016}. We adopt the following definition:
\begin{defn}[Definition \ref{defn of locally free}]
    For a given finite dimensional $H$-based QP $(H,A,\s)$, we say $(H,A,\s)$ (or the Jacobian algebra $\mathcal{P}(H,A,\s)$) is \textbf{locally free} if $e_k\mathcal{P}(H,A,\s)$ and $\mathcal{P}(H,A,\s)e_k$ are free as left and right $H_k$-modules, respectively. 
\end{defn}

Several mutation invariants for $H$-based QPs are established.
\begin{cor}[Corollary \ref{sum of mutation invariants}, Proposition \ref{locally free is mutation invariants}]
    Suppose $(H,A,\s)$ is a reduced $H$-based QP satisfying (\ref{condition on mutation of GQP 1}), and let $(H,\overline{A}, \overline{\s})=$ $\mu_{k}(H,A,\s)$ be a reduced $H$-based QP obtained from $(H,A,\s)$ by the mutation at $k$. Then the algebras $\mathcal{P}(H,A,\s)_{\hat{k}, \hat{k}}$ and $\mathcal{P}(H,\overline{A}, \overline{\s})_{\hat{k}, \hat{k}}$ are isomorphic to each other, and $\mathcal{P}(H,A, \s)$ is finite-dimensional (resp. locally free) if and only if so is $\mathcal{P}(H,\overline{A}, \overline{\s})$.
\end{cor}
 
The proof of these results heavily depend on an adapted framework of complete path algebras, as originally introduced by Derksen, Weyman, and Zelevinsky \cite{derksen2008quivers}, with essential modifications to suit our $H$-based setting. Having fixed the base $H$ of $Q$, we define for any vertices $i,j$ of $Q$ the bimodule
$$_{i}A_{j}=H_i \textrm{Span}_K\{a\in Q_1^\circ\mid ha=i,ta=j\}H_j=\textrm{Span}_K\{\varepsilon_i^la\varepsilon_j^s\mid a \in Q_1^\circ,1\leq l\leq d_i,1\leq j\leq d_i\},$$ 
where $Q^\circ_1$is the set of arrows obtained by removing all loops from $Q$.
We then associate such pair $(Q,\mathbf{d})$ with the \textbf{arrow span}
\begin{align*}
    A:=\bigoplus_{j\to i\in Q^\circ_1} {_{i}A_{j}},
\end{align*}
which forms a finite-dimensional $H$-bimodule. The \textbf{complete path algebra} of $(Q,\mathbf{d})$ is then defined as the graded tensor algebra  $$\overline{T_H(A)}=\prod_{d=0}^{\infty}A^d,$$ where $A^d$ 
stands for the $d$-fold tensor power of $A$ as an $H$-bimodule. We view 
 $\overline{T_H(A)}$ as a topological algebra via the $m^\circ$-adic topology, where $m^\circ$ is the two-sided ideal generated by $A$.

In studying the mutations of $H$-based QP-decorated representations, we find it difficult to construct mutations for arbitrary decorated representations. To overcome this limitation, we focus exclusively on the mutations of  general (decorated) representations and general presentations \cite{derksen2015general}. More precisely, following \cite{fei2023general}, for any $g\in\mathbb{Z}^{Q_0}$, we consider the projective presentation space
$$\PHom_{\mathcal{P}}(g):=\PHom_{\mathcal{P}}(P([g]_+),P([-g]_+)).$$
where $P(\beta):=\bigoplus_{k\in Q_0}\beta(k)P_k$ for a vector $\beta\in \mathbb{Z}^{Q_0}$, and $P_k$ is the indecomposable projective representation corresponding to vertex $k$. The vector $g$ is called the \textbf{weight vector} of the projective presentation space. As explained in \cite{fei2023general}, there is an open subset $U$ of $\PHom_{\mathcal{P}}(g)$ such that the cokernels of presentations in $U$ lie entirely in an irreducible component $\mathrm{PC}(g)$ of the representation variety of the Jacobian algebra $\mathcal{P}$. A general representation in the \textbf{principal component} $\mathrm{PC}(g)$ is called a \textbf{general representation of weight $g$}. As shown in \cite{plamondon2012}, the principal component is exactly the \textbf{strongly reduced component} introduced in \cite{Gei2011}. Similarly, for the injective case, the injective presentation space $\IHom(\check{g})$ and the irreducible component $\mathrm{PC}(\check{g})$ are defined dually.

\begin{thm}[Theorem \ref{mutation rule of g}, Proposition \ref{general is mutation invarints}]
Let $(Q(B),\mathbf{d},\s)$ be a reduced and locally free $H$-based QP satisfying \eqref{condition on mutation of GQP 1}. Then the mutation at any vertex $k$ sends a general (decorated) representation of weight $g$ (resp. $\check{g}$) to a general representation of weight $g'$ (resp. $\check{g}'$), where $g'$ and  $\check{g}'$ are given by 
\begin{align}
    g'(i)=\begin{cases}
    -g(k) &\text{if $i=k$}\\
    g(i)+d_k[-b_{i,k}]_+[g(k)]_+-d_k[b_{i,k}]_+[-g(k)]_+& \text{if $i\neq k$}.
    \end{cases}
\end{align}
\begin{align}
    \Check{g}'(i)=\begin{cases}
    -\Check{g}(k) &\text{if $i=k$}\\
    \Check{g}(i)+d_k[-b_{k,i}]_+[\check{g}(k)]_+-d_k[b_{k,i}]_+[-\check{g}(k)]_+& \text{if $i\neq k$}.
    \end{cases}
\end{align}
\end{thm}

Recognizing that $F$-polynomials play a fundamental role in classical QP theory, we present an analogous generalization for $H$-based QP representations (with $d_k \leq 2$)  by incorporating Jordan type stratifications of quiver Grassmannians $\mathrm{Gr}_{\mathbf{e},\mathbf{t}}(M)$, where $\mathbf{t}$ records the Jordan block structure of the nilpotent operators $M(\partial_{\varepsilon_k}\s)$ acting on $\ker M(\varepsilon_k)/\mathrm{Im}\,M(\varepsilon_k)$. More precisely, we denote by $t_{k,l}$ the number of Jordan blocks of size $l$ associated with $M(\partial_{\varepsilon_k}\mathcal{S})$, and call $\mathbf{t} = (t_{k,l})_{k,l}$ the \textbf{Jordan type} of $M$.  The quasi-projective variety $\mathrm{Gr}_{\mathbf{e},\mathbf{t}}(M)$ is defined by
\begin{equation}
    \mathrm{Gr}_{\mathbf{e},\mathbf{t}}(M)=\{N\in \mathrm{Gr}_{\mathbf{e}}(M)\mid \text{the Jordan type of $N$ is $\mathbf{t}$}\}.
\end{equation}
\begin{defn}[Definition \ref{defn of F-pplynomial}]
Given a decorated representation $\mathcal{M}=(M,V)$ of $(Q,\mathbf{d},\s)$, its \textbf{$F$-polynomial} is defined as 
\begin{equation}
F_\mathcal{M}(\mathbf{y},\mathbf{z})=\sum_{\mathbf{e},\mathbf{t}} \chi(\mathrm{Gr}_{\mathbf{e},\mathbf{t}}(M)) \mathbf{f}^\mathbf{t}(\mathbf{z})\mathbf{y}^\mathbf{e},
\end{equation}
where $\mathbf{f}^\mathbf{t}(\mathbf{z})=\prod_{i,l}f_{i,l}^{t_{i,l}}(z_{i})$, $\mathbf{y}^\mathbf{e}=\prod_i y_i^{e_i}$, and $\chi(-)$ denotes the topological Euler characteristic, and the polynomials $\mathbf{f}^{\mathbf{t}}(\mathbf{z})$ defined by the recurrence relation \eqref{eq:recursion}.
\end{defn} 

Next we recall the $\mathbf{y}$-seeds and their mutations in generalized cluster algebras. A \textbf{$\mathbf{y}$-seed} in a semifield $\mathbb{P}$ is a pair $(\mathbf{y},B)$, where 
\begin{itemize}
    \item $\mathbf{y}$ is an $n$-tuple of elements in $\mathbb{P}$, and
    \item $B$ is an $n\times n$ skew-symmetrizable integer matrix.
\end{itemize} 
The \textbf{$\mathbf{y}$-seed mutation} at $k\in[1, n]$ transforms $(\mathbf{y}, B)$ into a $\mathbf{y}$-seed $\mu_k(\mathbf{y}, B)= (\mathbf{y'}, B')$ ,
where $B'= \mu_k(B)$, and $\mathbf{y}'$ is given by the mutation rule \eqref{mutation of y}.

As one of our main results, we establish the mutation rule for $F$-polynomials of $H$-based QPs with $d_k\leq 2$:
\begin{thm}[Theorem \ref{mutation of F-polynomial}]
    Let $(Q(B),\mathbf{d},\mathcal{S})$ be a locally free and nondegenerate $H$-based QP with $d_k\leq 2$ for some vertex $k\in Q_0$.  Let $\mathcal{M}=(M,V)$ be a general representation of $(Q(B),\mathbf{d},\mathcal{S})$, and let $\overline{\mathcal{M}}=(\overline{M},\overline{V})=\mu_k(\mathcal{M})$. Suppose that the $\mathbf{y}$-seed $(\mathbf{y}',B_1)$ in $\mathbb{Q}_{\mathrm{sf}}(y_1,\cdots,y_n)$ is obtained from $(\mathbf{y},B)$ by mutation at $k$.
Then the $F$-polynomial $F_{\mathcal{M}}$ and $F_{\overline{\mathcal{M}}}$ are related by 
\begin{equation}
(\sum_{l=0}^{d_k}z_{k,l}y_k^l)^{- [-\check{g}(k)]_+}F_{\mathcal{M}}(y_1,\cdots,y_n)= (\sum_{l=0}^{d_k}z_{k,l}(y'_k)^l)^{-[-\check{g}'(k)]_+}F_{\overline{\mathcal{M}}}(y'_1,\cdots,y'_n),
\end{equation}
where $z_{k,l}=1$ for $l\leq d_k$ when $d_k=1$, and $z_{k,0}=z_{k,2}=1$ with $z_{k,1}=z_k$ when $d_k=2$.
\end{thm}

These results on $\check{g}$-vectors and  $F$-polynomials for $H$-based QP-representations naturally lead us to consider their roles in the context of generalized cluster algebras. Analogous to classical cluster algebras, the structure of generalized cluster algebras is controlled by a family of integer vectors called \textbf{$\mathbf{g}$-vectors}, and a family of integer polynomials called \textbf{$F$-polynomials}. In more detail, we consider the $n$-regular tree $\mathbb{T}_n$ with edges labeled by the set $\{1,\ldots,n\}$ and write $t \overset{k}{\rule[3pt]{6mm}{0.05em}} t'$ if the vertices $t$ and $t'$ of $\mathbb{T}_n$ are connected by an edge labeled with $k$. Additionally, we fix a vertex $t_0\in \mathbb{T}_n$ and a skew-symmetrizable $n\times n$ integer matrix $B$. Associated with $t_0$ and B, we define a family of integer vectors $\mathbf{g}_{l;t}^{B,t_0}$ and a family of integer polynomials $F_{l;t}^{B,t_0}\in \mathbb{Z}[\mathbf{y},\mathbf{z}]$ by the recurrence relations \eqref{defn of g in gca}-\eqref{mutation of F}, where $l=1,\ldots,n$ and $t\in \mathbb{T}_n$.

In the case where the exchange matrices are skew-symmetric and the mutation degree satisfies $d_k\leq 2$ for each $k$, we apply Theorem \ref{mutation rule of g} and Theorem \ref{mutation of F-polynomial} to establish an interpretation of $\mathbf{g}$-vectors and $F$-polynomials in terms of representations of locally free $H$-based QPs. Specifically, as stated in Theorem \ref{thm: interpretation of g and F}, we construct, for given $t_0$, $l$, and $t$ above, an indecomposable $H$-based QP-representation $\mathcal{M}_{l;t}^{B;t_0}$ of $Q(B)$ such that 
$$\mathbf{g}_{l;t}^{B,t_0}=\check{\mathbf{g}}_{\mathcal{M}_{l;t}^{B;t_0}};\quad F_{l;t}^{B,t_0}=F_{\mathcal{M}_{l;t}^{B;t_0}}.$$
This interpretation allows us to derive several properties of $\mathbf{g}$-vectors and $F$-polynomials as detailed in Corollaries \ref{apply 1}-\ref{apply 3}.

Moreover, we extend the construction of the generic character developed in \cite{plamondon2012} for cluster algebras to generalized cluster algebras.  Let $\mathcal{A}(Q,\mathbf{d})$ be the corresponding generalized cluster algebra. We assume that the coefficient semifield of the generalized cluster algebra $\mathcal{A}(Q,\mathbf{d})$ is the tropical semifield $\mathrm{Trop}(\mathbf{z})$ and all $y$-variables are set to 1. 
We define the \textbf{generic character} $C_{\text{gen}}:\mathbb{Z}^n\to \mathbb{Z}[\mathbf{x}^{\pm 1},\mathbf{z}^{\pm 1}]$ by
\begin{equation}
    C_{\text{gen}}(\check{g})=\mathbf{x}^{\check{g}}F_{\ker(\check{g})}(\hat{\mathbf{y}},\mathbf{z}),
\end{equation}
where $\hat{\mathbf{y}}$ is given by $\prod_{j=1}^{n} x_{j}^{b_{j ,k}}$.

Analogous to the upper cluster algebra, the \textbf{upper generalized cluster algebra} $\overline{\mathcal{A}}(Q,\mathbf{d})$ is defined as the intersection:
\begin{equation*}
    \overline{\mathcal{A}}(Q,\mathbf{d})=\bigcap_{(\mathbf{x},\mathbf{y},B)\in \Sigma}\mathcal{L}_{\mathbf{x}},
\end{equation*}
where $\mathcal{L}_{\mathbf{x}}=\mathbb{Z}[\mathbf{x}^\pm,\mathbf{z}]$.
\begin{thm}[Theorem \ref{bases in gca}]\label{bases in gca introduction}
    Suppose that $(Q, \mathbf{d},\mathcal{S})$ is a nondegenerate and locally free $H$-based QP with $d_k\leq 2$ for all vertices $k$, and $B(Q)$ has full rank. Then the generic character $C_{\text{gen}}$ maps $\mathbb{Z}^n$ (bijectively) to a set of linearly independent elements in $\overline{\mathcal{A}}(Q,\mathbf{d})$ containing all cluster monomials.
\end{thm}

 This paper is organized as follows. In Section \ref{section: quiver with potential}, we establish the algebraic foundation for $H$-based quivers, introducing their path algebras and potentials. After defining the right-equivalence relation, we prove the Splitting theorem (Theorem \ref{splitting thm}). In Section \ref{section:mutation of GQP}, we develop the mutation theory for $H$-based QPs.  In Section \ref{section:Mutation Invariants in QP}, we provide some mutation invariants in $H$-based QPs (Proposition \ref{jacobian algebra mutation invariants}, Proposition \ref{finite dim is invariants}, Proposition \ref{locally free is mutation invariants}). In Section \ref{section: nondegenerate}, we characterize nondegenerate $H$-based QPs, establishing their existence for all underlying quivers (Corollary \ref{eixstence of nondenerate GQP}). 

The representation theory of $H$-based QPs is developed in subsequent sections. In Section \ref{section: Decorated Representations}, we introduce decorated representations of $H$-based QPs and their right-equivalence (Definitions \ref{defn of decorated Representations} and \ref{defn of right-equivalence of rep}). In Section \ref{section: general pre}, we begin with a brief review of general presentations \cite{derksen2015general}, and  proceed to study some properties around general presentations, including  Theorem \ref{pc} and Proposition \ref{general rep is locally free}. Using results from Section \ref{section: general pre}, we construct the mutation of $H$-based QP representations in Section \ref{section: mutation of decorated Representations}. In Section \ref{section: mutation of presentatoins}, we then follow the strategy of \cite{fei2024crystalstructureuppercluster} to define the mutation of presentations, proving their compatibility with the mutation of representations (\ref{mutation of pre and rep}) and deriving the $g$-vector mutation rule (Theorem \ref{mutation rule of g}). Section \ref{section: some Mutation invariants in Decorated representations} gives some mutation invariants in decorated representations. We prove that general presentations are well-behaved under mutations (Proposition \ref{general is mutation invarints}), which guarantees the applicability of arbitrary mutation sequences. In Section \ref{section: F-polynomial}, we introduce the $F$-polynomials of $H$-based QPs-representations and their mutation rules (Theorem \ref{mutation of F-polynomial}). We begin Section \ref{section: Application to generalized cluster algebras} by  briefly reviewing generalized cluster algebras, then present in Theorem \ref{thm: interpretation of g and F} and Theorem \ref{bases in gca} two key applications of our results.

In Appendix, we provide detailed proofs for selected technical results from Sections \ref{section: quiver with potential}, \ref{section:mutation of GQP}, \ref{section:Mutation Invariants in QP}, \ref{section: nondegenerate}, \ref{section: mutation of decorated Representations}, where we adapt the stratification technique from \cite{derksen2008quivers} to the $H$-based setting. 
\section{\texorpdfstring{$H$}{}-based Quiver with Potentials}\label{section: quiver with potential}
In this section, we give the concept of $H$-based quivers with potentials and some results in \cite{derksen2008quivers} by generalizing the settings in \cite{derksen2008quivers}. 
\subsection{Quivers and Complete Path Algebras}\label{section: Quivers and Complete Path Algebras}
Let  $Q^\circ:=(Q_0,Q^\circ_1,h,t)$ be a quiver without loops where  $Q_0:=\{1,\cdots,n\}$ is the vertex set and $Q_1^\circ$ is the arrow set with $a:ta\to ha$ for each directed edge $a\in Q_1^\circ$. Now we construct a quiver $Q=(Q_0,Q_1,h,t)$ with the vertex set $Q_0$ and the arrow set $Q_1=Q_1^\circ\cup \{\varepsilon_i,1\leq i\leq n\}$, where $\varepsilon_i$ is a loop at vertex $i$.  Clearly, $Q$ is obtained from $Q^\circ$ by adding a loop at each vertex.

Throughout, let $K$ be a field. Let $Q$ be a quiver defined above and $\mathbf{d}=(d_1,\cdots,d_n)$ be an $n$-tuple of positive integers. We  denote by $H_i$  the  truncated polynomial ring $K[\varepsilon_i]/\langle\varepsilon_i^{d_i}\rangle$ for each vertex $i$. For all $i,j$ in $Q_0$, we define $$_{i}A_{j}=H_i \textrm{Span}_K\{a\in Q_1^\circ,ha=i,ta=j\}H_j.$$ Then $_i A_j$ is an $H_i$-$H_j$ bimodule, which is free as a left $H_i$-module and free as a right $H_j$-module. Now we can associate such pair $(Q,\mathbf{d})$ two objects
\begin{align*}
    H:=\bigoplus_{i=1}^n H_i\quad \text{and}\quad A:=\bigoplus_{j\to i\in Q^\circ_1} {_{i}A_{j}}.
\end{align*}
Clearly, $A$ is an $H$-bimodule. We call $H$ the \textbf{base} of $(Q,\mathbf{d})$ and  call $A$ the \textbf{arrow span} of $(Q,\mathbf{d})$.
\begin{defn}
    The \textbf{path algebra} of $(Q,\mathbf{d})$ is defined as the graded tensor algebra  $$T_H(A)=\bigoplus_{d=0}^{\infty}A^d$$ where $A^d$ denote the $H$-bimodule 
   \begin{align*}
    \underbrace{A\otimes_H A\otimes_H\cdots\otimes_H A}_{d},
\end{align*}
with the convention $A^0=H$. 
\end{defn}

 We denote by $R=K^{Q_0}=\text{Span}_K\{e_1,\cdots,e_n\}$ the \textbf{span of idempotent} of $T_H(A)$, and let $A^\circ:=K^{Q_1^\circ}=\text{Span}_K\{a,a\in Q^\circ_1\}$.  In what follows we write $\otimes _{i}$ for a tensor product $\otimes_{H_i}$. If there is no danger of misunderstanding, we also just write $\otimes$ instead of $\otimes_i$.
\begin{defn}\label{complete path algebra}
The \textbf{complete path algebra} of $Q$ with respect to $\mathbf{d}$ is defined as 
\begin{align*}
    \overline{T_H(A)}:=\prod_{d=0}^{\infty} A^d.
\end{align*}
\end{defn}
Thus the elements of $\overline{T_H(A)}$ are (possibly infinite) $K$-linear combinations of the elements of a path basis in $T_H(A)$; and the multiplication in $\overline{T_H(A)}$ naturally extends the multiplication in $T_H(A)$. 

Let $\mathfrak{m}^\circ=\mathfrak{m}^\circ(A)$ denote the (two-side) ideal of $\overline{T_H(A)}$ given by
\begin{equation}
    \mathfrak{m}^\circ=\prod_{d=1}^{\infty} A^d.
\end{equation}
Thus the power of $\mathfrak{m}^\circ$ are given by $$(\mathfrak{m}^\circ)^n=\prod_{d=n}^{\infty} A^d.$$ 
We view $\overline{T_H(A)}$ as a topological $K$-algebra via the \textbf{$\mathfrak{m}^\circ$-adic topology} having the powers of $\mathfrak{m}^\circ$ as a basic system of open neighborhoods of 0. Thus the closure of any subset $U\subseteq\overline{T_H(A)}$ is given by 
\begin{align}\label{toplogy in complete algebra}
    \overline{U}=\bigcap_{n=0}^\infty (U+(\mathfrak{m}^\circ)^n).
\end{align}
We denote by $H^\circ$ the ideal of $H$ generated by $\varepsilon_i, 1\leq i\leq n=|Q_0|$, and denote by $\mathfrak{m}=\mathfrak{m}^\circ+H^\circ$. 
For any $k\geq 1$, there exists $n\in \mathbb{N}^+$ large enough such that $\mathfrak{m}^n\subseteq(\mathfrak{m}^\circ)^k$ since $H^\circ$ is nilpotent. Recalling the definition of topology in $\overline{T_H(A)}$ (see (\ref{toplogy in complete algebra})), we have that $\mathfrak{m}^\circ$-adic topology and $\mathfrak{m}$-adic topology on $\overline{T_{H}(A)}$ coincide.
\begin{lem}\label{two topologies concides}
   The $\mathfrak{m}^\circ$-adic topology and $\mathfrak{m}$-adic topology on $\overline{T_{H}(A)}$ are equivalent.
\end{lem}
In \cite{derksen2008quivers}, if $A$ and $A'$ are finite-dimensional $H$-bimodules, then any  algebra homomorphism  $\varphi:\overline{T_H(A)}\to \overline{T_H(A')}$ such that  $\varphi|_R=\text{Id}$ is continuous with respect to the  $\mathfrak{m}$-adic topology. It follows from Lemma \ref{two topologies concides} that 
\begin{lem}
    Let $A$ and $A'$ be two finite-dimensional $H$-bimodules, then any algebra homomorphism between $\overline{T_H(A)}$ and $\overline{T_H(A')}$ such that $\varphi|_H=\text{Id}$ is continuous.
\end{lem}
Therefore, $\varphi$ is uniquely determined by its restriction to $A^1=A$. which is an $H$-bimodule homomorphism $\varphi|_A:A\to \mathfrak{m}^\circ(A')=A'\oplus\mathfrak{m}^\circ(A')^2$. We write $\varphi|_A=(\varphi^{(1)},\varphi^{(2)})$, where $\varphi^{(1)}:A\to A'$ and $\varphi^{(2)}:A\to \mathfrak{m}^\circ(A')^2$ are $H$-bimodule homomorphisms.

\begin{pro}\label{determined morphism}
Any pair $(\varphi^{(1)},\varphi^{(2)})$ of $H$-bimodule homomorphisms $\varphi^{(1)}:A\to A'$ and $\varphi^{(2)}:A\to \mathfrak{m}^\circ(A')^2$ give rise to a unique homomrophism of topological algebras $\varphi:\overline{T_H(A)}\to
\overline{T_H(A')}$ such that $\varphi|_H=\text{Id}$, and $\varphi|_A=(\varphi^{(1)},\varphi^{(2)})$.  Furthermore, $\varphi$ is an isomorphism if and only if $\varphi^{(1)}$ is an $H$-bimodule isomorphism $A\to A'$.
\end{pro}
\begin{proof}
The proof is similar to the one given in \cite[Proposition 2.4]{derksen2008quivers}.
\end{proof}
\begin{defn}
 Let $\varphi$ be the automorphism of $\overline{T_H(A)}$ corresponding to a pair $\left(\varphi^{(1)}, \varphi^{(2)}\right)$ as in Proposition \ref{determined morphism}. If $\varphi^{(2)}=0$, then we call $\varphi$ a \textbf{diagonal automorphism}. Furthermore, $\varphi^{(1)}$ is uniquely determined by $\varphi^{(1)}|_{A^{\circ}}:A^\circ\to A^\circ \oplus (A/A^\circ)$. Similarly, we can write $\varphi^{(1)}|_{A^{\circ}}=(\varphi^{(1,1)},\varphi^{(1,2)})$ where  $\varphi^{(1,1)}:A^\circ\to  A^\circ $ and $\varphi^{(1,2)}:A^\circ\to (A/A^\circ)$ are $R$-linear maps. If $\varphi^{(1,2)}=0$ and $\varphi^{(2)}=0$, we call $\varphi$ a \textbf{change of arrows}. If $\varphi^{(1,1)}$ is the identity automorphism of $A^\circ$, we call $\varphi$  a \textbf{quasi-unitriangular automorphism}. If $\varphi^{(1)}$ is the identity automorphism of $A$, we say that $\varphi$ is a \textbf{unitriangular automorphism}.
\end{defn}

\subsection{Potentials and their Jacobian ideals}
In this section we introduce some of our main objects of study: potentials and their Jacobian ideals in the complete path algebra $\overline{T_H(A)}$ given by Definition \ref{complete path algebra}. We fix a path basis in $T_H(A)$; recall that it consists of the elements $\varepsilon_{i}^{l_i} \in H=A^{0}$ for $0\leq l_i\leq d_i-1$ together with the products $\varepsilon_{ha_1}^{l_1}a_{1} \varepsilon_{ta_1}^{l_2}\cdots \varepsilon_{ha_d}^{l_d}a_{d}\varepsilon_{ta_d}^{l_{d+1}}$ (paths) such that all $a_{k}$ belong to $Q^\circ_{1}$, and $t\left(a_{k}\right)=$ $h\left(a_{k+1}\right)$ for $1 \leq k<d$. Then each space $A^{d}$ has a direct $H$-bimodule decomposition $A^{d}=\bigoplus_{i, j \in Q_{0}} A_{i, j}^{d}$, where the component $A_{i, j}^{d}$ is spanned by the paths $\varepsilon_{ha_1}^{l_1}a_{1} \varepsilon_{ta_1}^{l_2}\cdots \varepsilon_{ha_d}^{l_d}a_{d}\varepsilon_{ta_d}^{l_{d+1}}$ with $h\left(a_{1}\right)=i$ and $t\left(a_{d}\right)=j$.
\begin{defn} 
We have the following definition:
\begin{itemize}
    \item For each $d \geq 1$, we define the \textbf{cyclic part} of $A^{d}$ as the sub-$R$-bimodule $A_{\text {cyc }}^{d}=$ $\bigoplus_{i \in Q_{0}} A_{i, i}^{d}$. Thus, $A_{\text {cyc }}^{d}$ is the span of all paths $\varepsilon_{ha_1}^{l_1}a_{1} \varepsilon_{ta_1}^{l_2}\cdots \varepsilon_{ha_d}^{l_d}a_{d}\varepsilon_{ta_d}^{l_{d+1}}$ with $h\left(a_{1}\right)=t\left(a_{d}\right)$; we call such paths \textbf{cyclic}.
    \item We define a closed vector subspace $\overline{T_H(A)}_{\text {cyc}} \subseteq \overline{T_H(A)}$ by setting
    $$\overline{T_H(A)}_{\text{cyc}}=\prod_{d=1}^{\infty} A_{\mathrm{cyc}}^{d}$$
    and call the elements of $\overline{T_H(A)}_{\text {cyc }}$ \textbf{potentials}.
    \item For every $\xi \in (A^\circ)^{\star}$, we define the \textbf{cyclic derivative} $\partial_{\xi}$ as the continuous $K$-linear map $\overline{T_H(A)}_{\mathrm{cyc}} \rightarrow \overline{T_H(A)}$ acting on paths by
    \begin{equation}\label{derivative}
        \begin{aligned}
            \partial_{\xi}\left(\varepsilon_{ha_1}^{l_1}a_{1} \varepsilon_{ta_1}^{l_2}\cdots \varepsilon_{ha_d}^{l_d}a_{d}\varepsilon_{ta_d}^{l_{d+1}}\right)=& \sum_{k=1}^{d} \xi\left(a_{k}\right) \varepsilon_{ha_{k+1}}^{l_{k+1}} a_{k+1} \varepsilon_{ta_{k+1}}^{l_{k+2}} \cdots \varepsilon_{ha_{d}}^{l_{d}} a_{d} \varepsilon_{ta_{d}}^{l_{d+1}}\\ &\varepsilon_{ha_{1}}^{l_{1}}a_{1} \varepsilon_{ta_1}^{l_{2}} \cdots \varepsilon_{ha_{k-1}}^{l_{k-1}} a_{k-1} \varepsilon_{ta_{k-1}}^{l_k}.
        \end{aligned}
    \end{equation}
    \item  For every potential $\mathcal{S}$, we define its \textbf{Jacobian ideal} $J(\mathcal{S})$ as the closure of the (two-sided) ideal in $\overline{T_H(A)}$ generated by the elements $\partial_{\xi}(\mathcal{S})$ for all $\xi \in (A^\circ)^{\star}$; clearly, $J(\mathcal{S})$ is a two-sided ideal in $\overline{T_H(A)}$.
    \item We call the quotient $\overline{T_H(A)} / J(\mathcal{S})$ the \textbf{Jacobian algebra} of $\mathcal{S}$, and denote it by $\mathcal{P}(Q,\mathbf{d}, \mathcal{S})$ or $\mathcal{P}(H,A,\mathcal{S})$.
\end{itemize}
\end{defn}
The definitions imply that  every potential $\mathcal{S}$ belongs to $\mathfrak{m}^\circ(A)^{2}$. An easy check shows that a cyclic derivative $\partial_{\xi}: \overline{T_H(A)}_{\mathrm{cyc}} \rightarrow \overline{T_H(A)}$ does not depend on the choice of a path basis. Furthermore, cyclic derivatives do not distinguish between the potentials that are equivalent in the following sense.
\begin{defn}\label{def of cyc eq}
Two potentials $\mathcal{S}$ and $\mathcal{S}^{\prime}$ are \textbf{cyclically equivalent} if $\mathcal{S}-\mathcal{S}^{\prime}$ lies in the closure of the span of all elements of the form $a_{1} \cdots a_{d}-a_{2} \cdots a_{d} a_{1}$, where $a_{1} \cdots a_{d}$ is a cyclic path with $a_i$ possibly in $A\backslash A^\circ$.
\end{defn}
\begin{pro}\label{cyceq}
If two potentials $\mathcal{S}$ and $\mathcal{S}^{\prime}$ are cyclically equivalent, then $\partial_{\xi}(\mathcal{S})=$ $\partial_{\xi}\left(\mathcal{S}^{\prime}\right)$ for all $\xi \in (A^\circ)^{\star}$, hence $J(\mathcal{S})=J\left(\mathcal{S}^{\prime}\right)$ and $\mathcal{P}(H,A, \mathcal{S})=\mathcal{P}\left(H,A, \mathcal{S}^{\prime}\right)$.
\end{pro}
It is easy to see that the definition of cyclical equivalence does not depend on the choice of a path basis. In fact, it can be given in more invariant terms as follows which can be found in \cite{derksen2008quivers}.
\begin{defn}
    For any topological $K$-algebra $U$, its \textbf{trace space} $\operatorname{Tr}(U)$ is defined as $\operatorname{Tr}(U)=U /\{U, U\}$, where $\{U, U\}$ is the closure of the vector subspace in $U$ spanned by all commutators. We denote by $\pi=\pi_{U}: U \rightarrow \operatorname{Tr}(U)$ the canonical projection.
\end{defn}
\begin{pro}\label{trace}
Two potentials $\mathcal{S}$ and $\mathcal{S}^{\prime}$ are cyclically equivalent if and only if $\pi_{\overline{T_H(A)}}(\mathcal{S})=\pi_{\overline{T_H(A)}}\left(\mathcal{S}^{\prime}\right)$. Thus, the Jacobian ideal and the Jacobian algebra of a potential $\mathcal{S}$ depend only on the image of $\mathcal{S}$ in $\operatorname{Tr}(\overline{T_H(A)})$.
\end{pro}
For $a \in Q_{1}^\circ$, we will use the notation $\partial_{a}$ for the cyclic derivative $\partial_{a^{\star}}$, where $(Q_{1}^\circ)^{\star}=\left\{a^{\star} \mid a \in Q_{1}^\circ \right\}$ is the dual basis of $Q_{1}^\circ$ in $A^{\star}$. As in \cite{derksen2008quivers},  we have some \enquote{differential calculus}.

For every $\xi \in (A^\circ)^{\star}$, we can also define a continuous $K$-linear map
$$
\Delta_{\xi}: \overline{T_H(A)} \rightarrow \overline{T_H(A)} \widehat{\otimes} \overline{T_H(A)}:=\prod_{0\leq d,e}A^d\otimes_K A^e
$$
by setting $\Delta_{\xi}(\varepsilon)=0$ for $\varepsilon \in H=A^{0}$, and
\begin{align}\label{defn of delta}
    \Delta_{\xi}\left(\varepsilon_{ha_1}^{l_1} a_{1} \cdots \varepsilon_{ha_d}^{l_d}a_{d}\varepsilon_{ta_d}^{l_{d+1}}\right)=\sum_{k=1}^{d} \xi\left(a_{k}\right) \varepsilon_{ha_1}^{l_1}a_{1} \cdots a_{k-1} \varepsilon_{ha_{k}}^{l_k}\otimes \varepsilon_{ta_{k}}^{l_{k+1}} a_{k+1} \cdots a_{d} \varepsilon_{ta_d}^{l_{d+1}}
\end{align}
for any path $\varepsilon_{ha_1}^{l_1} a_{1} \cdots \varepsilon_{ha_d}^{l_d}a_{d}\varepsilon_{ta_d}^{l_{d+1}}$ of length $d \geq 1$. Note that $\Delta_{\xi}$ does not depend on the choice of a path basis.

Next, we denote by $(f, g) \mapsto f \square g$ a continuous $K$-bilinear map $\left(\overline{T_H(A)} \widehat{\otimes} \overline{T_H(A)}\right) \times$ $\overline{T_H(A)} \rightarrow \overline{T_H(A)}$ given by
\begin{align}\label{defn of  square}
    (u \otimes_K v) \square g=v g u
\end{align}
for $u, v \in T_H(A)$.

Compared with the cyclic derivative $\partial_\xi$ and differential calculus $\Delta_{\xi}$ defined in \cite{derksen2008quivers}, we simply restrict $\partial_\xi$ and  $\Delta_{\xi}$ to $A^\circ$. Therefore, both the cyclic Leibniz rule and the cyclic chain rule remain valid for $\overline{T_H(A)}$. 

\begin{lem}[Cyclic Leibniz rule]\label{cyclic leibniz rule}
  For any finite sequence of vertices $i_{1}, \ldots, i_{d}, i_{d+1}=i_{1}$, let $f_{1}, \ldots f_{d}\in \overline{T_H(A)}$ such that $f_{k} \in \overline{T_H(A)}_{i_{k}, i_{k+1}}$. For every $\xi \in (A^{\circ})^\star$, we have
$$\partial_{\xi}\left(f_{1} \cdots f_{d}\right)=\sum_{k=1}^{d} \Delta_{\xi}\left(f_{k}\right) \square\left(f_{k+1} \cdots f_{d} f_{1} \cdots f_{k-1}\right).$$
\end{lem}
\begin{lem}[Cyclic chain rule]\label{chain rule}
Suppose that $\varphi: \overline{T_H(A)} \rightarrow \overline{T_H(A^{\prime})}$ is an algebra homomorphism as in Proposition \ref{determined morphism}. Then, for every potential $S \in \overline{T_H(A)}_{\text{cyc}}$ and $\xi \in (A^{\prime\circ})^\star$, we have:
$$
\partial_{\xi}(\varphi(\mathcal{S}))=\sum_{a\in Q_{1}^\circ } \Delta_{\xi}(\varphi(a)) \square \varphi\left(\partial_{a}(\mathcal{S})\right).
$$
\end{lem}
By Lemma \ref{chain rule}, we have the following proposition whose proof is the same as \cite[Proposition 3.7]{derksen2008quivers}
\begin{pro}\label{iso of jacobian algebra}
Every algebra isomorphism
$$
\varphi: \overline{T_H(A)} \rightarrow \overline{T_H(A')} 
$$
such that $\varphi|_{H}=\mathrm{Id}$, sends $J(\mathcal{S})$ onto $J(\varphi(\mathcal{S}))$, inducing an isomorphism of Jacobian algebras $\mathcal{P}(H,A,\mathcal{S}) \rightarrow \mathcal{P}\left(H,A^{\prime}, \varphi(\mathcal{S})\right)$.
\end{pro}

\subsection{\texorpdfstring{$H$}{}-based Quiver with Potentials}
\begin{defn}
Let $Q$ be a quiver that has exactly one loop at each vertex, with the base $H$ and the arrow span $A$. Let $\mathcal{S}\in \overline{T_H(A)}_{\text{cyc}}$ be a potential. We say that a triple  $(H, A,\mathcal{S})$ ( or $(Q,\mathbf{d},\s)$) is an \textbf{$H$-based quiver with potential} ($H$-based QP for short) if it satisfies the following condition:
\begin{align}\label{potiential}
    \text{No two cyclically equivalent cyclic paths appear in the decomposition of $\mathcal{S}$}.
\end{align}
\end{defn}
Condition (\ref{potiential}) excludes, for instance, any non-zero potential $S$ cyclically equivalent to 0.
\begin{defn}\label{defn of right-equivalence of gqp}
 Let $(H,A,\mathcal{S})$ and $(H,A',\mathcal{S}')$ be $H$-based QPs on the same base $H$. By a \textbf{right-equivalence} between $(H,A,\mathcal{S})$ and $(H,A',\mathcal{S}')$ we mean an algebra isomorphism $\varphi: \overline{T_H(A)}\rightarrow \overline{T_H(A')}$ such that $\varphi|_{H}=\mathrm{Id}$, and $\varphi(\mathcal{S})$ is cyclically equivalent to $\mathcal{S}^{\prime}$.
\end{defn}

In view of Proposition \ref{trace}, any algebra homomorphism $ \overline{T_H(A)}\rightarrow \overline{T_H(A')}$ such that $\varphi|_{H}=\mathrm{Id}$, sends cyclically equivalent potentials to cyclically equivalent ones. It follows that right-equivalences of $H$-based QPs have the expected properties: the composition of two right-equivalences, as well as the inverse of a right-equivalence, is again a right-equivalence. Note also that an isomorphism $\varphi: \overline{T_H(A)}\rightarrow \overline{T_H(A')}$ induces an isomorphism of $H$-bimodules $A$ and $A^{\prime}$ (cf. Proposition \ref{determined morphism}), so in dealing with right-equivalent $H$-based QPs we can assume without loss of generality that $A=A^{\prime}$.

In view of Propositions \ref{cyceq} and \ref{iso of jacobian algebra}, any right-equivalence of $H$-based QPs $(H,A,\mathcal{S}) \cong (H,A',\mathcal{S}')$ induces an isomorphism of the Jacobian ideals $J(\mathcal{S}) \cong J(\mathcal{S}')$ and of the Jacobian algebras $\mathcal{P}(H,A,\mathcal{S}) \cong \mathcal{P}(H,A',\mathcal{S}')$.

For every two $H$-based QPs $(H,A,\mathcal{S})$ and $(H,A',\mathcal{S}')$, we can form their direct sum $(H,A,\mathcal{S})\oplus(H,A',\mathcal{S}')=\left(H,A \oplus A^{\prime}, \mathcal{S}+\mathcal{S}^{\prime}\right)$; it is well-defined since both complete path algebras $\overline{T_H(A)}$ and $\overline{T_H(A')}$ have canonical embeddings into $\overline{T_H(A\oplus A')}$ as closed $H$-subalgebras.

Now we begin our analysis of $H$-based QPs with the trivial case.  
\begin{defn}
We say that an $H$-based QP $(H,A,\mathcal{S})$ is \textbf{trivial} if $\mathcal{S} \in A^{2}$ and $\partial \s=A^\circ$, or equivalently, $\mathcal{P}(H,A,\mathcal{S})=H$.
\end{defn}

The following description of trivial $H$-based QPs is seen by standard linear algebra.

\begin{pro}\label{triv}
 An $H$-based QP $(H,A,\mathcal{S})$ with $\mathcal{S} \in A^{2}$ is trivial if and only if the set of arrows $Q^\circ_{1}$ consists of $2N$ distinct arrows $a_{1}, b_{1}, \cdots, a_{N}, b_{N}$ such that each $a_{k} b_{k}$ is a cyclic 2-path, and there is a change of arrow $\varphi$ such that $\varphi(S)$ is cyclically equivalent to $a_{1} b_{1}+\cdots+a_{N} b_{N}$.
\end{pro}
 Returning to general $H$-based QPs, we now see that taking direct sums with trivial $H$-based QPs preserves the Jacobian algebra, following the proof in \cite[Proposition 4.5]{derksen2008quivers}.
\begin{pro}\label{trival part of jacobian}
 If $(H,A,\mathcal{S})$ is an arbitrary $H$-based QP, and $(H,C,\mathcal{T})$ is a trivial one, then the canonical embedding   $\overline{T_H(A)}\rightarrow\overline{T_H(A\oplus C)}$ induces an isomorphism of Jacobian algebras $\mathcal{P}(H,A,\mathcal{S}) \rightarrow \mathcal{P}(H,A \oplus C, \mathcal{S}+\mathcal{T})$.
\end{pro}

For an arbitrary $H$-based QP $(H,A,\mathcal{S})$, let $\mathcal{S}^{\circ(2)}$ denote the degree $2$ homogeneous component of $\mathcal{S}$ in $(A^\circ)^2$.
 We call $(H,A,\mathcal{S})$ \textbf{reduced} if $\mathcal{S}^{\circ (2)}=0$. We define the \textbf{trivial} and \textbf{reduced parts} of $(H,A,\mathcal{S})$ as the finite-dimensional $H$-bimodules given by
$$
A_{\text {triv }}=A_{\text {triv }}(\mathcal{S})=H\partial \mathcal{S}^{\circ (2)} H, \quad A_{\text {red }}=A_{\text {red }}(\mathcal{S})=A / (H\partial \mathcal{S}^{\circ (2)}H).
$$
where
$$
\partial \mathcal{S}^{\circ (2)}:=\left\{\partial_{\xi}(\mathcal{S}^{\circ (2)}) \mid \xi \in (A^\circ)^{\star}\right\} \subseteq A.
$$

We now have the following fundamental result, which will play a crucial role  in later sections.
\begin{thm}[Splitting Theorem]\label{splitting thm}
 For every $H$-based QP $(H,A,\mathcal{S})$ with the trivial 
 arrow span $A_{\text{triv }}$ and the reduced  arrow span $A_{\text{red }}$, there exist a trivial $H$-based QP $\left(H,A_{\text{triv}}, \mathcal{S}_{\text{triv }}\right)$ and a reduced $H$-based QP $\left(H,A_{\text{red}}, \mathcal{S}_{\text{red}}\right)$ such that $(H,A,\mathcal{S})$ is right-equivalent to the direct sum
 \begin{align}\label{existence of split thm}
    \left(H, A_{\text {triv }}, S_{\text {triv }}\right) \oplus\left(H, A_{\text {red }}, S_{\text {red }}\right).
\end{align}
 Furthermore, the right-equivalence class of each of the $H$-based QPs $\left(H,A_{\text{triv }}, \s_{\text{triv }}\right)$ and $\left(H,A_{\text{red }}, \s_{\text{red }}\right)$ is determined by the right-equivalence class of $(H,A,\s)$.
\end{thm}
\begin{proof}
To prove this, we observe that the $m^\circ$-adic and $m$-adic topologies coincide on $\overline{T_H(A)}$. Since the original proof in \cite{derksen2008quivers} relies on the $m$-adic topology of the completed path algebra, the proof follows directly from \cite{derksen2008quivers}.  All technical constructions transfer with the natural adaptation to our setting. See Appendix \ref{app:thm splitting} for details.
\end{proof}
 \begin{defn}
     We call the component $\left(H,A_{\text{red }}, \s_{\text{red }}\right)$ in the decomposition (\ref{existence of split thm}) the \textbf{reduced part} of an $H$-based QP $(H,A,\s)$ (by Theorem \ref{splitting thm}, it is determined by $(H,A,\s)$ up to right-equivalence).
\end{defn}
\begin{defn}\label{2-acyclic}
    We call a quiver with $H$-base $(Q,\mathbf{d})$ (as well as $(H,A)$ ) \textbf{2-acyclic} if $Q^\circ$ has no oriented 2-cycles, i.e., satisfies the following condition:
\begin{align}\label{condition of 2-acyclic}
    \text{ For every pair of vertices $i \neq j$, either $A_{i, j}=\{0\}$ or $A_{j, i}=\{0\}$.}
\end{align}
\end{defn}

In the rest of this section we study the conditions on an $H$-based QP $(H,A,\s)$ guaranteeing that its reduced part is 2-acyclic. We need some preparation.

For a quiver $(Q,\mathbf{d})$ with the  arrow span $A$, let $\mathcal{C}=\mathcal{C}(A)$ denote the set of cyclic paths on $A$ up to cyclical equivalence. Thus, $\mathcal{C}$ is either empty (if $Q^\circ$ has no oriented cycles at all), or countable. The space of potentials up to cyclical equivalence is naturally identified with $K^{\mathcal{C}}$. We say that a $K$-valued function on $K^{\mathcal{C}}$ is \textbf{polynomial} if it depends on finitely many components of a potential $S$ and can be expressed as a polynomial in these components. For a nonzero polynomial function $F$, we denote by $U(F) \subset K^{\mathcal{C}}$ the set of all potentials $S$ such that $F(S) \neq 0$. By a \textbf{regular function} on $U(F)$ we mean a ratio of two polynomial functions on $K^{\mathcal{C}}$ such that the denominator vanishes nowhere on $U(F)$; in particular, any function of the form $G / F^{n}$, where $G$ is a polynomial, is regular on $U(F)$. If $A^{\prime}$ is the arrow span of another quiver $Q^{\prime}$, we say that a map $K^{\mathcal{C}(A)} \rightarrow K^{\mathcal{C}\left(A^{\prime}\right)}$ is \textbf{polynomial} if its every component is a polynomial function; similarly, a map $U(F) \rightarrow K^{\mathcal{C}\left(A^{\prime}\right)}$ is regular if its every component is a regular function on $U(F)$. 

Now let $\left\{a_{1}, b_{1}, \ldots, a_{N}, b_{N}\right\}$ be any maximal collection of distinct arrows in $Q^\circ$ such that $b_{k} a_{k}$ is a cyclic 2-path for $k=1, \ldots, N$. Then the quiver obtained from $Q$ by removing this collection of arrows is clearly 2-acyclic. To such a collection we associate a nonzero polynomial function on $K^{\mathcal{C}(A)}$ given by
\begin{align}\label{det}
    D_{a_{1}, \ldots, a_{N}}^{b_{1}, \ldots, b_{N}}(S)=\operatorname{det}\left(x_{b_{q} a_{p}}\right)_{p, q=1, \ldots, N}
\end{align}
where $x_{b_{q} a_{p}}$ is the sum of the coefficients of $b_{q} a_{p}$ and $a_{p} b_{q}$ in a potential $S$, with the convention that $x_{b_{q} a_{p}}=0$ unless $b_{q} a_{p}$ is a cyclic 2-path.

\begin{pro}\label{condition of 2-cyclic}
     The reduced part $\left(H,A_{\mathrm{red}}, \s_{\mathrm{red}}\right)$ of a $H$-based QP $(H,A,\s)$ is 2-acyclic if and only if the polynomial function $D_{a_{1}, \ldots, a_{N}}^{b_{1}, \ldots, b_{N}}(\s) \neq 0$ for some collection of arrows as above. Furthermore, if $A^{\prime}$ is the arrow span of the quiver obtained from $Q$ by removing all arrows $a_{1}, b_{1}, \ldots, a_{N}, b_{N}$, then there exists a regular map $G: U\left(D_{a_{1}, \ldots, a_{N}}^{b_{1}, \ldots, b_{N}}\right) \rightarrow K^{\mathcal{C}\left(A^{\prime}\right)}$ such that, for any $H$-based QP $(H,A,\s)$ with $\s \in U\left(D_{a_{1}, \ldots, a_{N}}^{b_{1}, \ldots, b_{N}}\right)$, the reduced part $\left(H,A_{\text {red }}, \s_{\text {red }}\right)$ is right-equivalent to $\left(H,A^{\prime}, G(\s)\right)$.
\end{pro}

The proof of Proposition \ref{condition of 2-cyclic} follows by tracing the construction of $\left(H, A_{\text {red }}, S_{\text {red }}\right)$ given in the proof of Lemma \ref{lem in split thm}. Note that we use the following convention. If $A$ is 2-acyclic from the start then the only collection $\left\{a_{1}, b_{1}, \ldots, a_{N}, b_{N}\right\}$ as above is the empty set; in this case, the function $D_{a_{1}, \ldots, a_{N}}^{b_{1}, \ldots, b_{N}}$ is understood to be equal to 1 and $G$ is just the identity mapping.
 
\section{Mutation of \texorpdfstring{$H$}{}-based QPs} \label{section:mutation of GQP}

 Let $(H,A,\s)$ be an $H$-based QP. Suppose that a vertex $k \in Q_{0}$  satisfies the following condition:
\begin{align}\label{condition on mutation of GQP 1}
    \text{For every vertex $i\neq k$, either $A_{i, k}$ or $A_{k, i}$ is zero.}
\end{align}
Replacing $\s$ if necessary with a cyclically equivalent potential, we can also assume that
\begin{align}\label{condition on nutation of GQP 2}
    \text{No cyclic path occurring in the expansion of $\s$ starts (and ends) at $k$.}
\end{align}
Under these conditions, we associate to $(H,A,\s)$ an $H$-based QP 
 $\widetilde{\mu}_{k}(H,A,\s)=(H,\widetilde{A},\widetilde{\s})$ on the same set of vertices $Q_{0}$. We define the homogeneous components $\widetilde{A}_{i, j}$ as follows:
\begin{align}
    \widetilde{A}_{i, j}= \begin{cases}\left(A_{j, i}\right)^{\star} & \text { if } i=k \text { or } j=k ;\\ A_{i, j} \oplus A_{i, k} A_{k, j} & \text { otherwise };\end{cases}
\end{align}
here the product $A_{i, k} A_{k, j}$ is understood as a submodule of $A^{2}\subseteq \overline{T_H(A)}$. Thus, the $H$-bimodule $\widetilde{A}$ is given by
\begin{align}\label{mutation of A}
    \widetilde{A}=\bar{e}_{k} A \bar{e}_{k} \oplus  A e_k A  \oplus (e_{k} A )^\star\oplus\left( A e_{k}\right)^\star,
\end{align}
where we use the notation
$$
\bar{e}_{k}=1-e_{k}=\sum_{i \in Q_{0}-\{k\}} e_{i}.
$$

We associate to $Q_{1}$ the set of arrows $\widetilde{Q}_{1}$ in the following way:
 \begin{enumerate}
     \item [(1)] For each pair $i\xrightarrow{a} k\xrightarrow{b} j$, create $d_k$ new arrows $[b\varepsilon_k^la]$ corresponding to the product $b\varepsilon_k^l a\in \bar{e}_{k} A e_k A \bar{e}_{k}$ for $0\leq l\leq d_k-1$;
     
     \item[(2)] Reverse the direction of all arrows incident to $k$;
     \item[(3)] Keep all the arrows $c \in Q_{1}$ not incident to $k$.
 \end{enumerate}
 
\begin{rmk}
    Let $\{e_k^\star,\varepsilon_k^\star,\ldots,(\varepsilon_k^{d_k-1})^\star\}$ be the standard dual basis of $H_k$. Then $H_k^\star$ is a $H_k$-bimodule as $\varepsilon_k(\varepsilon_k^l)^\star=(\varepsilon_k^l)^\star\varepsilon_k:\varepsilon_k^s\mapsto (\varepsilon_k^l)^\star(\varepsilon_k^{s+1})$. So we have that $\varepsilon_k(\varepsilon_k^l)^\star=(\varepsilon_k^l)^\star\varepsilon_k=(\varepsilon_k^{l-1})^\star$. Thus $(\varepsilon_k^l)^\star\mapsto \varepsilon_k^{d_k-1-l}$ induces an isomorphism of $H_k$-bimodules between $H_k$ and $H_k^\star$. Therefor, in what follows, we write the elements $a^\star(\varepsilon_k^l)^\star$ in $(\widetilde{A})_{i,k}$ $(\text{resp. }(\varepsilon_k^l) b^\star$ in $(\widetilde{A})_{k,j})$ as $a^\star\varepsilon_k^{d_k-1-l}$ $(\text{resp. }\varepsilon_k^{d_k-1-l}b^\star)$.
\end{rmk} 

We now associate to $\s$ the potential $\widetilde{\mu}_{k}(\s)=\widetilde{\s} \in \overline{T_H(\widetilde{A})}$ given by
\begin{align}\label{mutation of s}
    \widetilde{\s}=[\s]+\Delta_k,
\end{align}
where 
\begin{align}
\Delta_{k}=\Delta_{k}(A)=\sum_{l=0}^{d_k-1}[b\varepsilon_k^la]a^\star\varepsilon_k^{d_k-1-l}b^\star
\end{align}
and $[\s]$ is obtained by substituting $\left[a_{p} \varepsilon_k^l a_{p+1}\right]$ for each factor $a_{p} \varepsilon_k^l a_{p+1}$ with $t\left(a_{p}\right)=$ $h\left(a_{p+1}\right)=k$ of any cyclic path $\varepsilon_{ha_1}^{l_1} a_{1
} \cdots \varepsilon_{ha_d}^{l_d}a_{d}\varepsilon_{ta_d}^{l_{d+1}}$ occurring in the expansion of $\s$ (recall that none of these cyclic paths starts at $k$ ). It is easy to see that both $[\s]$ and $\Delta_{k}$ do not depend on the choice of a basis $Q_{1}$ of $A$.

The following proposition is immediate from the definitions.
\begin{pro}\label{trivial part of mutation}
     Suppose an $H$-based QP $(H,A,\s)$ satisfies (\ref{condition on mutation of GQP 1}) and (\ref{condition on nutation of GQP 2}), and an $H$-based QP $\left(H,A^{\prime}, \s^{\prime}\right)$ is such that $e_{k} A^{\prime}=A^{\prime} e_{k}=\{0\}$. Then we have
$$
\widetilde{\mu}_{k}\left(H,A \oplus A^{\prime}, \s+\s^{\prime}\right)=\widetilde{\mu}_{k}(H,A,\s) \oplus\left(H,A^{\prime},\s^{\prime}\right) .
$$
\end{pro}
\begin{thm}\label{mutation uniquele right-equivalence class of the GQP}
    The right-equivalence class of the $H$-based QP $(H,\widetilde{A}, \widetilde{\s})=\widetilde{\mu}_{k}(H,A,\s)$ is determined by the right-equivalence class of $(H,A,\s)$.
\end{thm}

The proof appears in Appendix \ref{app:mutation uniquele right-equivalence class of the GQP}.

Note that even if an $H$-based QP $(H,A,\s)$ is assumed to be reduced, the $H$-based QP $\widetilde{\mu}_{k}(H,A,\s)=(H,\widetilde{A}, \widetilde{\s})$ is not necessarily reduced because the component $[\s]^{\circ (2)} \in \widetilde{A}^{2}$ may be non-zero. Combining Theorem 
 \ref{splitting thm} and Theorem \ref{lem of mutation uniquele right-equivalence class of the GQP}, we obtain the following corollary.
\begin{cor}\label{reduced mutation uniquele right-equivalence class of the GQP}
Suppose  that an $H$-based QP $(H,A,\s)$ satisfies (\ref{condition on mutation of GQP 1}) and (\ref{condition on nutation of GQP 2}). We denote $\widetilde{\mu}_{k}(H,A,\s)$ by $(H,\widetilde{A}, \widetilde{\s})$. Let $(H,\bar{A}, \bar{\s})$ be a reduced $H$-based QP such that
\begin{align}\label{eq in reduced mutation uniquele right-equivalence class of the GQP}
    (H,\widetilde{A}, \widetilde{\s}) \cong\left(H,\widetilde{A}_{\text {triv }}, \widetilde{\s}^{\circ (2)}\right) \oplus(H,\overline{A}, \overline{\s})
\end{align}
 Then the right-equivalence class of $(H,\overline{A}, \overline{\s})$ is determined by the right-equivalence class of $(H, A,\s)$.
 \end{cor}
 \begin{defn}\label{defn of mutation of gqp}
     In the situation of Corollary \ref{reduced mutation uniquele right-equivalence class of the GQP}, we use the notation $\mu_{k}(H,A,\s)=(H,\overline{A}, \overline{\s})$ and call the correspondence $(H,A,\s) \mapsto \mu_{k}(H,A,\s)$ the \textbf{mutation} at vertex $k$.
 \end{defn}
 \begin{thm}\label{mutation of gqp is involution }
     The correspondence $\mu_{k}:(H,A,\s) \rightarrow(H,\overline{A}, \overline{\s})$ acts as an involution on the set of right-equivalence classes of reduced $H$-based QPs satisfying (\ref{condition on mutation of GQP 1}), that is, $\mu_{k}^{2}(H,A, \s)$ is right-equivalent to $(H,A,\s)$.
 \end{thm}
The proof appears in Appendix \ref{app: mutation of gqp is involution}.
\section{Some Mutation Invariants}\label{section:Mutation Invariants in QP}
 In this section, we fix a vertex $k$ and study the effect of the mutation $\mu_{k}$ on the Jacobian algebra $\mathcal{P}(H, A,\s)$. We will use the following notation: for an $H$-bimodule $B$, denote
\begin{align}
    B_{\hat{k}, \hat{k}}=\bar{e}_{k} B \bar{e}_{k}=\bigoplus_{i, j \neq k} B_{i, j}
\end{align}
Note that if $B$ is a (topological) algebra then $B_{\hat{k}, \hat{k}}$ is a (closed) subalgebra of $B$.
\begin{pro}\label{jacobian algebra mutation invariants}
     Suppose an $H$-based QP $(H,A,\s)$ satisfies (\ref{condition on mutation of GQP 1}) and (\ref{condition on nutation of GQP 2}), and let $(H,\widetilde{A}, \widetilde{\s})=\widetilde{\mu}_{k}(H,A,\s)$ be given by (\ref{mutation of A}) and (\ref{mutation of s}). Then the algebras $\mathcal{P}(H,A,\s)_{\hat{k}, \hat{k}}$ and $\mathcal{P}(H,\widetilde{A}, \widetilde{S})_{\hat{k}, \hat{k}}$ are isomorphic to each other.
\end{pro}
\begin{proof}
    In view of (\ref{mutation of A}), we have
    \begin{align}
        \widetilde{A}_{\hat{k}, \hat{k}}=A_{\hat{k}, \hat{k}} \oplus A e_{k} A. 
    \end{align}
Thus, the algebra $\overline{T_H(\widetilde{A}_{\hat{k}, \hat{k}})}$ is generated by the arrows $c \in Q^\circ_{1} \cap A_{\hat{k}, \hat{k}}$ and $[b \varepsilon_k^l a]$ for $a \in Q_{1}^\circ \cap e_{k} A$ and $b \in Q^\circ_{1} \cap A e_{k}$, $0\leq l\leq d_k-1$. The following fact is immediate from the definitions.
\begin{lem}\label{lem in jacobian algebra mutation invariants}
    The correspondence sending each $c \in Q_{1} \cap A_{\hat{k}, \hat{k}}$ to itself, and each generator $[b\varepsilon_k^l a]$ to $b\varepsilon_k^l a$ extends to an algebra isomorphism
$$
    \overline{T_H(\widetilde{A}_{\hat{k}, \hat{k}})} \rightarrow \overline{T_H(A)}_{\hat{k}, \hat{k}}.
$$
\end{lem}

Let $u \mapsto[u]$ denote the isomorphism $\overline{T_H(A)}_{\hat{k}, \hat{k}}\rightarrow \overline{T_H(\widetilde{A}_{\hat{k}, \hat{k}})}$ inverse of that in Lemma \ref{lem in jacobian algebra mutation invariants}. It acts in the same way as the correspondence $S \mapsto[S]$ in (\ref{mutation of s}): $[u]$ is obtained by substituting $\left[a_{p} \varepsilon_k^l a_{p+1}\right]$ for each factor $a_{p} \varepsilon_k^l a_{p+1}$ with $t\left(a_{p}\right)=h\left(a_{p+1}\right)=k$ of any path $\varepsilon_{ha_1}^{l_1}a_{1} \cdots a_{d}\varepsilon_{ta_d}^{l_{d+1}}$ occurring in the path expansion of $u$. Then the proof of Proposition \ref{jacobian algebra mutation invariants} follows immediately from the following lemma.
\begin{lem}\label{lemma 2 in jacion algebra invaraints}
    The correspondence $u \mapsto[u]$ induces an algebra isomorphism $\mathcal{P}(H,A,\s)_{\hat{k}, \hat{k}} \rightarrow \mathcal{P}(H,\widetilde{A}, \widetilde{\s})_{\hat{k}, \hat{k}}$.
\end{lem}
The proof appears in Appendix \ref{app: lemma 2 in jacion algebra invaraints}. 
\end{proof}
\begin{pro}\label{finite dim is invariants}
    In the situation of Proposition \ref{jacobian algebra mutation invariants}, if the Jacobian algebra $\mathcal{P}(H,A,\s)$ is finite-dimensional then so is $\mathcal{P}(H,\widetilde{A}, \widetilde{\s})$.
\end{pro}
The proof appears in Appendix \ref{app:finite dim is invariants}.

Remembering (\ref{eq in reduced mutation uniquele right-equivalence class of the GQP}) and using Proposition \ref{trival part of jacobian}, we see that Proposition \ref{jacobian algebra mutation invariants} and Proposition \ref{finite dim is invariants} have the following corollary.
\begin{cor}\label{sum of mutation invariants}
    Suppose $(H,A,\s)$ is a reduced $H$-based QP satisfying (\ref{condition on mutation of GQP 1}), and let $(H,\overline{A}, \overline{\s})=$ $\mu_{k}(H,A,\s)$ be a reduced $H$-based QP obtained from $(H,A,\s)$ by the mutation at $k$. Then the algebras $\mathcal{P}(H,A,\s)_{\hat{k}, \hat{k}}$ and $\mathcal{P}(H,\overline{A}, \overline{\s})_{\hat{k}, \hat{k}}$ are isomorphic to each other, and $\mathcal{P}(H,A, \s)$ is finite-dimensional if and only if so is $\mathcal{P}(H,\overline{A}, \overline{\s})$.
\end{cor}

We now introduce a key definition following Geiss-Leclerc-Schr\"{o}er \cite{Geiss_2016}, which will play a central role in our subsequent analysis.
\begin{defn}\label{defn of locally free}
    For a given finite dimensional $H$-based QP $(H,A,\s)$, we say $(H,A,\s)$ is \textbf{locally free} if the $H_k$-module $e_k\mathcal{P}(H,A,\s)$ is free. 
\end{defn}
The complete proof of the following proposition will be given in Section \ref{section: mutation of presentatoins}.
 \begin{pro}\label{locally free is mutation invariants}
     Suppose that $(H, A,\s)$ is a reduced and locally free $H$-based QP  satisfying (\ref{condition on mutation of GQP 1}). Then $\mu_k(H,A,\s)$ is also locally free.
 \end{pro}

We see that the class of $H$-based QPs with finite-dimensional (resp. locally free) Jacobian algebras is invariant under mutations. Let us now present another such class.
\begin{defn}
    For every $H$-based QP $(H,A,\s)$, we define its \textbf{deformation space} $\operatorname{Def}(H,A,\s)$ by
\begin{align}
    \operatorname{Def}(H,A,\s)=\operatorname{Tr}(\mathcal{P}(H,A,\s)) / H
\end{align}
\end{defn}

\begin{pro}\label{derformation spaces invariants}
    In the situation of Proposition \ref{jacobian algebra mutation invariants}, deformation spaces $\operatorname{Def}(H,\widetilde{A}, \widetilde{\s})$ and $\operatorname{Def}(H,A,\s)$ are isomorphic to each other.
\end{pro}
\begin{proof}
    In view of Proposition \ref{trace}, $\operatorname{Def}(H,A,\s)$ is isomorphic to $\operatorname{Tr}\left(\mathcal{P}(H,A, \s)_{\hat{k}, \hat{k}}\right) / H_{\hat{k}, \hat{k}}$. Therefore, our assertion is immediate from Proposition \ref{jacobian algebra mutation invariants}.
\end{proof}
\begin{defn}
    We call an $H$-based QP $(H,A,\s)$ \textbf{rigid} if $\operatorname{Def}(H,A,\s)=\{0\}$, or equivalently, if $\operatorname{Tr}(\mathcal{P}(H,A,\mathcal{S})) = H$.
\end{defn}

Combining Propositions \ref{trival part of jacobian} and \ref{derformation spaces invariants}, we obtain the following corollary.
\begin{cor}\label{rigid is invariants}
     If a reduced $H$-based QP $(H,A,\s)$ satisfies (\ref{condition on mutation of GQP 1}) and is rigid, then the $H$-based QP $(H,\overline{A}, \overline{\s})=\mu_{k}(H,A,\s)$ is also rigid.
\end{cor}

\section{Nondegenerate \texorpdfstring{$H$}{}-based QPs}\label{section: nondegenerate}
If we wish to be able to apply to a reduced $H$-based QP $(H, A,\s)$ the mutation at every vertex of $Q_{0}$, the $H$-bimodule $A$ must satisfy (\ref{condition on mutation of GQP 1}) at all vertices. Thus, the arrow span $A$ must be 2-acyclic (see Definition \ref{2-acyclic}). Such an arrow span $A$ can be encoded by a skew-symmetric integer matrix $B=B(A)=\left(b_{i, j}\right)$ with rows and columns labeled by $Q_{0}$, by setting
\begin{align}\label{defn of B(A)}
    b_{i,j}=\frac{1}{d_id_j}(\dim A_{i,j}- \dim A_{j,i}).
\end{align}
Indeed, the dimensions of the components $A_{i, j}$ are recovered from $B$ by
\begin{align}\label{1 in defn of B(A)}
    \dim A_{i,j}=d_id_j\left[b_{i, j}\right]_{+},
\end{align}
where we use the notation
$$
[x]_{+}=\max (x, 0) .
$$
\begin{pro}\label{mutation of Q}
 Let $(H,A,\s)$ be a 2-acyclic reduced $H$-based QP, and suppose that the reduced $H$-based QP $\mu_{k}(H,A,\s)=(H,\overline{A}, \overline{\s})$ obtained from $(H,A,\s)$ by the mutation at some vertex $k$  is also 2-acyclic. Let $B(A)=\left(b_{i, j}\right)$ and $B(\overline{A})=\left(\overline{b}_{i, j}\right)$ be the skew-symmetric integer matrices given by (\ref{defn of B(A)}). Then we have
 \begin{align}\label{muation of B(A)}
     \overline{b}_{i, j}= \begin{cases}-b_{i, j} & \text { if } i=k \text { or } j=k ; \\ b_{i, j}+d_k\left[b_{i, k}\right]_{+}\left[b_{k, j}\right]_{+}-d_k\left[-b_{i, k}\right]_{+}\left[-b_{k, j}\right]_{+} & \text {otherwise. }\end{cases}
 \end{align}
\end{pro}
\begin{proof}
First we note that by Proposition \ref{triv}, if $(H,C,\mathcal{T})$ is a trivial $H$-based QP then $\dim C_{i, j}=$ $\dim C_{j, i}$ for all $i, j$. In view of (\ref{eq in reduced mutation uniquele right-equivalence class of the GQP}), this implies that
\begin{align}\label{1 in muation of B(A)}
    \overline{b}_{i, j}=\operatorname{dim} \overline{A}_{i, j}-\operatorname{dim} \overline{A}_{j, i}=\operatorname{dim} \widetilde{A}_{i, j}-\operatorname{dim} \widetilde{A}_{j, i}.
\end{align}
where $(H,\widetilde{A}, \widetilde{\s})=\widetilde{\mu}_{k}(H,A,\s)$. Using (\ref{mutation of A}), we obtain
$$
\operatorname{dim} \widetilde{A}_{i, j}= \begin{cases}\operatorname{dim} A_{j, i} & \text { if } i=k \text { or } j=k ;\\ \operatorname{dim} A_{i, j}+\frac{1}{d_k}\operatorname{dim} A_{i, k} \operatorname{dim} A_{k, j} & \text { otherwise. }\end{cases}
$$
To obtain (\ref{muation of B(A)}), it remains to substitute these expressions into (\ref{1 in muation of B(A)}) and use (\ref{1 in defn of B(A)}).
\end{proof}
An easy calculation using the obvious identity $x=[x]_{+}-[-x]_{+}$ shows that the second case in (\ref{muation of B(A)}) can be rewritten in several equivalent ways as follows:
$$
\begin{aligned}
\overline{b}_{i, j} & =b_{i, j}+d_k\operatorname{sgn}\left(b_{i, k}\right)\left[b_{i, k} b_{k, j}\right]_{+} \\
& =b_{i, j}+d_k(\left[-b_{i, k}\right]_{+} b_{k, j}+b_{i, k}\left[b_{k, j}\right]_{+})
\end{aligned}
$$
It follows that the transformation $B \mapsto \overline{B}$ given by (\ref{muation of B(A)}) coincides with the matrix mutation at $k$ which plays a crucial part in the theory of generalized cluster algebras (cf. \cite{Nakanishi_2015}).

We see that the mutations of 2-acyclic $H$-based QPs provide a natural framework for matrix mutations. With some abuse of notation, we denote by $\mu_{k}(A)$ the 2-acyclic $H$-bimodule such that the skew-symmetric matrix $B\left(\mu_{k}(A)\right)$ is obtained from $B(A)$ by the mutation at $k$; thus, $\mu_{k}(A)$ is determined by $A$ up to an isomorphism.

Note that the matrix mutations at arbitrary vertices can be iterated indefinitely, while the 2-acyclicity condition (\ref{condition of 2-acyclic}) can be destroyed by an $H$-based QP mutation. We will study $H$-based QPs for which this does not happen.
\begin{defn}\label{defn of nondegenerate}
Let $k_{1}, \ldots, k_{l} \in Q_{0}$ be a finite sequence of vertices such that $k_{p} \neq$ $k_{p+1}$ for $p=1, \ldots, l-1$. We say that an $H$-based QP $(H,A,\s)$ is \textbf{$\left(k_{l}, \cdots, k_{1}\right)$-nondegenerate} if all the $H$-based QPs $$(H,A,\s), \mu_{k_{1}}(H,A, \s), \mu_{k_{2}} \mu_{k_{1}}(H,A,\s), \ldots, \mu_{k_{l}} \cdots \mu_{k_{1}}(H,A,\s)$$ are 2-acyclic (hence well-defined). We say that $(H,A,\s)$ is \textbf{nondegenerate} if it is $\left(k_{l}, \ldots, k_{1}\right)$-nondegenerate for every sequence of vertices as above.
\end{defn}
To state our next result recall the terminology introduced before Proposition \ref{condition of 2-cyclic}. In particular, for a given quiver $(Q,\mathrm{d})$ with the arrow span $A$, the $H$-based QPs on $A$ are identified with the elements of $K^{\mathcal{C}(A)}$. 
\begin{pro}\label{prop of eixstence of nondenerate GQP}
    Suppose that the based field $K$ is infinite, $Q$ is a 2-acyclic quiver with the base $H$ and arrow span $A$, a sequence of vertices $k_{1}, \ldots, k_{\ell}$ is as in Definition \ref{defn of nondegenerate}, and $A^{\prime}=\mu_{k_{l}} \cdots \mu_{k_{1}}(A)$. Then there exist a non-zero polynomial function $F: K^{\mathcal{C}(A)} \rightarrow$ $K$ and a regular map $G: U(F) \rightarrow K^{\mathcal{C}\left(A^{\prime}\right)}$ such that every $H$-based QP $(H,A,\s)$ with $\s \in$ $U(F)$ is $\left(k_{l}, \ldots, k_{1}\right)$-nondegenerate, and, for any $H$-based QP $(A, S)$ with $\s \in U(F)$, the $H$-based QP $\mu_{k_{\ell}} \cdots \mu_{k_{1}}(H,A,\s)$ is right-equivalent to $\left(H,A^{\prime}, G(\s)\right)$.
\end{pro}
\begin{proof}
    The proof follows the same strategy as in \cite[Proposition 7.3]{derksen2008quivers}, 
with the obvious modifications needed for our setting. See Appendix \ref{app:prop of eixstence of nondenerate GQP} for details.
\end{proof}

The following corollary is obtained by applying Proposition \ref{prop of eixstence of nondenerate GQP} and a proof along the lines of \cite[Corollary 7.4]{derksen2008quivers}.
\begin{cor}\label{eixstence of nondenerate GQP}
For every 2-acyclic arrow span $A$, there exists a countable family $\mathcal{F}$ of nonzero polynomial functions on $K^{\mathcal{C}(A)}$ such that the $H$-based QP $(H,A, \s)$ is nondegenerate whenever $\s \in \bigcap_{F \in \mathcal{F}} U(F)$. In particular, if the based field $K$ is uncountable, then there exists a nondegenerate $H$-based QP on $A$.
\end{cor}

\section{Decorated Representations}\label{section: Decorated Representations}
\begin{defn}\label{defn of decorated Representations}
For a given $H$-based QP $(Q,d,\mathcal{S})(\text{or } (H,A,\s))$, a \textbf{decorated representation} of $(Q,d,\mathcal{S})$ is a pair $\mathcal{M}=(M,V)$, where $V$ is a finite-dimensional $H$-module and $V_i$ is $H_i$-free, and $M$ is a finite-dimensional $\overline{T_H(A)}$-module annihilated by $J(\mathcal{S})$. 
\end{defn}
Equivalently, $M$ is a finite-dimensional $\mathcal{P}(H,A,\s)$-module. When appropriate, we will view $M$ as the decorated representation $(M,0)$.  We will sometimes write $\mathcal{M}=(H,A,\s,M,V)$ and refer such a 5-tuple as an $H$-based QP-representation. 

We have $M=\oplus_{i\in Q_0}M_i$ and $V=\oplus_{i\in Q_0} V_i$, where $M_i=e_iM$ and $V_i=e_iV$. With some abuse of notation, for $u\in \overline{T_H(A)}$ or $u\in \mathcal{P}(H,A,\s)$, we denote the multiplication operator $m\mapsto um$ on $M$ simply as $M(u):M\to M$. In particular, for each arrow $a\in Q_1$, we have $M(a):M_{ta}\to M_{ha}$, and $M(a)|_{M_i}=0$ for $i\neq ta$. 

Note that every finite-dimensional $\overline{T_H(A)}$-module $M$ is \textbf{nilpotent}, that is, $M$ is annihilated by $\mathfrak{m}^\circ$ for $n\gg 0$.

In the rest of this section, we present a representation-theoretic interpretation of Theorem \ref{splitting thm}. To do this, we first introduce right-equivalence for $H$-based QP-representations.
\begin{defn}\label{defn of right-equivalence of rep}
    Let $(H,A,\s)$ and $(H,A^\prime,\s^\prime)$ be $H$-based QPs on the same vertices, and let $\mathcal{M}=(M,V)(\text{resp. } \mathcal{M}^\prime=(M^\prime,V^\prime))$ be a decorated representation of $(H,A,\s)(\text{resp. }(H,A^\prime,\s^\prime))$. A \textbf{right-equivalent} between $\mathcal{M}$ and $\mathcal{M^\prime}$ is a triple $(\varphi,\psi,\eta)$, where:
    \begin{itemize}
        \item $\varphi:\overline{T_{H}(A)}\to \overline{T_{H}(A)}$ is a right-equivalence between $(H,A,\s)$ and $(H,A^\prime,\s^\prime)$(see Definition \ref{defn of right-equivalence of gqp});
        \item $\psi:M\to M^\prime$ is an $H$-module isomorphism such $\psi\circ M(u)=M'(\varphi(u))\circ \psi$ for all $u\in A$;
        \item $\eta:V\to V^\prime$ is an isomorphism of $H$-modules.
    \end{itemize}
\end{defn}
Let $\mathcal{M}=(H,A,\s,M,V)$ be an $H$-based QP-representation, and let $\varphi:\overline{T_H(A_{\red}\oplus C)}\to \overline{T_H(A)}$ be a right-equivalence of $H$-based QPs $(H,A_{\red},\s_{\red})\oplus(H,C,\mathcal{T})$ and $(H,A,\s)$, where $(H,A_{\red},\s_{\red})$ is a reduced $H$-based QP and $(H,C,\mathcal{T})$ is a trivial $H$-based QP, see Theorem \ref{splitting thm}. We define a $\overline{T_H(A_{\red})}$-module $M'$ by setting $M'=M$ as a $H$-module with action of $\overline{T_H(A_{\red})}$ given by $M'(u)=M(\varphi(u))$. In view of Proposition \ref{trival part of jacobian}, this makes $\mathcal{M}_{\red}=(H,A_{\red},\s_{red},M',V)$ a well-defined $H$-based QP-representation.
\begin{defn}\label{reduce of M}
    We call the $H$-based QP-representation $\mathcal{M}_{\red}$ given by above construction the \textbf{reduced part} of $\mathcal{M}$.
\end{defn}
This terminology is justified by the following.
\begin{pro}\label{unique of he right-equivalence class of reduced rep}
    The right-equivalence class of $\mathcal{M}_{\red}$ is determined by the right-equivalence class of $\mathcal{M}$.
\end{pro}
\begin{proof}
The same proof as \cite[Proposition 10.5]{derksen2008quivers}  also works here. 
\end{proof}

\section{General Presentations}\label{section: general pre}
The results of this section form a crucial part of our study. Specifically, it provides the foundation for defining the mutations of $H$-based 
QP-representations. Throughout, we assume the Jacobian algebra $\mathcal{P}=\mathcal{P}(Q,\mathbf{d},\s)$ is finite-dimensional, in which case completion is unnecessary and so $\mathcal{P}$ is a basic algebra. 

Let $\rep(Q,\mathbf{d},\s)$ (or $\rep(\mathcal{P})$)  denote the category of all finite-dimensional decorated representations of $(Q,\mathbf{d},\s)$. Let $\rep_\alpha(Q,\mathbf{d},s)$ be the space of $\alpha$-dimensional representation of $\rep(Q,\mathbf{d},\s)$. We note that the category $\rep(Q,\mathbf{d},\s)$ has enough projective and injective objects. Moreover, all projectives and injectives are parametrized by the vertices of $Q$, that is, all indecomposable projectives (resp. injectives)  are of the form $P_k:=\mathcal{P}e_k$(resp. $I_k:=(e_k\mathcal{P})^\star$) for some vertex $k\in Q_0$. They are characterized by the property that $\Hom_{\mathcal{P}}(P_k,\mathcal{M})=M_k$(resp. $\Hom_{\mathcal{P}}(\mathcal{M},I_k)=M_k^\star$) for any $\mathcal{M}\in \rep(Q,\mathbf{d},\s)$.

Following the work of H. Derksen and J. Fei \cite{derksen2015general}, we call a homomorphism between two projective(resp. injective) representations, a \textbf{projective} (resp. \textbf{injective}) \textbf{presentation}. It is convenient to view a presentation $P_-\to P_+$ as an element in the homotopy category $K^b(\text{proj-}\mathcal{P})$ of bounded complexes of projective representations of $\mathcal{P}$.  For any $\beta\in \mathbb{N}^{Q_0}$, we denote $P(\beta):=\oplus_{k\in Q_0}\beta(k)P_k$, where $\beta(k)P_k$ is the direct sum of $\beta(k)$ copies of $P_k$.

\begin{defn}\label{defn of g_M}
The \textbf{$g$-vector} of a projective presentation $$d:P(\beta_-)\to P(\beta_+)$$ is the difference $\beta_--\beta_+\in \mathbb{Z}^{Q_0}$. The 
 \textbf{$\check{g}$-vector} of an injective presentation $$\check{d}:I(\check{\beta}_+)\to I(\check{\beta}_-)$$ is the vector $\check{\beta}_- -  \check{\beta}_+\in \mathbb{Z}^{Q_0}$.
\end{defn}
The $g$-vector is just the corresponding element in the Grothendieck group of $K^b(\text{proj-}\mathcal{P})$. Let $\Rep(\mathcal{P})$ denote the set of isomorphism classes of decorated representations of $(Q,\mathbf{d},\s)$. For each vertex $k \in Q_0$, we define the \textbf{negative generalized simple representation} of $k$ as the decorated representation $E_k^-:=(0,H_k)$. 
 There is a bijection between  two additive categories $\Rep(\mathcal{P})$ and $K^b(\text{proj-}\mathcal{P})$, mapping any representation $M$ to its minimal presentation in $\Rep(\mathcal{P})$, and $E_k^-=(0,H_k)$ to $P_k\to 0$. Now we can naturally extend the classical AR-translation to decorated representation:
 \begin{center}
     \begin{tikzcd}
\mathcal{M} \arrow[r] \arrow[d, leftrightarrow] & \tau\mathcal{M} \arrow[d, leftrightarrow] \\
d_{\mathcal{M}} \arrow[r] & \nu (d_{\mathcal{M}})
\end{tikzcd}
 \end{center}
where $\nu$ is the Nakayama functor (see \cite{elements}).
 Suppose that $\mathcal{M}$ corresponds to a projective presentation $d_{\mathcal{M}}$, then we define the $g$-vector $g_{\mathcal{M}}$ of $\mathcal{M}$ as the $g$-vector of $d_{\mathcal{M}}$. If working with the injective presentations, we can also define the $\check{g}$-vector $\check{g}_{\mathcal{M}}$ of $\mathcal{M}$ in a similar way.

 \begin{defn}
Given any projective presentation $d:P(\beta_-)\to P(\beta_+)$ and any representation  $N\in \rep(\mathcal{P})$, we define $\Hom(d,N)$ and $\E(d,N)$ to be the kernel and cokernel of the induced map:
\begin{equation}\label{defn of E-invariants}
0\xrightarrow{}\Hom(d,N)\xrightarrow{}\Hom_{\mathcal{P}}(P(\beta_+),N)\xrightarrow{}\Hom_{\mathcal{P}}(P(\beta_-),N)\xrightarrow{}\E(d,N)\xrightarrow{}0.
\end{equation}
Similarly for an injective presentation $\check{d}:I(\check{\beta}_+)\to I(\check{\beta}_-)$, we define $\Hom(M,\check{d})$ and $\check{\E}(M,\check{d})$ to be the kernel and cokernel of the induced map $\Hom_{\mathcal{P}}(M,I(\check{\beta}_+))\to \Hom_{\mathcal{P}}(M,I(\check{\beta}_-))$. 
\end{defn}
It is clearly that $\Hom(d,N)=\Hom_{\mathcal{P}}(\Coker(d), N)$ and $\Hom(M,\check{d})=\Hom_{\mathcal{P}}(M,\ker(\check{d}))$. For any two decorated representations $\mathcal{M}=(M,V_1)$ and $\mathcal{N}=(N,V_2)$, we set $\Hom_{\mathcal{P}}(\mathcal{M},\mathcal{N}) :=\Hom(d_{\mathcal{M}},N)=\Hom(M,\check{d}_{\mathcal{N}})$, and $\E_{\mathcal{P}}(\mathcal{M},\mathcal{N}):=\E(d_\mathcal{M},N)$, and $\check{\E}_{\mathcal{P}}(\mathcal{M},\mathcal{N}):=\check{\E}_{\mathcal{P}}(M,\check{d}_\mathcal{N})$. We also set $\E(d_\mathcal{M},d_\mathcal{N}):=\E_{\mathcal{P}}(\mathcal{M},\mathcal{N})$ and $\check{\E}(\check{d}_\mathcal{M},\check{d}_\mathcal{N}):=\check{\E}_{\mathcal{P}}(\mathcal{M},\mathcal{N})$. We denote by $\hom_{\mathcal{P}}(\mathcal{M},\mathcal{N})$ the dimension of $\Hom_{\mathcal{P}}(\mathcal{M},\mathcal{N})$, and denote by $\e_{\mathcal{P}}(\mathcal{M},\mathcal{N})$ the dimension of $\E(d_\mathcal{M},d_\mathcal{N})$. We set $\e_{\mathcal{P}}(\mathcal{M}) := \e_{\mathcal{P}}(\mathcal{M},\mathcal{M})$ for the case when $\mathcal{N} = \mathcal{M}$. 

 We say that a general presentation in $\Hom_{\mathcal{P}}\left(P\left(\beta_{-}\right), P\left(\beta_+\right)\right)$ has property $\heartsuit$ if there is some dense open subset $U$ of $\Hom_{\mathcal{P}}\left(P\left(\beta_{-}\right), P\left(\beta_+\right)\right)$ such that all presentations in $U$ have property $\heartsuit$. For example,  a general presentation $d$ in $\Hom_{\mathcal{P}}\left(P\left(\beta_{-}\right), P\left(\beta_+\right)\right)$ has the following properties: $\Hom(d,N)$ has constant dimension for a fixed $N\in \rep(\mathcal{P})$. If $N=A^\star$, then $\Coker d$  has a constant dimension vector, which we denote by $\underline{\dim}(\beta_--\beta+)$.  As explained in \cite{derksen2015general}, the functions $\hom_{\mathcal{P}}(-,-)$ and $\e_{\mathcal{P}}(-,-)$ are upper semicontinuous on $\PHom_{\mathcal{P}}(g_1)\times \PHom_{\mathcal{P}}(g_2)$ for any $g_1,g_2\in \mathbb{Z}^n$. We will denote their generic (minimal) values by $\hom_{\mathcal{P}}(g_1,g_2)$ and $\e_{\mathcal{P}}(g_1,g_2)$. We shall denote by $\e_{\mathcal{P}}(g)$ the generic value of $\e_{\mathcal{P}}(d,d)$ for $d\in \PHom_{\mathcal{P}}(g)$. Note the difference between $e_{\mathcal{P}}(g)$ and $e_{\mathcal{P}}(g,g)$. The latter does not require the two arguments $d_1$ and $d_2$ to be the same in $e_{\mathcal{P}}(d_1,d_2)$.
\begin{defn}\label{reduced presentation space}
For any $g \in \mathbb{Z}^{Q_{0}}$ we associate the \textbf{reduced presentation space}
$$
\PHom_{\mathcal{P}}(g):=\Hom_{\mathcal{P}}\left(P\left([g]_{+}\right), P\left([-g]_{+}\right)\right).
$$
\end{defn}
 The following lemma is well-known for any finite-dimensional basic algebra, which motivates the Definition \ref{reduced presentation space} (see \cite{derksen2015general}). 
 \begin{lem}\label{general pre is homotopy equivalent}
 A general presentation in $\Hom_{\mathcal{P}}\left(P\left(\beta_{-}\right), P\left(\beta_+\right)\right)$ is homotopy equivalent to a general presentation in $\PHom_{\mathcal{P}}(\beta_- -\beta_+)$.
\end{lem}

A weight vector $g \in K_{0}(\operatorname{proj-}\mathcal{P})$ is called \textbf{positive-free} if a general presentation in  $\PHom_{\mathcal{P}}(g)$ contains no a direct summand of the form $P_{-} \rightarrow 0$.  Any $g \in K_{0}(\operatorname{proj-}\mathcal{P})$ can be decomposed as $g=g^{\prime}+g^{+}$with $g^{\prime}$ positive-free and $g^{+} \in \mathbb{Z}_{>0}^{q}$ such that a general presentation in $\PHom_{\mathcal{P}}(g)$ is a direct sum of a presentation in $\PHom_{\mathcal{P}}(g)$ and a presentation in $\PHom_{\mathcal{P}}(g^+)$. 

\begin{defn}\label{defn of g-coherent}
Let $\mathcal{M}$ be a decorated representation of $(Q,\mathbf{d},\s)$, and its minimal presentation is $$d_{\mathcal{M}}:P(\beta_-)\to P(\beta_+).$$ Then $\mathcal{M}$ is called \textbf{$g$-coherent} if $\min(\beta_+(k),\beta_-(k))=0$ for all vertex $k$, or equivalently, $\beta_-=[g_{\mathcal{M}}]_+$ and $\beta_+=[-g_{\mathcal{M}}]_+$.

Similarly, we give the definition of \textbf{$\check{g}$-coherent}.
\end{defn}
\begin{lem}[{\cite[Lemma 5.11]{fei2017cluster}}]\label{general is g-coherent}
A general presentation in $\PHom_{\mathcal{P}}(g)$ corresponds to a $g$-coherent decorated representation.
\end{lem}

The presentation space $\PHom_{\mathcal{P}}(g)$ comes with a natural group action by
$$\operatorname{Aut}(g):=\operatorname{Aut}_{\mathcal{P}}\left([g]_{+}\right) \times \operatorname{Aut}_{\mathcal{P}}\left([-g]_{+}\right).$$
In \cite[Section 2]{derksen2015general}, the authors considered the incidence variety
\begin{equation}\label{incidence variety}
    \begin{split}
        Z(Y, X) =& \Bigl\{(f, \pi, M) \in Y \times \operatorname{Hom}\left(P([-g]_{+}), K^{\alpha}\right) \times X \mid 
        \pi \in \Hom_{A}\left(P([-g]_{+}), M\right) \\  &\text{ and $P([g]_{+}) \rightarrow P([-g]_{+}) \rightarrow M \rightarrow 0$ is exact } \Bigr\}
    \end{split}
\end{equation}
for any $\operatorname{Aut}_{\mathcal{P}}(g)$-stable subvariety $Y$ of $\PHom_{\mathcal{P}}(g)$ and $\mathrm{GL}_{\alpha}$-stable subvariety $X$ of $\rep_{\alpha}(\mathcal{P})$. It comes with two projections $p_{1}: Z \rightarrow \PHom_{\mathcal{P}}(g)$ and $p_{2}: Z \rightarrow \Rep_{\alpha}(\mathcal{P})$. The projection $p_{1}$ is a principal $\mathrm{GL}_{\alpha}$-bundle over its image (\cite[Theorem 2.3, Lemma 2.4]{derksen2015general}). As discussed in \cite{fei2023general}, there is an open subset $U \subset \operatorname{PHom}_{\mathcal{P}}(g)$ such that $W=p_{2}\left(p_{1}^{-1}(U)\right)$ lies entirely in an irreducible component of $\Rep_{\alpha}(A)$, which is called the \textbf{principal component} of $g$, and denoted by $\operatorname{PC}(g)$ (\cite[Definition 3.5]{fei2023general}). The definition of the principal component $\operatorname{PC}(\check{g})$ of a $\check{g}$-vector $\check{g}$ is similar. When we say a general representation of weight $g$ (resp. $\check{g}$) in $\rep(\mathcal{P})$, we mean a general representation in $\operatorname{PC}(g)$ (resp. $\operatorname{PC}(\check{g})$).

 Next we give a concrete way to compute  $\beta_{+,M}$ and $\beta_{-,M}$ for a general decorated representation $M\in\rep(\mathcal{P})$. Let $S_k$ be the simple representation supported on vertex $k$, that is $\underline{\dim}(S_k)=\theta_k$ (Here $\{\theta_k\}$ is the standard basis of $\mathbb{Z}^{|Q_0|}$).  Suppose that $P(\beta_-)\xrightarrow{d}P(\beta_+)\xrightarrow{} M \to 0$ is the minimal presentation of $M$. Apply $\Hom_{\mathcal{P}}(-,S_k)$ to the presentation and we get 
$$0\xrightarrow{}\Hom_{\mathcal{P}}(M,S_k)\xrightarrow{}\Hom_{\mathcal{P}}(P_+,S_k)\xrightarrow{d}\Hom_{\mathcal{P}}(P_-,S_k)\xrightarrow{}\cdots$$
The minimality of the presentation implies that $d$ is a zero map, and so 
\begin{align}
     \beta_{+,M}(k)=\dim \Hom_{\mathcal{P}}(M,S_k),\quad \beta_{-,M}(k)=\dim \Ext^1_{\mathcal{P}}(M,S_k).
\end{align}
Similarly, we have that 
\begin{align}
    \check{\beta}_{+,M}(k)=\dim \Hom_{\mathcal{P}}(S_k,M),\quad \check{\beta}_{-,M}(k)=\dim \Ext^1_{\mathcal{P}}(S_k,M).
\end{align}
To proceed further, we need  resolutions of $S_k$.
\begin{lem}\label{resolution of s_k}
Part of the projective (resp. injective) resolution of $S_k$ is given by:
\begin{align}\label{sp}
    \cdots \xrightarrow{} P_k\oplus (\bigoplus_{a\in Q^\circ_1,ha=k}P_{ta})\xrightarrow{h_k}P_k\oplus (\bigoplus_{b\in Q^\circ_1,tb=k}P_{hb})\xrightarrow{f_k=(\varepsilon_k, b)_b} P_k\xrightarrow{}S_k\xrightarrow{}0
\end{align}
\begin{align}\label{si}
    0\xrightarrow{} S_k \xrightarrow{} I_k\xrightarrow{f_k'=(\varepsilon_k^\star, a^\star)^T} I_k\bigoplus (\bigoplus_{a\in Q^\circ_1,ha=k}I_{ta})\xrightarrow{h_k^\star}I_k\oplus (\bigoplus_{b\in Q^\circ_1,tb=k}I_{hb})\xrightarrow{}\cdots
\end{align}
where \begin{align*}
    h_k=\begin{pmatrix} \varepsilon_k^{d_k-1} & (\partial_{[\varepsilon_k a]}\s)_a\\
    0 & _{b}(\partial_{[ba]}\mathcal{S})_{a}
    \end{pmatrix},
    \quad 
    h_k^\star= \begin{pmatrix} (\varepsilon_k^{d_k-1})^\star & 0\\
    (\partial_{[b \varepsilon_k]}\s)^\star & _{b}(\partial_{[b a]}\mathcal{S})^\star_{a}.
    \end{pmatrix}
 \end{align*} 
\end{lem}
\begin{proof}
  Clearly, $\Img f_k$ is generated by $\varepsilon_k$ and $b$, so $\Coker f_k=S_k$. 
  The relation $f_kh_k = 0$ immediately implies $\Img h_k \subseteq \ker f_k$. To complete the proof that \eqref{sp} forms a projective resolution of $S_k$, it remains to show the reverse inclusion $\ker f_k \subseteq \Img h_k$.
  
  Let nonzero element $(x_1, x_2)^T+J\in \ker f_k$, where $x_1\in P_k$ and $x_2\in \bigoplus_{b\in Q^\circ_1,tb=k}P_{hb}$. Then $x_1\varepsilon_k+x_2b+J=0$, i.e. $\exists r\in J$ s.t. $x_1\varepsilon_k+x_2b=r$ with $tr=k$. Thus, $r$ consists of terms of the form $p
 \partial_{c}\mathcal{S}q$ where $q$ is a path from $k$ to $hc$ for some arrow $c\in Q_1^\circ$. If $q=q' b\varepsilon_k^t$ for some $0\leq t\leq d_k-1$ , $b \in Q_1^\circ $ with $tb =k$, and path $q'$ starting at $hb$, then: (1) when $t \geq 1$, $p \partial_c \mathcal{S} (q' b \varepsilon_k^{t-1}) \in J$ appears in $x_1$;  (2) when $t = 0$, $p \partial_c \mathcal{S} (q' b) \in J$ appears in $x_2$. Consequently, we may choose a representative of $(x_1, x_2)^T + J$  such that $r$ contains only terms of the form  $q_a
 (\partial_{a}\mathcal{S})\varepsilon_k^l$ and $q_0\varepsilon_k^{d_k-1}$ where $ha=k$ and $0\leq l\leq d_k-1$. That is, $$r=q_0\varepsilon_k^{d_k}+\sum_{a,b\in Q_1^\circ, ha=tb=k}q_{a}((\partial_{[b a]}\mathcal{S})b+\partial_{[\varepsilon_k a]}\s \varepsilon_k)$$ where $q_{a},q_0$ are  linear combinations of paths starting at $ta$ and $k$, respectively. Therefore $x_1=q_0\varepsilon_k^{d_k-1}+\sum_a q_a\partial_{[\varepsilon_k a]}\s$ and $x_2= (\sum_a q_{a}\partial_{[b a]}\mathcal{S})^T_b$ which means $(x_1,x_2)^T+J=h_k((q_0+J,(q_{a}+J)_a)^T)\in \Img h_k.$  This proves that $\ker f_k\subseteq \Img h_k$. 
  
  The proof of the injective resolution is similar.
\end{proof}
Let $\underline{\dim}(M)\in \mathbb{Z}^{|Q_0|}$  be the dimension vector of $M\in \rep(\mathcal{P})$, let $\rank(V_k)$ be the rank of the free $H_k$-module $V_k$. Applying $\Hom_{\mathcal{P}}(-,M)$ to \eqref{sp}, we have the following complex: 
\begin{equation}\label{spm}
    0\xrightarrow{}\Hom_{\mathcal{P}}(S_k,M)\xrightarrow{}M_k\xrightarrow{f_{k;M}}M_k\oplus(\bigoplus_{tb=k}M_{hb})\xrightarrow{h_{k;M}}M_k\oplus (\bigoplus_{ha=k}M_{ta})\xrightarrow{}\cdots.
\end{equation}
Then, by direct calculation, we have that 
\begin{equation}\label{checkgm}
    \begin{aligned}
    -\check{g}_{\mathcal{M}}(k)=&\dim\Hom_\mathcal{P}(S_k,M)-\dim\Ext_\mathcal{P}(S_k,M)-\rank (V_k)\\
    =&\dim M_k-\dim\ker h_{k;M}- \rank (V_k)\\
    =&\rank (h_{k;M})-\sum_{tb=k}\dim M_{hb}-\rank(V_k)
    \end{aligned}
\end{equation}
 The upper-semicontinuity of $\rank (h_{k;M})$ on $\operatorname{PC}(g)$ implies that $\rank(h_{k;M})$ admits a generic value on $\operatorname{PC}(g)$. So a general representation $\mathcal{M}$ in $\operatorname{PC}(g)$ has a constant $\check{g}$-vector, denoted $\check{g}$. Moreover, we will show that a general representation $\mathcal{M}$ in $\operatorname{PC}(g)$ is also a general representation in $\operatorname{PC}(\check{g})$.
 \begin{thm}\label{pc}
Suppose $\mathcal{P}(Q, \mathbf{d},\s)$ is  finite-dimensional, then $\PHom_{\mathcal{P}}(g)$ and $\IHom_{\mathcal{P}}(\check{g})$ have the same principal component.
\end{thm}
\begin{proof}
We follow the strategy of \cite[Theorem 3.11]{fei2023general}. 

It is enough to show that there is an open subset $U$ of $\PHom_{\mathcal{P}}(g)$ and an open subset $\check{U}$ of $\IHom_{\mathcal{P}}(\check{g})$ such that they correspond to the same open subset $W$ in the principal component $\operatorname{PC}(g)$. To achieve this, we need the following key lemma.
\begin{lem}[{\cite[Lemma 3.8]{fei2023general}}]\label{[Lemma 3.8]fei2023general}
    Suppose that $\mathcal{P}$ is a basis algebra. Then:
    \begin{itemize}
        \item[(1)] For any $\operatorname{Aut}_{\mathcal{P}}(g)$-stable subset $X$ of $\PHom_{\mathcal{P}}(g)$ which maps onto an open subset of some irreducible component $C$ of $\Rep_{\alpha}(\mathcal{P})$, there exist open subsets $U$ and $W$ of $X$ and $C$ ,respectively, such that there is an isomorphism of geometric quotients:
    $$U/\operatorname{Aut}_{\mathcal{P}}(g)\to W/\operatorname{GL}_{\underline{\dim}(g)}.$$
    \item [(2)] When $X=\PHom_{\mathcal{P}}(g)$ and $C=\operatorname{PC}(g)$, the dimension of the quotient $U/\operatorname{Aut}_{\mathcal{P}}(g)$ is equal to $\e_{\mathcal{P}}(g)$.
    \end{itemize}
    Dually, we have a similar statement for $\IHom(\check{g})$. In this case , the dimension formula is $\dim(U/\operatorname{Aut}_{\mathcal{P}}(g))=\check{\e}_{\mathcal{P}}(\check{g}
    )$.
\end{lem}
Now let $U$ and $W$ be the open subsets of $\PHom_{\mathcal{P}}(g)$ and $\operatorname{PC}(g)$, respectively, as in Lemma \ref{[Lemma 3.8]fei2023general}. Recall the incidence variety $Z(Y, X)$ in (\ref{incidence variety}) with projections $p_1$ and $p_2$. Let $\check{U}' = p_1(p_2^{-1}(W))$ be the (constructible) subset of $\IHom_{\mathcal{P}}(\check{g})$. In view of Lemma \ref{[Lemma 3.8]fei2023general} (1), by possibly shrinking $\check{U}'$ and $W$, we may assume the isomorphism:

\[\check{U} / \operatorname{Aut}_{\mathcal{P}}(\check{g}) \cong W / \mathrm{GL}_{\underline{\dim}(g)}.\]
So by Lemma \ref{[Lemma 3.8]fei2023general}(2), we have that
\[\dim(\check{U} / \operatorname{Aut}_{\mathcal{P}}) = \dim(W / \mathrm{GL}_{\underline{\dim}(g)}) = \dim(U / \mathrm{Aut}_{\mathcal{P}}(g)) = \e_{\mathcal{P}}(g).\]
To show Theorem \ref{pc}, we also need one more lemma.
\begin{lem}\label{lem1}
For any $\mathcal{P}(Q,S,\mathbf{d})$ decorated representations $\mathcal{M}$ and $\mathcal{N}$, we have the following equalities
\begin{enumerate}
    \item[(1)] $\e_{\mathcal{P}(Q)}(\mathcal{M})=\e_{\mathcal{P}(Q^{op})}(\mathcal{M}^\star)$ and $\check{\e}_{\mathcal{P}(Q)}(\mathcal{M})=\check{\e}_{\mathcal{P}(Q^{op})}(\mathcal{M}^\star)$;
    \item[(2)] $\e_{\mathcal{P}}(\mathcal{M},\mathcal{N})=\hom_{\mathcal{P}}(\mathcal{N},\tau(\mathcal{M}))$ and $\check{\e}_{\mathcal{P}}(\mathcal{M},\mathcal{N})=\hom_{\mathcal{P}}(\tau^{-1}(\mathcal{N}),\mathcal{M})$ where $\tau$ is the classical Auslander-Reiten translation;
    \item[(3)] $\e_{\mathcal{P}}(\mathcal{M})=\check{\e}_{\mathcal{P}}(\mathcal{M})$;
\end{enumerate}
\end{lem}
\begin{proof}
(1) Using (\ref{defn of E-invariants}) and (\ref{checkgm}), we can express $\check{\e}_{\mathcal{P}}(\mathcal{M})$ as 
\begin{equation}\label{em}
    \begin{aligned}
    \check{\e}_{\mathcal{P}}(\mathcal{M})=& \dim \Hom_{\mathcal{P}}(M,M)+  \check{g}_{\mathcal{M}}(\underline{\dim}M)\\
    =& \dim \Hom_{\mathcal{P}}(M,M)+ \sum_{k\in Q_0}\dim M_k(\rank (V_k)-\rank h_{k;M})+\sum_{a\in Q_1^\circ}\dim M_{ha}\dim M_{ta}
\end{aligned}
\end{equation}
We observe that passing from $\mathcal{M}$ to $\mathcal{M}^\star$ does not change any of the terms in (\ref{em}), since $\Hom_{{\mathcal{P}}(Q^{op})}(M^\star,M^\star)$ is isomorphic to $\Hom_{\mathcal{P}}(M,M)$, and $h_{k;M^\star}=(h_{k;M})^\star$. Therefore, the  second equality in the statement holds.

The proof of the first equality is similar.

(2) To emphasize the dependence of indecomposable projective modules $P_k$ (for $k\in Q_0$) on the underlying $H$-based QP $(Q,S,\mathbf{d})$, we will denote them $P_k(Q,S,\mathbf{d})$. The indecomposable injective $\mathcal{P}(Q,S,\mathbf{d})$-modules $I_k(Q,S,\mathbf{d})$ satisfy $$I_k(Q,S,\mathbf{d})=P_k(Q^{op},S^{op},\mathbf{d})^\star.$$
Hence, every projective presentation $d_\mathcal{M}$ gives rise to the injective presentation $d_{\mathcal{M^\star}}$. Then, by definition, we have $\Check{\E}_{\mathcal{P}(Q)}(\mathcal{M},\mathcal{N})=\E_{{\mathcal{P}}(Q^{op})}(\mathcal{N}^\star,\mathcal{M}^\star)$. Recall the Nakayama functor $\nu$  has the property $\nu(P_k(Q,S))=I_k(Q,S)$ for every vertex $k$. Consider the minimal presentation of $\mathcal{M}$
$$P_-\xrightarrow{}P_+\xrightarrow{}\mathcal{M}\xrightarrow{}0.$$
It follows form \cite[Section IV, Proposition 2.4]{elements} that the sequence $$0\xrightarrow{}\tau(\mathcal{M})\xrightarrow{}\nu(P_-)\xrightarrow{}\nu(P_+)$$ is exact. Apply $\Hom_{\mathcal{P}}(N,-)^\star$ to it, we have the following exact sequence $$\Hom_{\mathcal{P}}(P_-,N)\xrightarrow{}\Hom_{\mathcal{P}}(P_+,N)\xrightarrow{}\Hom_{\mathcal{P}}(\mathcal{N},\tau(\mathcal{M}))^\star\xrightarrow{}0.$$
It follows from (\ref{defn of E-invariants}) that $\E_{\mathcal{P}}(\mathcal{M},\mathcal{N})=\Hom_\mathcal{P}(\mathcal{N},\tau(\mathcal{M}))^\star$. We have that $\tau(\mathcal{N}^\star)=\tau^{-1}(\mathcal{N})^\star$, so it follows that
\begin{align*}
    \Check{\E}_{\mathcal{P}(Q)}(\mathcal{M},\mathcal{N})&=\E_{{\mathcal{P}}(Q^{op})}(\mathcal{N}^\star,\mathcal{M}^\star)=\Hom_{\mathcal{P}(Q^{op})}(\mathcal{M}^\star,\tau(\mathcal{N}^\star))^\star\\
    &=\Hom_{\mathcal{P}(Q^{op})}(\mathcal{M}^\star,\tau^{-1}(\mathcal{N})^\star)^\star=\Hom_{\mathcal{P}(Q)}(\tau^{-1}(\mathcal{N}),\mathcal{M}).
\end{align*}

(3) By part (2) we obtain that $\Check{\E}_{\mathcal{P}(Q)}(\mathcal{M})=\E_{\mathcal{P}(Q^{op})}(\mathcal{M}^\star)$. It follows from (1) that $\e_{\mathcal{P}(Q)}(\mathcal{M})=\e_{\mathcal{P}(Q^{op})}(\mathcal{M}^\star)=\check{\e}_{\mathcal{P}(Q)}(\mathcal{M})$.
\end{proof}

Let $d_{\mathcal{M}}$ be a general presentation of weight $g$  such that $\e_{\mathcal{P}}(d_{\mathcal{M}})=\e_{\mathcal{P}}(g)$ and $\check{g}_{\mathcal{M}}=\check{g}$. Let $\check{d}_{\mathcal{M}}$ be the corresponding injective presentation of $\mathcal{M}$. Then we have that $\check{\e}_{\mathcal{P}}(\check{g})\leq \check{\e}_{\mathcal{P}}(\check{d}_{\mathcal{M}})=\e_{\mathcal{P}}(d_{\mathcal{M}})=\e_{\mathcal{P}}(g)$ by part (3) of Lemma \ref{lem1}. 

Hence we have $\dim(\check{U} / \operatorname{Aut}_{\mathcal{P}})\geq \check{\e}_{\mathcal{P}}(\check{g})$. But $\dim \check{U} / \mathrm{Aut}_{\mathcal{P}}(g) = \check{\e}_{\mathcal{P}}(\check{g})$ as well, where $\check{U}$ is an open subset of $\IHom(\check{g})$ as claimed in Lemma \ref{[Lemma 3.8]fei2023general}. So $\dim \check{U}' = \dim \check{U}$. Hence, $\check{U}'$ is in fact open in $\IHom(\check{g})$.
\end{proof}

It follows from Lemma \ref{general is g-coherent} and Theorem \ref{pc} that 
\begin{cor}
    A general representation is both $g$-coherent and $\Check{g}$-coherent.
\end{cor}


 
Next, we introduce the representation-theoretic definition of local freeness for $H$-based QP-representations, which was proposed by Geiss-Leclerc-Schr\"{o}er \cite{Geiss_2016}. 
\begin{defn}\label{defn of locally free of rep}
    A decorated representation $\mathcal{M}=(M,V)$ of $(Q,\mathbf{d},\s)$ is called \textbf{locally free} if $M_i$ is a free $H_i$-module for all $i\in Q_0$.
\end{defn}
\begin{pro}\label{general rep is locally free}
    Let $(Q,\mathbf{d},\s)$ be a locally free $H$-based QP (see Definition \ref{defn of locally free}). Then a general decorated representation of $(Q,\mathbf{d},\s)$ is locally free.
\end{pro}
\begin{proof}
     We consider a projective presentation of a representation $M: P({\beta_-})\xrightarrow{d} P({\beta_+})\twoheadrightarrow M$. We can view $d$ as a matrix whose entries are linear combinations of paths. For  any $k\in Q_0$,  we apply $\Hom_{\mathcal{P}}(P_k,-)$ to this projective presentation and get a exact sequence:
     $$\Hom_{\mathcal{P}}(P_k,P(\beta_-))\xrightarrow{\Hom_{\mathcal{P}}(P_k,d)}\Hom_{\mathcal{P}}(P_k,P(\beta_+))\xrightarrow{}\Hom_{\mathcal{P}}(P_k,M)\xrightarrow{}0.$$ 
     Using the fact that $\Hom_{\mathcal{P}}(P_k,M')=M'_k$ for any representations $M'$ , we can rewrite the exact sequence as 
 $$P(\beta_-)_{k}\xrightarrow{d}P(\beta_+)_{k}\xrightarrow{}M_{k}\xrightarrow{}0.$$
     Since $(Q,\mathbf{d},\s)$ is locally free, $P(\beta_\pm)$ are free $H_k$-modules. Because the cokernel of a general homomorphism between two free $H_k$-module is free, one can see that $M_k$ must be free when $d$ is in general position.
\end{proof}

\section{Mutation of Decorated Representations}\label{section: mutation of decorated Representations}
The aim of this section is to extend the definition of $H$-based QP-mutation in Corollary \ref{reduced mutation uniquele right-equivalence class of the GQP} and Definition \ref{defn of mutation of gqp} to the level of $H$-based QP-representations, and to prove a representation-theoretic extension of Theorem \ref{mutation of gqp is involution }. Before proceeding, we first need to establish some basic setup.

Throughout, we assume that $(Q,\mathbf{d},\s)$ is a locally free $H$-based QP with arrow span $A$. Recall that for any $i,j$ in $Q_0$, we denote $_i A_j=e_i Ae_j$ which is free as a left $H_i$-module and free as a right $H_j$-module. Let $_{j}A^\star_{i}$ be the corresponding $H_j$-$H_i$ bimodule coming from the dual bimodule $A^\star$. Then \cite[Section 5]{Geiss_2016}  shows that there is an $H_i$-$H_j$-bimodule isomorphism between $_{i}A_j$ and $\Hom_{H_j}({_{j}A^\star_i}, H_j)$( or $\Hom_{H_i}({_{j}A_i}, H_i)$). Now for any decorated representation $\mathcal{M}=(M,V)$ representation of $(Q,\mathbf{d},\s)$, inspired by \cite{Geiss_2016}, we have the following definitions. 

Note that $M_i$ is an $H_i$-module via the map $M(\varepsilon_i)$. For $i,j\in  Q_0$, if $j\to i\in Q_1^\circ$, we define a left $H_i$-module morphism  $$M_{i, j}:{_{i}A_{j}} \otimes M_j\to M_i$$ by 
\begin{align}
    M_{i,j}(\varepsilon_i^l a\varepsilon_j^f\otimes m)=M(\varepsilon_i)^l\circ M(a)\circ M(\varepsilon_j)^f(m).
\end{align}
If $i\to j\in  Q_1^{\circ}$, we define a left $H_i$-module morphism  $$\check{M}_{j, i}:M_i\to {_{i}A^\star_j} \otimes M_j$$ by
\begin{align}
    \check{M}_{j, i}(m)=\sum_{\substack{0\leq f\leq d_k-1\\a\in e_jAe_i\cap Q^{\circ}_1}}\varepsilon_i^f a^\star \otimes (M(a)\circ M(\varepsilon_i)^{d_i-f-1}(m)).
\end{align}

 For $k\in Q_0$, we  define the following vertex sets :
\begin{align*}
   \In(k):=\{i\in Q_0\mid \exists a\in Q_1^\circ, ta=i, ha=k \}, \quad \out(k):=\{j\in Q_0 \mid \exists b\in Q_1^\circ, ta=k, ha=j\}.
\end{align*}
Then we construct the following $H_k$-modules: 
\begin{align}\label{defn of M_in and M_out}
    M_{k,\In}:=\bigoplus_{i\in \In(k)}{_{k}A_{i}}\otimes M_i,\quad M_{k,\out}:=\bigoplus_{j\in \out(k)} {_{k}A^\star_{j}}\otimes M_j.
\end{align}
Next, we define  two $H_k$-module homomorphisms:
\begin{align}
    \alpha_k&:=(M_{k,i})_{i\in \In(k)}:M_{k,\In}\longrightarrow M_k,\label{defn of alpha}\\ 
    \beta_k&:=(\check{M}_{j,k})_{j\in \out(k)}:M_{k}\longrightarrow M_{k,\out}.\label{defn of beta}
\end{align}
\begin{rmk}
    This implies that  any  representation $M$ of $(H,A,\s)$ gives a tuple $(M_k,M_{k,i},\check{M_{j,k}})$ where $M_k$ is an $H_k$-module for all $k\in Q_0$, $M_{k,i}: {_kA_i}\otimes M_i\to M_k$ and  $\check{M}_{j,k}: M_k\to {_k A^\star_j}\otimes M_j$ are $H_k$-linear maps defined above for all $i,j\in Q_0$ with  $i\to k,k\to j\in Q_1^\circ$.
    
     Conversely,  such a tuple $(M_i,M_{k,i},\check{M}_{j,k})$ define a representation $M$ as follows: 
     \begin{itemize}
         \item $M=\bigoplus_{k}M_k$ as $H$-module;
         \item The action of an arrow $a \in Q_1^\circ$ on $ M$ is given by: $M(a)(m)=M_{ha,ta}(m)$,
         or equivalently,  $M(a)(m)=(\check{M}_{ha,ta}(m))_a$ where the element $(\check{M}_{ha,ta}(m))_a$ is uniquely determined by $\check{M}_{ha,ta}(m)=\sum\limits_{\substack{0\leq f\leq d_k-1\\a \in e_{ha}Ae_{ta}\cap Q^\circ_1}} \varepsilon_{ta}^f a^\star\otimes (\check{M}_{ha,ta}(m))_{a\varepsilon_{ta}^{d_k-1-f}}$.
     \end{itemize} 
\end{rmk}
We also introduce an $H_k$-module morphism $\gamma_k=(\gamma_{i,j}):M_{k,\out}\longrightarrow M_{k,\In}$ as follows. Replacing $\s$ if necessary by a cyclically equivalent potential, we may assume that $\s\in \overline{T_{H}(A)}_{\hat{k},\hat{k}}$. This allow us define the component $\gamma_{i,j}$ by  setting
\begin{equation}\label{defn of gamma}
\begin{aligned}
    \gamma_{i,j} \colon {_{k}A^\star_{j}} \otimes M_j &\longrightarrow {_{k}A_{i}} \otimes M_i \\
    \varepsilon_k^f b^\star \otimes m &\longmapsto \sum_{\substack{0\leq l\leq d_k-f-1\\a\in e_kAe_i\cap Q^\circ_1}} \varepsilon_k^{l+f} a \otimes M(\partial_{[b\varepsilon_k^{l}a]}\mathcal{S})(m)
\end{aligned}
\end{equation}

for $b:k\to j\in Q_1^\circ$.
\begin{lem}
We have the following triangle of $H_k$-linear maps with $\alpha_k\gamma_k=0$ and $\gamma_k\beta_k=0$.
\begin{equation}\label{triangle}
    \begin{tikzcd}
  & M_k \arrow[rd, "\beta_k"] &                                     \\
{M_{k,in}} \arrow[ru, "\alpha_k"] &  & {M_{k,out}} \arrow[ll, "\gamma_k"']
\end{tikzcd}
\end{equation}
\end{lem}
\begin{proof}
For any $\varepsilon_k^f b^\star\otimes m\in {_{k}A^\star_{j}}\otimes M_j$,
\begin{align*}
    \alpha_k\gamma_k(\varepsilon_k^f b^\star\otimes m)&=\alpha_k\Big(\sum_{\substack{0\leq l\leq d_k-f-1\\a\in e_kA\cap Q^\circ_1}}\varepsilon_k^{l+f} a\otimes M(\partial_{[b\varepsilon_k^{l}a]}S)(m)\Big)\\
    &=\sum_{\substack{0\leq l\leq d_k-1\\a\in e_kA\cap Q^\circ_1}}M(\varepsilon_k^{l+f})M(a)M(\partial_{[b\varepsilon_k^la]}\mathcal{S})(m)\\
    &=M(\varepsilon_k^{f})M(\partial_{b}\mathcal{S})(m).
\end{align*}
For any $m\in M_k$,
\begin{align*}
    \gamma_k\beta_k(m)&=\gamma_k\Big(\sum_{\substack{0\leq f\leq d_k-1\\b\in Ae_k\cap Q^\circ_1}}\varepsilon_k^f b^\star \otimes (M(b)\circ M(\varepsilon_k)^{d_k-f-1}(m))\Big)\\
    &=\sum_{\substack{0\leq f\leq d_k-1\\a\in e_kA\cap Q^\circ_1,b\in Ae_k\cap Q^\circ_1}}\sum_{0\leq l\leq d_k-1}\varepsilon_k^{l+f}a\otimes M(\partial_{[b\varepsilon_k^l a]}\mathcal{S}\cdot b\cdot  \varepsilon_k^{d_k-f-1})(m)\\
    &=\sum_{\substack{0\leq f\leq d_k-1\\a\in e_kA\cap Q^\circ_1}}\varepsilon_k^f a\otimes M(\partial_{a}\mathcal{S}\cdot\varepsilon_k^{d_k-f-1})(m).
\end{align*}
Therefor, $\alpha_k\gamma_k=0$ and $\gamma_k\beta_k=0$.
\end{proof}

Let $E_k$ be the indecomposable representation with $E_k(i)=0$ for $i\neq k$ and $E(k)=H_k$. We call it the \textbf{generalized simple representations} corresponding to vertex $k$. 
\begin{lem}\label{resolution of e_k}
Part of the projective (resp. injective) resolution of $E_k$ is given by:
\begin{align}
\label{pro}\cdots\xrightarrow{}\bigoplus_{ha=k}d_k P_{ta}\xrightarrow{ \sideset{_{b\varepsilon_k^m}}{_{\varepsilon^n_ka}}{\mathop{\left([n-m]_+\partial_{[b\varepsilon_k^{(n-m)}a]}\mathcal{S}\right)}}} \bigoplus_{tb=k}d_kP_{hb}\xrightarrow{ \left(b\varepsilon_k^{d_k-1},\cdots,b\varepsilon_k,b\right)_b}P_k\xrightarrow{}E_k\xrightarrow{} 0\\
     \label{inj}0\xrightarrow{}E_k\xrightarrow{}I_k\xrightarrow{\left(a,\varepsilon_k a,\cdots,\varepsilon_k^{d_k-1}a\right)_a^T}\bigoplus_{ha=k}d_k I_{ta}\xrightarrow{\sideset{_{b\varepsilon_k^m}}{_{\varepsilon^n_k a}}{\mathop{\left([n-m]_+\partial_{[b\varepsilon_k^{(n-m)}a]}\mathcal{S}\right)}} }\bigoplus_{tb=k}d_k I_{hb}\xrightarrow{}\cdots
\end{align}
\end{lem}
\begin{proof}
Let $f= (b,b\varepsilon_k,\cdots,b\varepsilon_k^{d_k-1})_b$ and $h= \sideset{ _{b\varepsilon_k^m}}{ _{\varepsilon^n_ka}}{\mathop{\left([n-m]_+\partial_{[b\varepsilon_k^{(n-m)}a]}\mathcal{S}\right)}}$.
 It is clear that $\Img f$ is generated by $b\varepsilon_k^l$ where $0\leq l\leq d_k-1$ and $tb=k$, so $\Coker f= E_k$. Also it follows directly from the definition of Jacobin algebra that $fh=0$ implying that $\Img h\subseteq \ker f$. To complete the proof that (\ref{pro}) is a projective resolution of $E_k$, it remains to show that $\ker f\subseteq \Img h$.
 
 Suppose that $ (p_{b,l}+J)^{T}_{b,l} \in\ker f$, where $p_{b,l}+J\in P_{hb}$, then $\sum_{1\leq l\leq d_k} p_{b,l} b\varepsilon_k^{d_k-l}=r$ for some element $r$ in the Jacobin ideal and $tr=k$. Clearly, we can assume that any component of $r$ does not contain $\varepsilon_k^{d_k}$. Suppose that $r$ contains a term of the form $p
 \partial_{c}\mathcal{S}q$ where $q$ is a path from $k$ to $hc$ for some arrow $c\in Q_1^\circ$.  If $q=q' b\varepsilon_k^{d_k-l}$ for some $1\leq l\leq d_k$ , $b \in Q_1^\circ $ with $tb =k$, and path $q'$ starting at $hb$, then $p \partial_c \mathcal{S} (q' b \varepsilon_k^{d_k-l}) \in J$ appears in $p_{b,l}$. Consequently, we may choose a representative of $(p_{b,l}+J)_{b,l}$  such that $r$ just contains terms of the form  $p
 (\partial_{a}\mathcal{S})\varepsilon_k^l$ where $ha=k$ and $0\leq l\leq d_k-1$. 
 
 Now we can write $r$ as the form 
 \begin{align*}
    r&=\sum_{\substack{a\in Q_1^\circ, ha=k,\\1\leq t\leq d_k}} q_{a,t}(\partial_{a}\mathcal{S})\varepsilon_k^{d_k-t}\\
    &=\sum_{l,t,a,b} q_{a,t}(\partial_{a}\mathcal{S})(b\varepsilon_k^{l})^{-1}(b\varepsilon_k^{l})\varepsilon_k^{d_k-t}\\
    &=\sum_{l,t,a,b} q_{a,t}(\partial_{[b\varepsilon_k^l a]}\mathcal{S})b\varepsilon_k^{d_k+l-t}\\
    &=\sum_{l,t,a,b} q_{a,t}([t-l]_+\partial_{[b\varepsilon_k^{t-l} a]}\mathcal{S})b\varepsilon_k^{d_k-l} 
 \end{align*}
 where $q_{a,t}$ is a $K$-linear combination of paths start at $ta$ and $(b\varepsilon_k^{l})^{-1}$ is the formal inverse. We conclude that $p_{b,l}=\sum q_{a,t} ([t-l]_+\partial_{[b\varepsilon_k^{t-l} a]}S)$ and so $(p_{b,l}+J)^T_{b,l}=h((q_{a,t}+J)^{T}_{a,t}) \in \Img h$.
\end{proof}

Suppose $M$ is a  representation of $(Q,\mathbf{d},\mathcal{S})$. Let
\begin{align}
    \cdots \xrightarrow{} P(\beta_2)\xrightarrow{d_1} P(\beta_1)\xrightarrow{d_0} P(\beta_0)\xrightarrow{}M\xrightarrow{} 0,\label{p}\\
    0\xrightarrow{} M\xrightarrow{} I(\check{\beta}_0)\xrightarrow{\Check{d_0}}I(\check{\beta}_1)\xrightarrow{\check{d_1}}I(\check{\beta}_2)\xrightarrow{}\cdots. \label{i}
\end{align}
be the minimal resolutions of $M$. Now we apply $\Hom_{\mathcal{P}}(-,M)$ and $\Hom_{\mathcal{P}}(M,-)$ to (\ref{pro}) and (\ref{inj}) respectively, then we get the following exact sequences:
\begin{align}
    0\xrightarrow{}\Hom_{\mathcal{P}}(E_k,M)\xrightarrow{}M_k\xrightarrow{\beta_k}M_{k,out}\xrightarrow{\gamma_k}M_{k,in}\xrightarrow{}\cdots,\label{apply hom(-,M)}\\
     0\xrightarrow{}\Hom_{\mathcal{P}}(M,E_k)\xrightarrow{}M_k^\star\xrightarrow{\alpha_k^\star}M_{k,in}^\star\xrightarrow{\gamma_k^\star}M_{k,out}^\star\xrightarrow{}\cdots.\label{apply hom(M,-)}
\end{align}
 We thus have that 
 \begin{align}
    \Hom_{\mathcal{P}}(E_k,M)\cong \ker \beta_k,\quad  \Hom_{\mathcal{P}}(M,E_k)\cong (\Coker \alpha_k)^\star,\label{(hom 1)}\\
    \Ext^1_{\mathcal{P}}(E_k,M)\cong\frac{\ker \gamma_k}{\Img \beta_k},\quad \Ext^1_{\mathcal{P}}(M,E_k)\cong(\frac{\ker \alpha_k}{\Img \gamma_k})^\star.\label{(ext 1)}
 \end{align}
Apply $\Hom_{\mathcal{P}}(-,E_k)$ and $\Hom_\mathcal{P} (E_k,-)$ to (\ref{p}) and (\ref{i}) respectively, we obtain
 the following sequences: 
 \begin{align}
     0\xrightarrow{}\Hom_{\mathcal{P}}(M,E_k)\xrightarrow{}\Hom_{\mathcal{P}}(P_0,E_k)\xrightarrow{d_0}\Hom_{\mathcal{P}}(P_1,E_k)\xrightarrow{d_1}\Hom_{\mathcal{P}}(P_2,E_k)\xrightarrow{}\cdots\\
     0\xrightarrow{}\Hom_{\mathcal{P}}(E_k,M)\xrightarrow{}\Hom_{\mathcal{P}}(E_k,I_0)\xrightarrow{\check{d}_0}\Hom_{\mathcal{P}}(E_k,I_1)\xrightarrow{\check{d}_1}\Hom_{\mathcal{P}}(E_k,I_2)\xrightarrow{}\cdots
 \end{align}
 From these sequences, we derive that
 \begin{align}
     \Hom_{\mathcal{P}}(M,E_k)\cong \ker d_0,\quad \Hom_{\mathcal{P}}(E_k,M)\cong \ker \check{d}_0,\label{(hom 2)}\\
     \Ext^1_{\mathcal{P}}(M,E_k)\cong \frac{\ker d_1}{\Img d_0},\quad \Ext^1_{\mathcal{P}}(E_k,M)\cong \frac{\ker \check{d_1}}{\Img \check{d_0}}.\label{(ext 2)}
 \end{align}
 Suppose that the representation $M$ is in general position, it is both $g$-coherent and $\check{g}$-coherent. This implies that $d_0$ and $\check{d}_0$ are zero maps. Combining \eqref{(hom 1)} and \eqref{(hom 2)}, we conclude that
 \begin{align*}
     \Coker \alpha_k\cong \Hom_{\mathcal{P}}(P_0,E_k)^\star\cong \beta_0(k)H_k,\quad 
     \ker \beta_k\cong \Hom_{\mathcal{P}}(E_k,I_0)\cong \check{\beta}_0(k) H_k.
 \end{align*}
 Hence $\ker \beta_k$ and $\Coker \alpha_k$ are free $H_k$-modules splitting in $M_k$ since $\ker \beta_k$  is an injective submodule and $\Coker \alpha_k$ is a quotient submodule of $M_k$.
\begin{rmk}
 Note that $H_k$ is a symmetric algebra, that is $H_k=H_k^\star$, and so self-injective. Then for any $H_k$-module N, we have that $N^\star\cong \Hom_{H_k}(N, H_k)$. In particular, $N\cong N^\star$ when $N$ is $H_k$-free.
\end{rmk}

Denote $\alpha_{k,0}=P(\beta_0)(\alpha_k)$ and $\alpha_{k,1}=P(\beta_1)(\alpha_k)$. Then we have the following commutative diagram of $H_k$-modules with exact rows:
$$\begin{tikzcd}
 P(\beta_1)_{k,\In}\arrow[r, "d_0"] \arrow[d,"\alpha_{k,1}"] &  P(\beta_0)_{k,\In} \arrow[r,two heads, "\pi"] \arrow[d, "\alpha_{k,0}"] &  M_{k,\In}\arrow[d, "\alpha_k"]  \\
P(\beta_1)_{k} \arrow[r, "d_0"]       & P(\beta_0)_{k}\arrow[r,two heads,"\pi"]       & M_{k}            
\end{tikzcd}$$

Suppose that $x\in \ker\alpha_k\subseteq M_{k,\In}$, then there exists $y\in P(\beta_0)_{k,\In}$ such that $x=\pi (y)$.  Since $\alpha_k\pi(y)=\pi\alpha_{k,0}(y)=0$, we have that $\alpha_{k,0}(y)\in \Img d_0$. Define an $H_k$-module morphism $\delta:\ker\alpha_{k}\xrightarrow{} d_0\left(\Coker \alpha_{k,1}\right)$ by $ x\mapsto \alpha_{k,0}(y)+\Img d_0 \alpha_{k,1}$. The map $\delta$ is well-defined since $d_0$ induces an $H_k$-module morphism between $\Coker \alpha_{k,1}$ and $\Coker\alpha_{k,0}$.
It is easy to check that the following sequence is exact:
\begin{align*}  \ker\alpha_{k,0}\xrightarrow{\pi}\ker\alpha_{k}\xrightarrow{\delta}d_0\left(\Coker \alpha_{k,1}\right)\xrightarrow{}0.
\end{align*}
For the projective representation $P(\beta_1)$, we have that $\Coker \alpha_{k,1}=\beta_1(k)H_k$. Thus, $d_0\left(\Coker \alpha_{k,1}\right)$ is free since $d_0$ is in general position and $(Q,\mathbf{d},\s)$ is locally free. Denote $\gamma_{k,0}=P(\beta_0)(\gamma_k)$ and $\gamma_{k,1}=P(\beta_0)(\gamma_k)$. Then we have the exact sequence
\begin{align*}
    \Img \gamma_{k,0}\xrightarrow{\pi}\Img\gamma_k\xrightarrow{}0.
\end{align*}
By \eqref{apply hom(M,-)}, $\Img \gamma_{k,0}=\ker\alpha_{k,0}$. Therefore, $\frac{\ker\alpha_k}{\Img \pi}=\frac{\ker\alpha_k}{\Img \gamma_k}\cong d_0\Coker \alpha_{k,1}$ is free. Hence, the exact sequences 
$$\begin{aligned}
    \frac{\ker \gamma_k}{\Img \beta_k}\hookrightarrow \frac{M_{k,\out}}{\Img\beta_k}\twoheadrightarrow \Img\gamma_{k},\quad
    \Img \gamma_k \hookrightarrow \ker\alpha_k \twoheadrightarrow \frac{\ker \alpha_k}{\Img\gamma_k}
\end{aligned}$$ split as $H
_k$-modules. Moreover, $\Ext^1_{\mathcal{P}}(M,E_k)$, $\Ext^1_{\mathcal{P}}(E_k,M)$ are free according to \eqref{(ext 1)}.

\begin{lem}
    Let $M\in \rep(\mathcal{P})$ be of projective (resp. injective) weight $\beta_--\beta_+$  (resp. $\check{\beta}_--\check{\beta}_+$). Suppose that $\Hom_{\mathcal{P}}(M,E_k), \Hom_{\mathcal{P}}(E_k,M)$, and $\Ext^1_{\mathcal{P}}(M,E_k), \Ext^1_{\mathcal{P}}(E_k,M)$ are free. Then: 
    \begin{align*}
   \beta_+(k)=\rank (\Hom_{\mathcal{P}}(M,E_k)),\quad \beta_-(k)=\rank (\Ext_{\mathcal{P}}(M,E_k)),\\
   \check{\beta}_+(k)=\rank (\Hom_{\mathcal{P}}(E_k,M)),\quad \check{\beta}_-(k)=\rank (\Ext_{\mathcal{P}}(E_k,M)).
\end{align*}
\end{lem}
\begin{proof}
 From \eqref{spm}, we have that 
$$\Hom_{\mathcal{P}}(S_k,M)=\ker f_{k;M}=\ker (M(b))_b\cap \ker M(\varepsilon_k)=\ker \beta_k\cap \ker M(\varepsilon_k).$$
Thus, in view of \eqref{(hom 1)}, $\hom_{\mathcal{P}}(S_k,M)=\rank (\Hom_{\mathcal{P}}(E_k,M))$.
Using the fact that $M\to I$ is the injective cover if and only if $\Hom_{\mathcal{P}}(S_k,M)\cong \Hom_{\mathcal{P}}(S_k,I)$, we obtain 
$\check{\beta}_+(k)=\rank (\Hom_{\mathcal{P}}(E_k,M))$. Moreover, we get that $M\to I$ is the injective cover if and only if $\Hom_{\mathcal{P}}(E_k,M)\cong \Hom_{\mathcal{P}}(E_k,I)$.

  Applying the functor $\Hom(E_k,-)$ to the exact sequence $0\to M\to I(\check{\beta}_+)\xrightarrow{\pi} C\to 0$, we get the following exact sequence:
  \begin{align*}
      0\xrightarrow{} \Hom_{\mathcal{P}}(E_k,M)\xrightarrow{}\Hom_{\mathcal{P}}(E_k, I(\check{\beta}_+))\xrightarrow{\pi}  \Hom_{\mathcal{P}}(E_k,C)\xrightarrow{} \Ext_{\mathcal{P}}^1(E_k,M) \xrightarrow{} \Ext_{\mathcal{P}}^1(E_k, I(\check{\beta}_+))=0
.  \end{align*}
By the above discussion, $\pi$ is the zero map, and so $\Hom_{\mathcal{P}}(E_k,C)\cong \Ext_{\mathcal{P}}^1(E_k,M)$ is free. Hence, $C\to I(\check{\beta}_-(k))$ is the injective cover if and only if $\rank(\Hom_{\mathcal{P}}(E_k,C))=\check{\beta}_-(k)$, as desired.

The proof of the first two equalities is similar.
\end{proof}

\begin{cor}\label{g-vector of M}
Suppose that $M$ is a general
representation of $(Q,\mathbf{d},\mathcal{S})$ of weights $\beta_--\beta_+$ and $\check{\beta}_--\check{\beta}_+$. For any $k\in Q_0$, we have 
\begin{equation}
    \begin{aligned}
      \Coker\alpha_k\cong \beta_+(k)H_k,\quad \frac{\ker\alpha_k}{\Img \gamma_k}\cong \beta_-(k)H_k,\\
      \ker \beta_k\cong \check{\beta}_+(k)H_k,\quad \frac{\ker \gamma}{\Img \beta_k}\cong \check{\beta}_-(k)H_k.
    \end{aligned}
\end{equation} 
\end{cor}

Next, considering the injective presentation of $M$ \eqref{i}, we denote $\beta_{k,i}=I(\check{\beta}_i)(\beta_k)$ and $\alpha_{k,i}=I(\check{\beta}_i)(\alpha_k)$ for $i=0,1$. Then we have the following commutative diagram of $H_k$-modules with exact rows:
$$\begin{tikzcd}
0\arrow[r]&\ker\alpha_k\arrow[d,hookrightarrow]\arrow[r,hookrightarrow] & \ker\alpha_{k,0}\arrow[r, "\check{d}_0"] \arrow[d,hookrightarrow] &  \ker\alpha_{k,1}  \arrow[d, hookrightarrow]   \\
0\arrow[r]&\ker \beta_k\alpha_k\arrow[r,hookrightarrow]\arrow[d,twoheadrightarrow,"\alpha_k"]       & \ker\beta_{k,0}\alpha_{k,0}\arrow[r,"\check{d}_0"]     \arrow[d,twoheadrightarrow,"\alpha_{k,0}"]  &     \ker\beta_{k,1}\alpha_{k,1} \\
0\arrow[r] & \ker\beta_k\cap\Img\alpha_k\arrow[r,hookrightarrow] & \ker\beta_{k,0}\cap\Img\alpha_{k,0}  & \ 
\end{tikzcd}$$
For the injective representation $I(\check{\beta}_0)$, we have that $\ker\beta_{k,0}\cap\Img\alpha_{k,0}$ is either  0 or $\check{\beta}_0(k)H_k^\star$. Then it is easy to check that ${\ker\beta_k\cap\Img\alpha_k}$ equals  
$$\alpha_k\left(\left(\bigoplus_{i\in \In(k)} {_{k}A_{i}}\otimes \check{\beta}_0(k)L_i\right)\cap \ker \check{d}_0 \right),$$
where $L_i=\operatorname{Span}_k\{(\varepsilon_k^la)^\star\mid a\in Q_0^\circ, ha=k,ta=i,0\leq l\leq d_k-1\}.$
Thus, ${\ker\beta_k\cap\Img\alpha_k}$ is free when $\check{d}_0$ is in general position, which implies that $\ker \beta_k/(\ker\beta_k\cap \Img\alpha_k)$ is free.

We are now ready to construct the \textbf{mutation} of a general decorated representation  $\mathcal{M}=(M,V)$ at $k$.

Let $(H,\widetilde{A},\widetilde{\s})=\widetilde{\mu}(H,A,\s)$ be given by (\ref{mutation of A}) and (\ref{mutation of s}). We associate to an $H$-based QP-representation $\mathcal{M}=(M,V)$  the $H$-based QP-representation $\widetilde{\mu}_k(\mathcal{M})=(H,\widetilde{A},\widetilde{\s},\overline{M},\overline{V})$ as follows. We set 
\begin{align}\label{mutation of Mk 1}
    \overline{M}_i=M_i,\quad \overline{V}_i=V_i
\end{align}
for $i\neq k$.
We define $\overline{M}_k$ and $\overline{V}_k$ by 
\begin{align}\label{mutation of Mk}
   \overline{M}_k= \frac{\ker \gamma_k}{\Img \beta_k}\oplus \Img\gamma_k\oplus \frac{\ker \alpha_k}{\Img\gamma_k}\oplus V_k,\quad \overline{V}_k=\frac{\ker\beta}{\ker\beta_k \cap \Img \alpha_k}.
\end{align}
Note that $\overline{M}_k$ is defined as the direct sum of $H_k$-modules. The construction of the action on $\overline{M}$ of all arrows in $\widetilde{A}$ follows the same approach as in \cite[Section 10]{derksen2010quivers}.  

First, for every arrow $c$ not incident to $k$, $\overline{M}(c)=M(c)$, and $\overline{M}([b\varepsilon_k^la])=M([b\varepsilon_k^l a])$ for all $a\in e_kA\cap Q^\circ_1$, $b\in Ae_k\cap Q^\circ_1$ and $0\leq l\leq d_k-1$.

To define the action of remaining arrow $a^\star$ and $b^\star$, we assemble them into the $H_k$-linear maps
$$\overline{\alpha}_k:M_{k,\out}=\overline{M}_{k,\In}\to \overline{M}_k,\quad \overline{\beta}_k:\overline{M}_{k}\to \overline{M}_{k,\out}=M_{k,\In}$$
in the same way as in (\ref{defn of alpha}) and (\ref{defn of beta}).
We will use the following notational convention: whenever we have a pair $U_{1} \subseteq U_{2}$ of $H_k$-modules, denote by $\iota: U_{1} \rightarrow U_{2}$ the inclusion map, and by $\pi: U_{2} \rightarrow U_{2} / U_{1}$ the natural projection. We now introduce the following \textbf{splitting data}:
\begin{align}\label{splitting data rho}
    \text{Choose an $H_k$-linear map $\rho: M_{k,\out} \rightarrow \ker \gamma_k$ such that $\rho \iota=\mathrm{id}_{\ker \gamma_k}$.}
\end{align}
\begin{align}\label{soplitting data sigma}
    \text{Choose an $H_k$-linear map $\sigma:\ker\alpha_k / \Img \gamma_k \rightarrow \ker\alpha_k$ such that $\pi \sigma=\mathrm{id}_{\ker \alpha_k / \Img\gamma_k}$.}
\end{align}
Then we define:
\setlength{\arraycolsep}{1.5pt}
\begin{align}\label{defn of mutation of alpha and beta}
    \overline{\alpha}_k=\left(\begin{array}{c}
-\pi \rho \\
-\gamma_k \\
0 \\
0
\end{array}\right), \quad \overline{\beta}_k=\left(\begin{array}{llll}
0, & \iota, & \iota \sigma, & 0
\end{array}\right).
\end{align}

\begin{pro}
     The above definitions make $\widetilde{\mu}_k(\mathcal{M})=(\overline{M},\overline{V})$ a decorated representation of $(\widetilde{Q},\mathbf{d},\widetilde{\mathcal{S}})$.
\end{pro}
\begin{proof}
    We only need to show that $M(\partial_{c}\widetilde{\mathcal{S}})=0$  for all $c\in \widetilde{Q}^\circ$. The case $c$ is not incident to $k$ is clear.
    
    If $c$ is one of the arrows $[b\varepsilon^l_k a]$, then 
    \begin{align}\label{parital [ba]}
        \partial_c\widetilde{\mathcal{S}} = \partial_{[b\varepsilon^l_k a]}\mathcal{S}+ a^\star(\varepsilon_k)^{d_k-l-1} b^\star=0.
    \end{align}
    In view of  the definitions of $\alpha_k$,$\beta_k$, and $\gamma_k$ in \eqref{defn of alpha}, \eqref{defn of beta}, and \eqref{defn of gamma}, respectively, it is enough to show $\overline{\beta}_k\overline{\alpha}_k=-\gamma_k$. But this follows at once by multiplying the row and the column given by (\ref{defn of mutation of alpha and beta}).
    
    If $c$ is one of the arrows $a^\star$, then $$\partial_c\widetilde{\mathcal{S}}=\sum_{\substack{0\leq l\leq d_k-1\\b\in Q^\circ \cap Ae_k}}(\varepsilon_k)^{d_k-l-1} b^\star b\varepsilon^l_k a .$$
    It suffices to show that $$M(\sum_{\substack{0\leq l\leq d_k-1\\b\in Q^\circ \cap Ae_k}}(\varepsilon^\star_k)^{d_k-l-1} b^\star b\varepsilon^l_k)=0$$, or equivalent, $\overline{\alpha}_k\beta_k=0.$
    
    If $c$ is one of the arrows $b^\star$, then $$\partial_c\widetilde{\mathcal{S}}=\sum_{0\leq l\leq d_k-1}b\varepsilon^l_k a a^\star(\varepsilon_k)^{d_k-l-1} .$$
    It suffices to show that $ \alpha_k\overline{\beta}_k=0$.
    In view of (\ref{defn of mutation of alpha and beta}), we have 
    \begin{align*}
        \setlength{\arraycolsep}{1.5pt}
        \overline{\alpha}_k\beta_k=\left(\begin{array}{c}
            -\pi \rho \beta_k  \\
             -\gamma_k\beta_k\\
             0\\
             0
        \end{array}\right)=0, \quad
        \alpha_k\overline{\beta}_k=\left(\begin{array}{llll}
            0, & \alpha_k\iota,&\alpha_k\iota\sigma,&0  
        \end{array}\right)=0,
    \end{align*}
    as desired.
\end{proof}
\begin{pro}\label{iso of splitting data}
    The isomorphism class of the decorated representation $\widetilde{\mu}_{k}(\mathcal{M})$ does not depend on the choice of the splitting data $\left(\ref{splitting data rho}\right)$-$\left(\ref{soplitting data sigma}\right)$.
\end{pro}
\begin{proof}
    The proof follows the same argument as in\cite[Proposition 10.9]{derksen2008quivers}, but with the linear maps replaced by $H_k$-linear maps.
\end{proof}
\begin{pro}\label{mutation of right-equivalence class of the representation}
    The right-equivalence class of the representation $\widetilde{\mu}_{k}(\mathcal{M})$ is determined by the right-equivalence class of $\mathcal{M}$.
\end{pro}
The proof appears in Appendix \ref{app: mutation of right-equivalence class of the representation}.

Note that in the above treatment of the operation $\mathcal{M} \mapsto \widetilde{\mu}_{k}(\mathcal{M})$ for an $H$-based QP-representation $\mathcal{M}=(H,A,\s, M, V)$, the $H$-based QP $(H,A,\s)$ was not assumed to be reduced. Recall from Proposition \ref{unique of he right-equivalence class of reduced rep} that we have a well-defined operation $\mathcal{M} \mapsto \mathcal{M}_{\red}$ on (right-equivalence classes of) $H$-based QP-representations. The following property is immediate from definitions.

\begin{pro}
     Let $(H,A,\s)$ be an locally free $H$-based QP satisfying $(\ref{condition on mutation of GQP 1})$. Then, for every 
 general decorated representation $\mathcal{M}$ of $(H,A,\s)$, the representation $\widetilde{\mu}_{k}(\mathcal{M})_{\red}$ is right-equivalent to $\widetilde{\mu}_{k}\left(\mathcal{M}_{\red}\right)_{\red}$.
\end{pro}

Recall that, according to Corollary \ref{reduced mutation uniquele right-equivalence class of the GQP} and Definition \ref{defn of mutation of gqp}, the correspondence $\widetilde{\mu}_{k}:(H, A,\s) \mapsto (H,\widetilde{A}, \widetilde{\s})$ gives rise to the mutation $(H, A,\s) \mapsto \mu_{k}(H, A,\s)=(H,\overline{A}, \overline{S})$, which is a well-defined bijective transformation on the set of right-equivalence classes of reduced QPs satisfying (\ref{condition on mutation of GQP 1}). Here $(H,\overline{A},\overline{S})$ is the reduced part of $(H,\widetilde{A}, \widetilde{\s})$. Now for every general  $H$-based QP-representation $\mathcal{M}=(H, A, \s, M, V)$ of a reduced $H$-based QP $(H,A,\s)$ we define
\begin{align}\label{defn of mutation of M}
    \mu_{k}(\mathcal{M})=\widetilde{\mu}_{k}(\mathcal{M})_{\red};
\end{align}
thus, $\mu_{k}(\mathcal{M})$ is a decorated representation $(H,\overline{A}, \overline{\s}, \overline{M}, \overline{V})$ of a reduced $H$-based QP $(H,\overline{A}, \overline{\s})$. Combining Propositions \ref{unique of he right-equivalence class of reduced rep} and \ref{mutation of right-equivalence class of the representation}, we obtain the following important corollary.
\begin{cor}\label{defn on mutation of M}
    Suppose that $\mathcal{M}$ is a general decorated representation, then the correspondence $\mathcal{M} \mapsto \mu_{k}(\mathcal{M})$ is a well-defined transformation on the set of right-equivalence classes of decorated representations of reduced and locally free $H$-based QPs satisfying $(\ref{condition on mutation of GQP 1})$.
\end{cor}
We refer to the transformation $\mathcal{M} \mapsto \mu_{k}(\mathcal{M})$ in Corollary \ref{defn on mutation of M} as the \textbf{mutation at vertex $k$}. With some abuse of terminology, we will talk about mutations of general decorated representations (rather than their right-equivalence classes).

The following result naturally extends Theorem \ref{mutation of gqp is involution }.
\begin{thm}\label{mutation of rep is involution}
      The mutation $\widetilde{\mu}_k$ of general decorated representations is an involution.
\end{thm}
\begin{proof}
The proof parallels the argument in  \cite[Theorem 10.13]{derksen2010quivers}, using $H_k$-linear maps instead of linear maps.
\end{proof}

Note that there is an obvious way to define direct sums of decorated representations of a given $H$-based QP. Hence, we can talk about \textbf{indecomposable}  $H$-based QP-representations.  Clearly, the right-equivalence relation respects direct sums and indecomposability. It is also immediate from the definitions that any mutation $\mu_{k}$ of $H$-based QP-representations sends direct sums to direct sums. Combining this with Theorem \ref{mutation of rep is involution}, we obtain the following corollary.

\begin{cor}
Any mutation $\mu_{k}$ is an involution on the set of right-equivalence classes of indecomposable decorated representations of reduced and locally free $H$-based QPs satisfying $(\ref{condition on mutation of GQP 1})$.
\end{cor}
\begin{rmk}
     Recall that $\overline{\alpha}_k$ and $\overline{\beta}_k$ are given by $(\ref{defn of mutation of alpha and beta})$ for a general decorated representation $\mathcal{M}=(M,V)$. As for $\overline{\gamma}_k$, by applying the definition $(\ref{defn of gamma})$ to the potential $\widetilde{\s}$ given by $(\ref{potiential})$, we see that 
\begin{align}\label{bar gamma=beta alpha}
    \overline{\gamma}_k=\beta_k\alpha_k.
\end{align}
As a direct consequence of the definition, we conclude that 
\begin{equation}\label{aftermutation}
    \begin{aligned}
    \setlength{\arraycolsep}{8pt}
    \begin{array}{ll}
        \ker \overline{\alpha}_k=\Img \beta_k,& \Img \overline{\alpha}_k= \frac{\ker \gamma_k}{\Img \beta_k}\oplus \Img\gamma_k\oplus \{0\}\oplus \{0\},\\
    \ker\overline{\beta}_k=\frac{\ker \gamma_k}{\Img \beta_k}\oplus \{0\}\oplus \{0\} \oplus V_k,& \Img\overline{\beta}_k=\ker\alpha,\\
\ker\overline{\gamma}_k=\ker(\beta_k\alpha_k), & \Img\overline{\gamma}_k=\Img(\beta_k\alpha_k).
    \end{array} 
\end{aligned}
\end{equation}
\end{rmk}

\section{Mutation of Presentations}\label{section: mutation of presentatoins}
In \cite{fei2024crystalstructureuppercluster}, the mutation of presentations for classical quivers with potentials was introduced, where it was shown to coincide with the mutation of decorated representations. In this section, we generalize this construction to the setting of $H$-based quivers with potentials.

Let $(Q,\mathbf{d},\s)$ be a locally free $H$-based quiver with potential, and $\mathcal{V}$ be a decorated representation of $(Q,\mathbf{d},\s)$. By the \textbf{extension} of $(Q,\mathbf{d},\s)$ by $\mathcal{V}$, we mean the following construction. 

We start with $(Q,\mathbf{d},\mathcal{S})$ and a new vertex $v$. Take the projective presentation $d_{\mathcal{V}}: P\left(\beta_{-}\right) \rightarrow$ $P\left(\beta_{+}\right)$ corresponding to $\mathcal{V}$. We assume that $\mathcal{V}$ is in general position and so $P(\beta_-)$ and $P(\beta_+)$ share no common summands. Then we draw $\beta_{+}(j)$ arrows from $v$ to $j$ and $\beta_{-}(i)$ arrows from $i$ to $v$ and set $d_v=1$. We view the map $d_{\mathcal{V}}$ as a matrix with entries a linear combination of paths. For each entry of $c:=d_{\mathcal{V}}(i, j): P_{i} \rightarrow P_{j}$, we add the potential $acb$ to the original potential $\mathcal{S}$ where $a$ is the added arrow corresponding to $P_{i}$ and $b$ is the added arrow corresponding to $P_{j}$. We denote the resulting  $H$-based quiver with potential by $(Q[\mathcal{V}], \mathbf{d}[\mathcal{V}],\mathcal{S}[\mathcal{V}])$ or in short $(Q,\mathbf{d},\mathcal{S})[\mathcal{V}]$ or $(Q,\mathbf{d},\mathcal{S})\left[d_{\mathcal{V}}\right]$, and abbreviate its Jacobian algebra to $\mathcal{P}[\mathcal{V}]$. If we restrict $(Q,\mathbf{d}, \mathcal{S})[\mathcal{V}]$ on $Q$ in the sense of \cite{derksen2008quivers}, then we get the original $H$-based QP $(Q,\mathbf{d}, \mathcal{S})$ back.

There is an obvious dual construction $(Q,\mathbf{d}, \mathcal{S})\left[\mathcal{V}^{*}\right]$ from the injective presentation $\check{d}_{\mathcal{V}}: I_{+} \rightarrow I_{-}$ corresponding to $\mathcal{V}$. It is easy to see that $(Q,\mathbf{d}, \mathcal{S})[\mathcal{V}]=(Q,\mathbf{d}, \mathcal{S})\left[\tau \mathcal{V}^{*}\right]$.

 Conversely, given a vertex $v \in Q_{0}$ and an $H$-based quiver with potentials $(Q,\mathbf{d}, \mathcal{S})$ satisfying \eqref{condition on mutation of GQP 1} at $v$, let $(Q,\mathbf{d}, \mathcal{S})_{\hat{v}}$ be the restriction of $(Q,\mathbf{d}, \mathcal{S})$ to the full subquiver of $Q_{0} \backslash\{v\}$.
\begin{defn}
    We call a vertex $v$ \textbf{simple} in $(Q,\mathbf{d}, \mathcal{S})$ if $d_v=1$ and for each pair of arrows $a:u\to v$ and 
$b:v\to w$, the partial derivative $\partial_{[ba]}[\s]$ contains no arrows incident to $v$. 
\end{defn} 
 Note that the newly constructed vertex $v$ in the extended quiver $(Q,\mathbf{d}, \mathcal{S})[\mathcal{V}]$ is simple. Conversely, if $v$ is simple in $(Q,\mathbf{d}, \s)$, then we can obtain a presentation $d_{v}$ of $(Q,\mathbf{d}, \mathcal{S})_{\hat{v}}$ as follows:
 \begin{align}
     d_v:\bigoplus_{a:i\to v\in Q_1^\circ}P_i\xrightarrow{\partial_{[ba]}[\s]} \bigoplus_{b:v\to j\in Q_1^\circ}P_j.
 \end{align}
Clearly we have that $(Q,\mathbf{d}, \mathcal{S})_{\hat{v}}[d_v]=(Q,\mathbf{d}, \mathcal{S})$. We also define $\check{d}_v=\nu(d_v)$.

To make sense of  Definition \ref{defn of mutation of pre}, we observe that the restriction of $\mu_{k}((Q,\mathbf{d}, \mathcal{S})[\mathcal{V}])$ to $\mu_{k}(Q)$ is $\mu_{k}(Q,\mathbf{d}, \mathcal{S})$, and the following lemma holds.
\begin{lem}[{\cite[Lemma 2.12]{fei2024crystalstructureuppercluster}}]
    For any mutation $\mu_k$ away from $v$, if $v$ is simple in $(Q,\mathbf{d},\s)$ (satisfying \eqref{condition on mutation of GQP 1}), then $v$ is simple in $\mu_k(Q,\mathbf{d},\s)$ as well.
\end{lem}
\begin{defn}\label{defn of mutation of pre}
 Given a presentation $d_{\mathcal{V}}$ of a locally free $H$-based QP $(Q,\mathbf{d}, \mathcal{S})$ that satisfies \eqref{condition on mutation of GQP 1}, we define $\mu_{k}(d_{\mathcal{V}})$ at vertex $k$ as the presentation $d_{v}$ of $\mu_{k}(Q,\mathbf{d}, \mathcal{S})$ obtained from $\mu_{k}((Q,\mathbf{d}, \mathcal{S})[\mathcal{V}])$ via the above construction.
\end{defn}
\begin{lem}\label{mutation of pre and rep}
    The mutation of general presentations is compatible with the mutation of representation, that is,
    \begin{align*}
    \mu_k(d_{\mathcal{V}})=d_{\mu_k(\mathcal{V})},\quad \mu_k(\check{d}_{\mathcal{V}})=\check{d}_{\mu(\mathcal{V})}.
    \end{align*} 
\end{lem}
\begin{proof}
    Let $\widetilde{\mu}_k(d_{\mathcal{V}})$ be the presentation of $(\widetilde{Q},\mathbf{d},\widetilde{\s})=\widetilde{\mu}_k(Q,\mathbf{d}, \mathcal{S})$ obtained from $(\widetilde{Q[\mathcal{V}]},\mathbf{d}, \widetilde{\s[\mathcal{V}]})=\widetilde{\mu}_k((Q,\mathbf{d}, \mathcal{S})[\mathcal{V}])$ via the above construction, that is, 
    \begin{align*}
\widetilde{\mu}_k(d_{\mathcal{V}}):\bigoplus_{\substack{a:i\to v\in \widetilde{Q[\mathcal{V}]}^\circ_1}}\widetilde{P}_i\xrightarrow{\partial_{[ba]}(\widetilde{\mu}_k(\s[\mathcal{V}]))} \bigoplus_{\substack{b:v\to j\in \widetilde{Q[\mathcal{V}]}^\circ_1}}\widetilde{P}_j,
    \end{align*}
    where $\widetilde{P}_i$ is the projective representation in $\widetilde{\mu}_k(Q,\mathbf{d}, \mathcal{S})$ corresponding to the vertex $i$. Let $\phi:\overline{T_H(\widetilde{A})}\to \overline{T_H(\overline{A}\oplus \widetilde{A}_{\mathrm{triv}}})$ be the right-equivalence of $H$-based QPs $(H,\widetilde{A},\widetilde{\s})$ and $(H,\overline{A},\overline{\s})\oplus (H,\widetilde{A}_{\mathrm{triv}},\widetilde{\s}^{\circ(2)})$. Then, the following commutative diagram follows from  Proposition \ref{trival part of jacobian}:
$$\begin{tikzcd}[column sep=huge]
 \bigoplus_{a:i\to v\in \widetilde{Q[\mathcal{V}]}^\circ_1}\widetilde{P}_i\arrow[r, "{\partial_{[ba]}(\widetilde{\mu}_k(\s[\mathcal{V}]))}"] \arrow[d,"\simeq" left,"\phi"] &  \bigoplus_{\substack{b:v\to j\in \widetilde{Q[\mathcal{V}]}^\circ_1}} \widetilde{P}_j  \arrow[d,"\simeq" left, "\phi"]   \\
\bigoplus_{\substack{a:i\to v\in \widetilde{Q[\mathcal{V}]}^\circ_1}}\overline{P}_i\arrow[r, "{\overline{d}}"]  &  \bigoplus_{\substack{b:v\to j\in \widetilde{Q[\mathcal{V}]}^\circ_1}}\overline{P}_j
\end{tikzcd}
$$
 where $\overline{d}=(\overline{d}_{b,a})=\Big(\phi\big(\partial_{[ba]}(\widetilde{\mu}_k(\s[\mathcal{V}]))\big)\Big)$. By definition, we have that $$\widetilde{\mu}_k(\s[\mathcal{V}])=[\s]+\sum_{\substack{a':i\to v\in Q[\mathcal{V}]^\circ_1\\b':v\to j\in Q[\mathcal{V}]_1^\circ}}[c_{b',a'}b'a']+\sum_{\substack{h:i\to k\in Q[\mathcal{V}]_1^\circ\\f:k\to j\in Q[\mathcal{V}]_1^\circ}}h^\star\varepsilon_k^{d_k-l-1}f^\star[f\varepsilon_k^lh].$$ Without loss of generality, we may assume that $\beta_-(k)=0$. Then,
 $$\overline{d}_{b,a}=\begin{cases}
     \phi(c_{b,a}) \quad &\text{if $a,b\in Q[\mathcal{V}]^\circ_1$ are not incident to $k$,}\\
     \phi\big(c_{b',a}(f\varepsilon_k^l)^{-1}\big) \quad& \text{if $b=[f\varepsilon_k^lb']$ for some $f:k\to j\in Q_1^\circ$ and $a\in Q[\mathcal{V}]_1^\circ$},\\
     \varepsilon_k^{d_k-1-l}f^\star \quad& \text{if $b=[f\varepsilon_k^lb']$ for some $f:k\to j\in Q_1^\circ$ and $a=(b')^\star$},\\
     0\quad & \text{otherwise}.
 \end{cases}$$
Recalling the reduction process, we observe that $\overline{d}$ is homotopy equivalent to $\mu_k(d_{\mathcal{V}})$. In view of Definition \ref{reduce of M} and \eqref{defn of mutation of M},it suffices to show that $\widetilde{\mu}_k(d_{\mathcal{V}})=d_{\widetilde{\mu}_k(\mathcal{V})}$.  We proceed in several steps.

    \textbf{Step 1} We start with the mutation of projective representations. 
    
    (1) We first consider the direct sum of all indecomposable projective modules except $P_k$. Let $P_{\hat{k}}=\oplus_{i\neq k}P_i$, then $\widetilde{\mu}_k(d_{P_{\hat{k}}})$ is $0\xrightarrow{}\widetilde{P}_{\hat{k}}$. Thus, we need to show that $\widetilde{\mu}_k(P_{\hat{k}})$ is right-equivalent to $\widetilde{P}_{\hat{k}}$. It suffices to construct an $H$-module isomorphism $\psi=(\psi_i):\widetilde{\mu}_k(P_{\hat{k}})\to \widetilde{P}_{\hat{k}}$ such that 
     \begin{align}\label{construct of iso}
        \text{$\psi \circ \widetilde{\mu}_k(P_{\hat{k}})(u)=\widetilde{P}_{\hat{k}}(u)\circ \psi$ for all $u\in \widetilde{Q}^\circ_1$}.
    \end{align}
   since they have no negative parts. In view of Proposition \ref{jacobian algebra mutation invariants}, there is an algebra isomorphism $\varphi:u\to [u]$ between $\mathcal{P}_{\hat{k},\hat{k}}$ and $\widetilde{\mathcal{P}}_{\hat{k},\hat{k}}$. So $\varphi$ induces an isomorphism of $H$-module $\widetilde{\mu}_k(P_{\hat{k}})_{\hat{k}}\to \widetilde{P}_{\hat{k},\hat{k}}$. As for the $H_k$-module $\widetilde{\mu}_k(P_{\hat{k}})_{k}$, it is given as in (\ref{mutation of Mk}), with the maps $\alpha_k,\beta_k$ and $\gamma_k$ replaced by $\overline{\alpha}_k,\overline{\beta}_k$ and $\overline{\gamma}_k$, respectively. Recalling the identifications of $\Coker \alpha_k$ and $\frac{\ker \alpha_k}{\Img\gamma_k}$ in Corollary \ref{g-vector of M}, we have that $\Coker \alpha_k=\frac{\ker \alpha_k}{\Img\gamma_k}=0$. It follows from \eqref{aftermutation} that
    \begin{align}\label{mutation of P_hat k}
    \widetilde{\mu}_k(P_{\hat{k}})_{k}=\frac{\ker \gamma_k}{\Img \beta_k\alpha_k}\oplus\Img \gamma_k \oplus\{0\} \oplus \{0\}=\Img \overline{\alpha}_k.
    \end{align} 
     For the maps $\widetilde{\alpha}_k,\widetilde{\beta}_k$ and $\widetilde{\gamma}_k$ of  $\widetilde{P}_{\hat{k}}$, we have $\Coker \widetilde{\alpha}_k=0$ and $\ker\widetilde{\alpha}_k=\Img \widetilde{\gamma}_k$. This yields the following diagram with exact rows, where the leftmost block commutes and all vertical maps except the dashed one are isomorphisms. 
     $$\begin{tikzcd}
 \widetilde{\mu}_k(P_{\hat{k}})_{k,\In}\arrow[r, "\overline{\gamma}_k"] \arrow[d,"\varphi"] &  \widetilde{\mu}_k(P_{\hat{k}})_{k,out} \arrow[r,two heads, "\overline{\alpha}_k"] \arrow[d, "\varphi"] &  \widetilde{\mu}_k(P_{\hat{k}})_k\arrow[d, dashed]  \\
(\widetilde{P}_{\hat{k}})_{k,\In} \arrow[r, "\widetilde{\gamma}_k"]                        & (\widetilde{P}_{\hat{k}})_{k,out}\arrow[r,two heads,"\widetilde{\alpha}_k"]                            & (\widetilde{P}_{\hat{k}})_{k}                              
\end{tikzcd}$$
Hence $\varphi$ induces an $H_k$-isomorphism $\psi_k:\widetilde{\mu}_k(P_{\hat{k}})_k\to (\widetilde{P}_{\hat{k}})_{k}$ such that $\widetilde{\alpha}_k\circ\varphi=\psi_k\circ \overline{\alpha}_k$.

 Now we get an isomorphism of $H$-modules $\psi=(\psi_i): \widetilde{\mu}_k(P_{\hat{k}})\to \widetilde{P}_{\hat{k}}$,  where $\psi_i=\varphi$ for $i\neq k$ and $\psi_k$ is constructed as above. It remains to check that $\varphi\circ\widetilde{\beta}_k=\overline{\alpha}_k\circ \psi_k$. 

 According to \eqref{parital [ba]} and \eqref{bar gamma=beta alpha}, we see that
 \begin{align*}
     \widetilde{\gamma}_k=-\varphi(\beta_k\alpha_k)=\varphi(\overline{\gamma}_k),\quad \widetilde{\beta}_k\widetilde{\alpha}_k=-\varphi(\gamma_k)=\varphi(\overline{\beta}_k\overline{\alpha}_k). 
 \end{align*}
Thus, $\varphi\circ\overline{\beta}_k\overline{\alpha}_k=\widetilde{\beta}_k\widetilde{\alpha}_k\circ\varphi=\widetilde{\beta}_k\circ\psi_k\circ\overline{\alpha}_k,$ which implies $\varphi\circ\widetilde{\beta}_k=\overline{\alpha}_k\circ \psi_k$ since $\overline{\alpha}_k$ is surjective.


  (2) For the projective representation $P_k$, we see that $\widetilde{\mu}_k(d_{P_k})$ equals to 
   \begin{align*}
 \widetilde{P}_k\xrightarrow{f_{\widetilde{\alpha}_k}}\bigoplus_{b:k\to j\in Q^\circ_1}d_k\widetilde{P}_j,
   \end{align*}
   where $f_{\widetilde{\alpha}_k}=(b^*,\cdots,\varepsilon_k^{d_k-1}b^*)_b^{T}$.
Consider the representation $M$ of $(Q,\mathbf{d},\s)$ whose presentation is \begin{align*}
    P_k\xrightarrow{f_{\alpha_k}}\bigoplus_{a:i\to k\in Q^\circ_1}d_kP_i
\end{align*}
where $f_{\alpha_k}=(a,\cdots,\varepsilon_k^{d_k-1}a)_a^{T}$. Note that $M$ is in general position and so the mutation of $M$ is well-defined.
To show $d_{\widetilde{\mu}_k(P_k)}=\widetilde{\mu}_k(d_{P_k})$, it is equivalent to show that $\widetilde{M}:=\widetilde{\mu}_k(M)$ is right-equivalent to $\widetilde{P}_k$ by Proposition \ref{mutation of rep is involution}. We proceed by constructing an explicit $H$-module isomorphism $\psi=(\psi_i)$ between  $\widetilde{M}$ and $\widetilde{P}_k$ that satisfies the condition (\ref{construct of iso}). For a representation $N$, we use the notation  $N_{\hat{k}}=\oplus_{i\neq k} N_i$. 
 We first consider the map $$f_{\beta_k}=(b\varepsilon_k^{d_k-1},\cdots,b)_b^{T}:\bigoplus_{b:k\to j\in Q^\circ_1}d_k(P_j)_{\hat{k}}\to (P_k)_{\hat{k}}.$$  Taking the dual of $f_{\beta_k}$, we get an $H_k$-linear map $\beta_k: (I_{\hat{k}})_{k}\to (I_{\hat{k}})_{k,\out}$ of the injective representation $I_{\hat{k}}$.  Recalling the identification of $\ker{\beta_k}$, we have $\ker{\beta_k}=0$, and so $f_{\beta_k}$ is surjective. Consequently, we obtain 
 \begin{align}\label{eq in mutation of pre}
     M_{\hat{k}}=\frac{\bigoplus_{a:i\to k\in Q^\circ_1}d_k(P_i)_{\hat{k}}}{\Img f_{\alpha_k}|_{(P_k)_{\hat{k}}}}=\frac{\bigoplus_{a:i\to k\in Q^\circ_1}d_k(P_i)_{\hat{k}}}{\Img f_{\alpha_k} f_{\beta_k}|_{\oplus d_k(P_j)_{\hat{k}}}}=\frac{\bigoplus_{a:i\to k\in Q^\circ_1}d_k(P_i)_{\hat{k}}}{\Img f_{\beta_k\alpha_k}|_{\oplus d_k(P_j)_{\hat{k}}}}.
 \end{align}

Next we consider the following exact sequence
\begin{align}\label{consturct of iso 2.0}
    \bigoplus_{b:k\to j\in Q^\circ_1}d_k(\widetilde{P}_j)_{\hat{k}}\xrightarrow{f_{\widetilde{\gamma}_k}}\bigoplus_{a:i\to k\in Q^\circ_1}d_k(\widetilde{P}_i)_{\hat{k}}\xrightarrow{f_{\widetilde{\beta}_k}}(\widetilde{P}_k)_{\hat{k}}\xrightarrow{}0,
\end{align}
where $f_{\widetilde{\beta}_k}=(a^*\varepsilon_k^{d_k-1},\cdots,a^*)$ and $f_{\widetilde{\gamma}_k}=\sideset{_{a^*\varepsilon_k^m}}{_{\varepsilon^n_k b^*}}{\mathop{\left([n-m]_+\partial_{[a^*\varepsilon_k^{(n-m)}b^*]}\widetilde{\s}\right)}}$.
(Take the dual of (\ref{consturct of iso 2.0}), use the facts that $(P_i)_{j}^*=(I_j)_{ i}$, $\ker\beta_k=0$,  and $\ker\gamma_k=\Img \beta_k$ for the injective representation $I_{\hat{k}}$, we can see that the sequence  (\ref{consturct of iso 2.0}) is exactly  exact). The map $f_{\widetilde{\beta}_k}$ induces an isomorophism of $H_{\hat{k}}$-module between $\frac{\bigoplus_{a:i\to k\in Q^\circ_1}d_k(\widetilde{P}_i)_{\hat{k}}}{\Img f_{\widetilde{\gamma}_k}|_{\oplus d_k(P_j)_{\hat{k}}}}$ and  $(\widetilde{P}_k)_{\hat{k}}$.
As discussed before, the map $\varphi:u \mapsto [u]$ induces an isomorphism between $(P_{\hat{k}})_{\hat{k}}$ and $(\widetilde{P}_{\hat{k}})_{\hat{k}}$.  Combining this with the equality in (\ref{eq in mutation of pre}), we have the following isomorphism:
$$\frac{\bigoplus_{a:i\to k\in Q^\circ_1}d_k(\widetilde{P}_i)_{\hat{k}}}{\Img f_{\widetilde{\gamma}_k}|_{\oplus d_k(\widetilde{P}_j)_{\hat{k}}}}\stackrel{\varphi}{\cong} \frac{\bigoplus_{a:i\to k\in Q^\circ_1}d_k(P_i)_{\hat
k}}{\Img f_{\beta_k\alpha_k}|_{\oplus d_k(P_j)_{\hat{k}}}}=\frac{\bigoplus_{a:i\to k\in Q^\circ_1}d_k(P_i)_{\hat{k}}}{\Img f_{\alpha_k}|_{(P_k)_{\hat{k}}}}=M_{\hat{k}}.$$
Therefore, we conclude that the map 
$$\psi_{\hat{k}}=f_{\widetilde{\beta}_k}\varphi: \overline{M}_{\hat{k}}\to (\widetilde{P}_k)_{\hat{k}}$$
is an $H_{\hat{k}}$-module isomorphism. Moreover, the equation (\ref{construct of iso}) holds for $c\in \widetilde{Q}_{1}^\circ$ not incident to $k$. 

As for the $H_k$-module $\widetilde{M}_k$, it is given in \eqref{mutation of Mk}, with the maps $\alpha_k,\beta_k$ and $\gamma_k$ replaced by $\overline{\alpha}_k,\overline{\beta}_k$ and $\overline{\gamma}_k$ in \eqref{defn of mutation of alpha and beta}, respectively. According to (\ref{aftermutation}) and the identifications of $\Coker \alpha_k$,$\frac{\ker \alpha_k}{\Img\gamma_k}$ in Corollary \ref{g-vector of M}, we have the following exact sequence: 
\begin{align*}
\overline{M}_{k,\In}\xrightarrow{\overline{\gamma}_k}\overline{M}_{k,out}\xrightarrow{\overline{\alpha}_k}\overline{M}_k\xrightarrow{}\frac{\ker \alpha_k}{\Img \gamma_k}\xrightarrow{}0.
\end{align*}
 Here $\frac{\ker \alpha_k}{\Img \gamma_k}$ is a  free $H_k$-module of rank 1. There exists an $H_k$-isomorphism $\psi_k^1:\frac{\ker \alpha_k}{\Img \gamma_k}\to \widetilde{E}_k$ such that $\psi_{\hat{k}}\overline{\beta}_k|_{\frac{\ker \alpha_k}{\Img \gamma_k}}=\widetilde{\beta}_k\psi_k^1.$

Recall the injective resolution of $\widetilde{E}_k$ (see (\ref{inj})). Applying the functor $\Hom_{\widetilde{\mathcal{P}}}(\widetilde{P}_k,-)^*$  to this resolution yields a diagram with exact rows, where the leftmost block commutes and all vertical maps except the dashed one are isomorphisms. 
$$\begin{tikzcd}
\widetilde{M}_{k,\In}\arrow[r, "\overline{\gamma}_k"] \arrow[d,"\psi_{\hat{k}}"] & \widetilde{M}_{k,\out} \arrow[r, "\overline{\alpha}_k"] \arrow[d, "\psi_{\hat{k}}"] & \overline{M}_k \arrow[r,two heads,"\pi"] \arrow[d, dashed] &  \frac{\ker \alpha_k}{\Img \gamma_k}\arrow[d, "\psi_k^1"] \\
(\widetilde{P}_{k})_{k,\In} \arrow[r, "\widetilde{\gamma}_k"]                        & (\widetilde{P}_k)_{k,\out}\arrow[r,"\widetilde{\alpha}_k"]                            & (\widetilde{P}_k)_{k}\arrow[r,two heads, "\pi"]                            & \widetilde{E}_k          
\end{tikzcd}$$
Therefore, there exists an $H_k$-module isomorphism $\psi_k$ such that $\psi_k$ is of the form $\textrm{diag} (\psi_k^2,\psi_k^1)$ and $ \psi_k^2\overline{\alpha}_k= \widetilde{\alpha}_k\psi_{\hat{k}}$. and  As in Step 1 (1), one can check that $\psi_{\hat{k}}\overline{\beta}_k=\widetilde{\beta}_k\psi_k$. Let $\psi_i=\psi_{\hat{k}}$ for $i\neq k$. It is the desired $H$-module isomorphism. Here with some abuse of notation, we may write $\psi_k$ in the form  
\begin{align}\label{step 1 2 notation}
\psi_k=\textrm{diag}\left(\widetilde{\alpha}_kf_{\widetilde{\beta}_k}\varphi\overline{\alpha}_k^{-1}, \widetilde{\beta}_kf_{\widetilde{\beta}_k}\varphi\overline{\beta}_k^{-1}\right).
\end{align} 

\textbf{Step 2} Let $\mathcal{M}=(M,V)$ be a general representation of $\mathcal{P}(Q,\mathbf{d},\s)$ and its minimal presentation is 
\begin{align*}
P(\beta_-)\xrightarrow{d_{\mathcal{M}}}P(\beta_+).
\end{align*}

 (1) We first consider the case that $\beta_{-}(k)=\beta_+(k)=0$. Then  $\widetilde{\mu}_k(d_{\mathcal{M}})$ is 
 \begin{align*}
     \widetilde{P}(\beta_-)\xrightarrow{}\widetilde{P}(\beta_+).
 \end{align*}
  Let $\widetilde{M}=\widetilde{\mu}_k(\mathcal{M})=(\overline{M},\overline{V})$. We see that $V_k=\overline{V}_k=0$ and $V_{\hat{k}}=\overline{V}_{\hat{k}}$. Recalling that the presentation of the negative part is of the form $P\to 0$,  we can assume the negative  part of $\mathcal{M}$ is zero.
  
 As discussed before, there is an isomorphism  between $\widetilde{\mu}_k(P_{\hat{k}})$ and $\widetilde{P}_{\hat{k}}$. Hence, $\widetilde{\mu}_k(d_{\mathcal{M}})$ is equivalent to the following exact sequence:       
 \begin{align}\label{step 1 2}
\widetilde{\mu}_k(P(\beta_-))\xrightarrow{d_{\mathcal{M}}}\widetilde{\mu}_k(P(\beta_+))\xrightarrow{}M'\xrightarrow{}0
 \end{align}
 To see that  $\overline{\mathcal{M}}$ is right-equivalent to $\Coker(\widetilde{\mu}_k(d_{\mathcal{M}}))$, it suffices to show that $M'=\overline{M}$.
Clearly,  $M'_i=\overline{M}_i$ for $i\neq k$, and $M'(c)=\overline{M}(c)$ for $c\in \widetilde{Q}_1$ not incident to $k$. As for the $H_k$-module $M'_k$, according to (\ref{mutation of P_hat k}), it is given by 
\begin{align*}
    \frac{\ker \gamma_{k,-}}{\Img \beta_{k,-}}\oplus\Img \gamma_{k,-} \oplus\{0\} \oplus \{0\}\xrightarrow{d_{\mathcal{M}}}\frac{\ker \gamma_{k,+}}{\Img \beta_{k,+}}\oplus\Img \gamma_{k,+} \oplus\{0\} \oplus \{0\}\xrightarrow{}M'_k\xrightarrow{}0
\end{align*}
(by applying the functor $\Hom_{\mathcal{P}}(\widetilde{P}_{\hat{k}},-)$ on \ref{step 1 2}).
 Here we use the notation $\gamma_{k,\pm}=P(\beta_\pm)(\gamma_k)$ and $\beta_{k,\pm}=P(\beta_\pm)(\beta_k)$. 
From the presentation of $M$, via homological algebra, we have the exact sequence:
\begin{align*}
\Coker\beta_{k,-}\xrightarrow{d_{\mathcal{M}}}\Coker\beta_{k,+}\xrightarrow{}\Coker\beta_{k,M}\xrightarrow{}0.
\end{align*}
Hence we have $M'_k=\Coker\beta_{k,M}=\frac{\ker \gamma_{k,M}}{\Img \beta_{k,M}}\oplus \Img \gamma_{k,M}\oplus\{0\} \oplus \{0\}$. Moreover, $d_\mathcal{M}$ induces the following complexes of $H_k$-modules via homological algebra:
\begin{align*}
    \frac{\ker \gamma_{k,-}}{\Img \beta_{k,-}}\xrightarrow{d_{\mathcal{M}}}\frac{\ker \gamma_{k,+}}{\Img \beta_{k,+}}\twoheadrightarrow{}\frac{\ker \gamma_{k,M}}{\Img \beta_{k,M}},\\
    \Img \gamma_{k,-} \xrightarrow{d_{\mathcal{M}}} \Img \gamma_{k,+} \twoheadrightarrow{}\Img \gamma_{k,M}.
\end{align*}
Moreover, the maps $\overline{\alpha}_{k,+}$ and $\overline{\beta}_{k,+}$ of $\widetilde{\mu}_k(P(\beta_+)$ with the splitting data $\rho$  induce the maps $\overline{\alpha}_{k,M'}$ and $\overline{\beta}_{k,M'}$ of the form 
\begin{align}
     \overline{\alpha}_{k,M'}=\left(\begin{array}{c}
-\pi \rho \\
-\gamma_{k,M} \\
0 \\
0
\end{array}\right), \quad \overline{\beta}_{k,M'}=\left(\begin{array}{llll}
0, & \iota, & 0, & 0
\end{array}\right).
\end{align}
This proves that $M'=\overline{M}$.

(2) Next we consider the case $\beta_-(k)=0$ and $\beta_{+}(k)\neq 0$. We can write $d_{\mathcal{M}}$ in the block form
\begin{align*}
    P(\beta_{-})\xrightarrow{\left(\begin{array}{c}
d_1 \\
d_2
\end{array}\right)} {\beta_+(k)}P_k\oplus P(\beta_{+}(\hat{k}))
\end{align*}
where $P(\beta_{+}(\hat{k}))=\oplus_{i\neq k} \beta_{+}(i)P_i$. Then $\widetilde{\mu}_k(d_{\mathcal{M}})$ equals to 
\begin{align*}
    \beta_+(k)\widetilde{P}_k\oplus \widetilde{P}(\beta_{-})\xrightarrow{
     \left(\begin{array}{cc}
d_0 & d_1'\\
0& d_2
\end{array}\right)
     } \bigoplus_{b:k\to j\in Q^\circ_1}d_k\beta_+(k)\widetilde{P}_j\oplus\widetilde{P}(\beta_{+}(\hat{k})),
\end{align*}
where
\begin{align*}
d_0 = \operatorname{diag}(\underbrace{f_{\widetilde{\alpha}_k}, \dots, f_{\widetilde{\alpha}_k}}_{\beta_+(k)}),\quad
 f_{\widetilde{\alpha}_k}=\left(\begin{array}{c}
b^{*}\\
\vdots\\b^*\varepsilon_k^{d_k-1}\end{array}\right)_b,\quad
     d_1'=\left(\begin{array}{c}
d_1b^{-1}\\
\vdots\\
d_1(b\varepsilon_k^{d_k-1})^{-1}
\end{array}\right)_b.
\end{align*}
Here $(b\varepsilon_k^{l})^{-1}$ is the formal inverse of $ \widetilde{P}(\beta_{-})$.
 From the discussion in step 1 (2), there is an isomorphism $\psi=(\psi_i)$ between  $\Coker f_{\widetilde{\alpha}_k}$ and $\widetilde{\mu}_k(P_k)$. Then we have the following exact sequence:
 \begin{align*}
     \widetilde{P}_k\xrightarrow{
      f_{\widetilde{\alpha}_k}
     } \bigoplus_{b:k\to j\in Q^\circ_1}d_k\widetilde{P}_j\xrightarrow{\psi\pi}\widetilde{\mu}_k(P_k).
 \end{align*}
Moreover, $\psi$ $\psi_i=f_{\beta_k}\varphi^{-1}$ for $i\neq k$ and $\psi_k=\operatorname{diag}\left(\overline{\alpha}_kf_{\beta_k}\varphi^{-1}\widetilde{\alpha}_k^{-1}, \overline{\beta}_kf_{\beta_k}\varphi^{-1}\widetilde{\beta}_k^{-1}\right)$.
Then we consider the composition of $\widetilde{\mu}_k(d_{\mathcal{M}})$ and the map $$f=\left(\begin{array}{cc}
   \Psi & 0 \\
    0 & \text{Id}
\end{array}\right):\bigoplus_{b:k\to j\in Q^\circ_1}d_k\beta_+(k)\widetilde{P}_j\oplus\widetilde{P}(\beta_{+}(\hat{k}))\to \beta_+(k)\widetilde{\mu}_k(P_k)\oplus\widetilde{P}(\beta_{+}(\hat{k})),$$
where $\Psi= \operatorname{diag}(\psi\pi,\cdots,\psi\pi)$.  Since $\Psi\circ d_0=0$, we get the following presentation of $M''$:  
\begin{align}\label{constrcut iso 2}
     \beta_+(k)\widetilde{P}_k\oplus \widetilde{P}(\beta_{-})\xrightarrow{
     \left(\begin{array}{cc}
0 & \Psi\circ d_1'\\
0& d_2
\end{array}\right)
     } \beta_+(k)\widetilde{\mu}_k(P_k)\oplus\widetilde{P}(\beta_{+}(\hat{k}))\twoheadrightarrow M''.
\end{align}
Since $f$ is surjective and $\ker f\subseteq \Img \widetilde{\mu}_k(d_{\mathcal{M}})$, we see that $M''=\Coker f\circ\widetilde{\mu}_k(d_{\mathcal{M}})\cong \Coker\widetilde{\mu}_k(d_{\mathcal{M}})$.
Recall the constructed isomorphism between $\widetilde{\mu}_k(P_{\hat{k}})$ and $\widetilde{P}_{\hat{k}}$ in step 1 (1), we  see that   (\ref{constrcut iso 2}) is equivalent to the following exact sequence: 
$$\begin{tikzcd}
 \widetilde{\mu}_k(P(\beta_{-})) \arrow[r, " \left(\begin{array}{c}
 d_1\\
d_2
\end{array}\right)"]                        & \beta_+(k)\widetilde{\mu}_k(P_k)\oplus\widetilde{\mu}_k(P(\beta_{+}(\hat{k})))\arrow[r]  \arrow[r, two heads]                           & M'' .                          
\end{tikzcd}$$
By a similar analysis in step 2 (1), we have that $ M''=\widetilde{\mu}_k(M)$, as desired.
The proof for the case  that $\beta_-(k)\neq 0$ and $\beta_{+}(k)=0$ follows by an analogous argument.

The proof of the injective case is similar.
\end{proof}
Following the proof of Lemma \ref{mutation of pre and rep}, we immediately obtain the validity of Proposition \ref{locally free is mutation invariants}. The following theorem is a direct consequence of the argument used in the proof of Proposition \ref{mutation of pre and rep}.
\begin{thm}\label{mutation rule of g}
Let $(Q(B),\mathbf{d},\s)$ be a reduced and locally free $H$-based QP satisfying \eqref{condition on mutation of GQP 1}. Let $\mathcal{M}=(M,V)$ be a general representation of $(Q,\mathbf{d},\mathcal{S})$, let $\overline{\mathcal{M}}=(\overline{M},\overline{V})=\mu_k(\mathcal{M})$ for some $k\in Q_0$. Let $g=g_{\mathcal{M}}$ and $g'=g_{\mu_k(\mathcal{M})}$. We use the similar notation for the dimension vector $d$. Then:
\begin{align}\label{gg'}
    g'(i)=\begin{cases}
    -g(k) &\text{if $i=k$}\\
    g(i)+d_k[-b_{i,k}]_+\beta_-(k)-d_k[b_{i,k}]\beta_+(k)& \text{if $i\neq k$}.
    \end{cases}
\end{align}
\begin{align}\label{chechgg'}
    \Check{g}'(i)=\begin{cases}
    -\Check{g}(k) &\text{if $i=k$}\\
    \Check{g}(i)+d_k[-b_{k,i}]_+\check{\beta}_-(k)-d_k[b_{k,i}]\check{\beta}_+(k)& \text{if $i\neq k$}.
    \end{cases}
\end{align}
\begin{align}\label{gbeta}
    g(k)=\beta'_+(k)-\beta_{+}(k), \quad\Check{g}(k)=\check{\beta}'_+(k)-\check{\beta}_+(k).
\end{align}
\begin{align}
    d'(k)=d_k\sum_{i\to k\in Q_1^\circ}d(i)-d(k)+d_k\beta_+(k)+d_k\Check{\beta}_-(k)=d_k\sum_{k\to j\in Q_1^\circ}d(j)-d(k)+d_k\beta_-(k)+d_k\Check{\beta}_+(k).
\end{align}
\end{thm}

\section{Some Mutation invariants in Decorated representations}\label{section: some Mutation invariants in Decorated representations}
Let $\mathcal{M}=(M,V)$ and $\mathcal{N}=(N,W)$ be general $H$-based QP-representations of a reduced and locally free $H$-based QP $(Q,\mathbf{d},\s)$. We fix a vertex $k\in Q_0$ and assume that $Q$ satisfies (\ref{condition on mutation of GQP 1}) and (\ref{condition on nutation of GQP 2}). Thus, the mutated $H$-based QP $(\overline{Q},\mathbf{d},\overline{\s})=\mu_k(Q,\mathbf{d},\s)$ is well-defined, as well as its $H$-based QP representations  $\overline{\mathcal{M}}=(\overline{M},\overline{V})=\mu_k(\mathcal{M})$ and $\overline{\mathcal{N}}=(\overline{N},\overline{W})=\mu_k(\mathcal{N})$.

Let $\mathcal{Q}$ be a subcategory of $\mathcal{P}(Q,\mathbf{d},\s)$-mod. We denote by $[\mathcal{Q}](M, N)$ the subgroup of $\Hom_{\mathcal{P}}(M, N)$ consisting of morphisms factoring through objects in $\mathcal{Q}$.  In this section, we will consider the case when $\mathcal{Q}=\add E_k$.

\begin{lem}\label{real1}
The following $H_k$-modules are isomorphic:
\begin{equation*}
    [\add E_k](M,N)\cong \Hom_{H_k}(\Coker \alpha_{k,M},\ker \beta_{k,N}).
\end{equation*}
\end{lem}
\begin{proof}
 By the discussion below (\ref{p}), we have that $\Hom_{\mathcal{P}}(M,E_k)\cong \Hom_{H_k}(\Coker \alpha_{k,M},H_k)$ and $\Hom_{\mathcal{P}}(E_k,N)\cong \Hom_{H_k}(H_k,\ker\beta_{k,N})$. Combining these with the natural isomorphisms
 \begin{align*}
     \Hom_{H_k}(\Coker \alpha_{k,M},\ker \beta_{k,N})&\cong  \Hom_{H_k}(\Coker \alpha_{k,M},H_k)\otimes_{H_k}\Hom_{H_k}(H_k,\ker\beta_{k,N}),\\
     [\add E_k](M,N) &\cong \operatorname{Hom}_{\mathcal{P}}(M, E_k) \otimes_{H_k} \operatorname{Hom}_{\mathcal{P}}(E_k, N),
 \end{align*}
 we conclude that
$[\add E_k](M,N) \cong \Hom_{H_k}(\Coker \alpha_{k,M}, \ker \beta_{k,N}).$ \qedhere
\end{proof}
\begin{pro}\label{real2}
The mutation $\mu_k$ induces an isomorphism 
$$\Hom_{\mathcal{P}} (M,N)/[\add E_k](M,N)\cong \Hom_{\overline{\mathcal{P}}}(\overline{M},\overline{N})/[\add\overline{E}_k](\overline{M},\overline{N}).$$
\end{pro}
\begin{proof}
The proof follows the same argument as in \cite[Proposition 6.1]{derksen2010quivers}. 
\end{proof}
\begin{cor}\label{cor1}
Let $\overline{\mathcal{P}}=\mathcal{P}(\overline{Q},\mathbf{d},\overline{\s})$. We have that 
\begin{enumerate}
    \item[(1)]  $\hom_{\overline{\mathcal{P}}}(\overline{\mathcal{M}},\overline{\mathcal{N}})-\hom_{\mathcal{P}}(\mathcal{M},\mathcal{N})=d_k(\beta_{-,\mathcal{M}}(k)\check{\beta}_{-,\mathcal{N}}(k)-\beta_{+,\mathcal{M}}(k)\check{\beta}_{+,\mathcal{N}}(k))$;
    \item[(2)] $\e_{\overline{{\mathcal{P}}}}(\overline{\mathcal{M}},\overline{\mathcal{N}})-\e_{\mathcal{P}}(\mathcal{M},\mathcal{N})=d_k(\beta_{+,\mathcal{M}}(k)\beta_{-,\mathcal{N}}(k)-\beta_{-,\mathcal{M}}(k)\beta_{+,\mathcal{N}}(k))$;
    \item[(2$^\star$)]$\check{\e}_{\overline{{\mathcal{P}}}}(\overline{\mathcal{M}},\overline{\mathcal{N}})-\check{\e}_{\mathcal{P}}(\mathcal{M},\mathcal{N})=d_k(\check{\beta}_{-,\mathcal{M}}(k)\beta_{+,\mathcal{N}}(k)-\check{\beta}_{+,\mathcal{M}}(k)\beta_{-,\mathcal{N}}(k))$.
\end{enumerate}
\end{cor}
\begin{proof}
(1) follows directly from Lemma \ref{g-vector of M}, Lemma \ref{real1} and Proposition \ref{real2}. (2) and ($2^\star$) are proved by straightforward calculation using (1), Theorem \ref{mutation rule of g} and the following equations
\begin{align*}
    \hom_{\mathcal{P}}(\mathcal{M},\mathcal{N})-\e_{\mathcal{P}}(\mathcal{M},\mathcal{N})=-g_{\mathcal{M}}(\underline{\dim}N);\\
    \hom_{\mathcal{P}}(\mathcal{M},\mathcal{N})-\check{\e}_{\mathcal{P}}(\mathcal{M},\mathcal{N})=-\check{g}_{\mathcal{N}}(\underline{\dim}M). \tag*{\qedhere}
\end{align*}
\end{proof}
\begin{lem}[\cite{derksen2015general}]
    The normal space of $d$ in $\Hom_\mathcal{P}\left(P(\beta_-),P(\beta_+)\right)$ can be identified with $E(d,d)$.
\end{lem}
\begin{pro}\label{general is mutation invarints}
If $d_{\mathcal{M}}$ is a general presentation, then so is  $d_{\mu_k(\mathcal{M})}$.
\end{pro}
\begin{proof}
    Using Lemma \ref{mutation of pre and rep}, we can take a dense open subset $U$ of $\PHom_{\mathcal{P}}(g)$ such that $e:=e_{\mathcal{P}}(d_{\mathcal{M}})$ is constant and $\widetilde{\mu}_k(d_{\mathcal{M}})=d_{\widetilde{\mu}_k(\mathcal{M})}$ when we let $d_{\mathcal{M}}$ vary inside $U$. Furthermore, Lemma \ref{mutation of pre and rep} and Corollary \ref{cor1} imply that $e_{\mathcal{P}}(d_{\mathcal{M}})=e_{\widetilde{\mathcal{P}}}(\widetilde{\mu}_k(d_{\mathcal{M}}))$ and $\widetilde{\mu}_k(d_{\mathcal{M}})\in \Hom(\widetilde{P}(\beta_-'),\widetilde{P}(\beta_+'))$ for some vectors $\beta_\pm'\in \mathbb{N}^{|Q_0|}$.   Set 
    \begin{align*}
        W:=\{(d_{\mathcal{M}},d_{\mathcal{M}'})\in U\times \Hom(\widetilde{P}(\beta_-'),\widetilde{P}(\beta_+'))\mid \text{$\widetilde{\mu}_k({d_{\mathcal{M}}})$ and $d_{\mathcal{M}'}$ are isomorphic}\}.
    \end{align*}
    Being the image of the morphism $U\times \operatorname{Aut}_{\mathcal{P}}\left(P(\beta'_{-})\right) \times \operatorname{Aut}_{\mathcal{P}}\left(P(\beta'_{+})\right)\to U\times \Hom(\widetilde{P}(\beta_-'),\widetilde{P}(\beta_+'))$ given by the rule  $(d_{\mathcal{M}},f'_-,f'_+)\mapsto (d_{\mathcal{M}}, f'_-\widetilde{\mu}_k(d_{\mathcal{M}})f'_+)$, the set $W$ is constructible in $U\times \Hom(\widetilde{P}(\beta_-'),\widetilde{P}(\beta_+'))$; and irreducible  since $\widetilde{\mu}_k:U\to \Hom(\widetilde{P}(\beta_-'),\widetilde{P}(\beta_+'))$ is a regular map. Let 
    $$\begin{tikzcd}
  & \PHom_{\mathcal{P}}(g)\times \Hom(\widetilde{P}(\beta_-'),\widetilde{P}(\beta_+')) \arrow[ld, "\pi_1"'] \arrow[rd, "\pi_2"] &           \\
\PHom_{\mathcal{P}}(g) &                                                           & \Hom(\widetilde{P}(\beta_-'),\widetilde{P}(\beta_+'))
\end{tikzcd}$$
be the projections, and set 
\begin{align*}
    p_1:=\pi_1\mid_W: W\to U\quad \text{and} \quad p_2:=\pi_2\mid_W:W\to \Hom(\widetilde{P}(\beta_-'),\widetilde{P}(\beta_+')).
\end{align*}
 Set $Y:=P_2(W)$. Being the image of an irreducible set under a continuous function, $Y$ is irreducible. For each $d_{\mathcal{M}}\in U$, the fiber $p_1^{-1}(d_{\mathcal{M}})\in W$ is canonically isomorphic to the $\operatorname{Aut}_{\mathcal{P}}\left(P(\beta'_{-})\right) \times \operatorname{Aut}_{\mathcal{P}}\left(P(\beta'_{+})\right)$-orbit of $\widetilde{\mu}_k(d_\mathcal{M})$. Hence, $p_1^{-1}(d_{\mathcal{M}})$ is irreducible and 
 \begin{align*}
     \dim (p_1^{-1}(d_{\mathcal{M}}))=\dim \Hom(\widetilde{P}(\beta_-'),\widetilde{P}(\beta_+'))-e.
 \end{align*}
 Thus, all fibers of $p_1$ are irreducible of the same dimension $\dim \Hom(\widetilde{P}(\beta_-'),\widetilde{P}(\beta_+'))-e$ and so $\dim W=\dim U+ \dim \Hom(\widetilde{P}(\beta_-'),\widetilde{P}(\beta_+'))-e= \dim\PHom_{\mathcal{P}}(g)+\dim\Hom(\widetilde{P}(\beta_-'),\widetilde{P}(\beta_+'))-e$. On the other hand, for $(d_{\mathcal{M},}d_{\mathcal{M'}})\in W$, we have 
 \begin{align*}
     p_2^{-1}(d_{\mathcal{M}'})=(\operatorname{Aut}_{\mathcal{P}}\left(P([g]_{-})\right) \times \operatorname{Aut}_{\mathcal{P}}\left(P([g]_{+})\right)\cdot d_{\mathcal{M}},d_{\mathcal{M'}}),
 \end{align*} 
 it is irreducible of dimension $\dim\PHom_{\mathcal{P}}(g)-e$. Hence $\dim W=\dim Y+\dim \PHom_{\mathcal{P}}(g)-e$. So  $\dim Y=\dim\Hom(\widetilde{P}(\beta_-'),\widetilde{P}(\beta_+'))$ implying that $Y$ is dense in $\Hom(\widetilde{P}(\beta_-'),\widetilde{P}(\beta_+'))$. Hence, $d_{\widetilde{\mu}(\mathcal{M})}$ is a general presentation in  $\Hom(\widetilde{P}(\beta_-'),\widetilde{P}(\beta_+'))$. 
 It follows from  Lemma \ref{general pre is homotopy equivalent} and the proof of Lemma \ref{mutation of pre and rep} that $d_{\mu_k(\mathcal{M})}$ is a general presentation in $\PHom_{\mathcal{P}}(g')$, where $g'$ is explained in \eqref{gg'}.
\end{proof}

\section{\texorpdfstring{$F$}{}-polynomials of \texorpdfstring{$H$}{}-based QP representations}\label{section: F-polynomial}
In this section we work with an $H$-based QP  $(Q,\mathbf{d},\s)$ and its $B$ matrix $B(Q)$. We denote by $H_{k,l}$ the indecomposable submodule of $H_k$ which is generated by $ \varepsilon_k^{d_k-l} $ for each vertex $k\in Q_0$ and $1 \leq l \leq d_k$. Now, we focus on the case where $d_k\leq 2$. Suppose that $M_k$ is an $H_k$-module, then $M_k$ can be decomposed uniquely to $\bigoplus\limits_{l=1}^{d_k}r_l H_{k,l}$ for some $r_l\in \mathbb{N}$. Let $r=(r_1,\cdots, r_{d_k})\in \mathbb{N}^{d_k}$, we call $r$ the type of $M_k$.  

Recall that for any $\xi \in Q^\circ_1$, the cyclic derivative $\partial_\xi$ is defined as in Equation~\eqref{derivative}. Although $\varepsilon_k$ is not an element of $Q^\circ_1$ for each vertex $k \in Q_0$, we may formally extend the definition of the cyclic derivative to allow $\xi = \varepsilon_k$. For any $H$-based QP-module $\mathcal{M} = (M,V)$, this induces an endomorphism $M(\partial_{\varepsilon_k}\mathcal{S}): M_k \to M_k$. Moreover, it is an $H_k$-linear map since:
\begin{align*}
\varepsilon_k\partial_{\varepsilon_k}\mathcal{S} &= \varepsilon_k\sum_{a,b}a\partial_{[b\varepsilon_ka]}\mathcal{S}b \\
    &= \sum_b \partial_b\mathcal{S}b - \sum_{b,a}a\partial_{[ba]}\mathcal{S}b\\
    &= \sum_a a\partial_a\mathcal{S} - \sum_{b,a}a\partial_{[ba]}\mathcal{S}b \\
    &= \sum_{a,b}a\partial_{[b\varepsilon_ka]}\mathcal{S}b\varepsilon_k.
\end{align*}
This induces an $H_k$-module endomorphism of $M_{k,1} := \ker M(\varepsilon_k)/\operatorname{Im} M(\varepsilon_k)$. We regard $M(\partial_{\varepsilon_k}\mathcal{S})$ as a $K$-linear map on the vector space $M_{k,1}$. It is nilpotent since the Jacobian algebra $\mathcal{P}(Q,\mathbf{d},\mathcal{S})$ is finite-dimensional. We denote by $t_{k,l}$ the number of Jordan blocks of size $l$ associated with $M(\partial_{\varepsilon_k}\mathcal{S})$, and call $\mathbf{t} = (t_{k,l})_{k,l}$ the \textbf{Jordan type} of $M$. Note that $\dim M_{k,1} = r_{k,1}$, where $r_k = (r_{k,1}, r_{k,d_2})$ is the type of $M_k$ when $d_k=2$. In the case $d_k=1$, the map $M(\partial_{\varepsilon_k}\mathcal{S})$ is trivial with $ M_{k,1} = M_k$.

Recall that the quiver Grassmannian $\mathrm{Gr}_{\mathbf{e}}(M)$ is the variety parametrizing $\mathbf{e}$-dimensional subrepresentations of $M$. In simple terms, an element of $\mathrm{Gr}_{\mathbf{e}}(M)$ is an $n$-tuple $(N_1,\cdots,N_n)$, where $N_i$ is a subspace of dimension $e_i$ in $M_i$, and $a_M(N_i)\subseteq N_j$ for any arrow $a:i \to j$. Thus, $\mathrm{Gr}_{\mathbf{e}}(M)$ is a closed subvariety of the product of ordinary Grassmannians $\prod\limits_{i=1}
^n\mathrm{Gr}_{e_i}(M_i)$, and is therefore a projective variety. Now for a given tuple of nonnegative integers $\mathbf{t}=(t_{i,l})_{i=1,\ldots,n;l=1,\ldots,\dim M_i}$, we define the subset $\mathrm{Gr}_{\mathbf{e},\mathbf{t}}(M)$ of $\mathrm{Gr}_{\mathbf{e}}(M)$ by
\begin{equation}
    \mathrm{Gr}_{\mathbf{e},\mathbf{t}}(M)=\{N\in \mathrm{Gr}_{\mathbf{e}}(M)\mid \text{the Jordan type of $N$ is $\mathbf{t}$}\}.
\end{equation}
Then $\mathrm{Gr}_{\mathbf{e},\mathbf{t}}(M)$ is a locally closed subset of $\mathrm{Gr}_{\mathbf{e}}(M)$, and so a quasi-projective variety.
\begin{defn}\label{defn of F-pplynomial} 
For each vertex $k \in Q_0$  and integer $l \geq 0$, we define polynomials $f_{k,l}(z_k) \in \mathbb{Z}[z_k]$ recursively by:
\begin{equation}\label{eq:recursion}
\begin{cases}
f_{k,0}(z_k) = 1, \\
f_{k,1}(z_k) = z_k, \\
f_{k,l}(z_k) = z_k f_{k,l-1}(z_k) - f_{k,l-2}(z_k), & \text{for } l \geq 2.
\end{cases}
\end{equation}
Given a decorated representation $\mathcal{M}=(M,V)$ of $(Q,\mathbf{d},\s)$, its \textbf{$F$-polynomial} is defined as 
\begin{equation}
F_\mathcal{M}(\mathbf{y},\mathbf{z})=\sum_{\mathbf{e},\mathbf{t}} \chi(\mathrm{Gr}_{\mathbf{e},\mathbf{t}}(M)) \mathbf{f}^\mathbf{t}(\mathbf{z})\mathbf{y}^\mathbf{e},
\end{equation}
where $\mathbf{f}^\mathbf{t}(\mathbf{z})=\prod_{i,l}f_{i,l}^{t_{i,l}}(z_{i})$, $\mathbf{y}^\mathbf{e}=\prod_i y_i^{e_i}$, and $\chi(-)$ denotes the topological Euler characteristic.
\end{defn} 
 
\begin{lem}
For all representations $M_1$ and $M_2$, we have $F_{M_1\oplus M_2}=F_{M_1}F_{M_2}$.
\end{lem}
\begin{proof}
    The proof follows similarly to \cite[Proposition 3.2]{derksen2010quivers}.
\end{proof}
\begin{thm}\label{mutation of F-polynomial}
Let $(Q(B),\mathbf{d},\mathcal{S})$ be a locally free and nondegenerate $H$-based QP with $d_k\leq 2$ for some vertex $k\in Q_0$.  Let $\mathcal{M}=(M,V)$ be a general representation of $(Q(B),\mathbf{d},\mathcal{S})$, and let $\overline{\mathcal{M}}=(\overline{M},\overline{V})=\mu_k(\mathcal{M})$. Suppose that the $\mathbf{y}$-seed $(\mathbf{y}',B_1)$ in $\mathbb{Q}_{\mathrm{sf}}(y_1,\cdots,y_n)$ is obtained from $(\mathbf{y},B)$ by mutation at $k$, where the mutation rule of $\mathbf{y}$ is given by 
\begin{align}\label{mutation for y}
y_{i}^{\prime}= \begin{cases} y_{k}^{-1} & i=k \\
y_{i}\left(y_{k}^{\left[b_{ki}\right]_{+}}\right)^{d_{k}}\left(\sum_{l=0}^{d_{k}} z_{k,l} y_{k}^{ l}\right)^{-b_{ki}} & i \neq k.\end{cases}
\end{align}
Then the $F$-polynomial $F_{\mathcal{M}}$ and $F_{\overline{\mathcal{M}}}$ are related by 
\begin{equation}\label{F-mutation}
(\sum_{l=0}^{d_k}z_{k,l}y_k^l)^{- \check{\beta}_+(k)}F_{\mathcal{M}}(y_1,\cdots,y_n)= (\sum_{l=0}^{d_k}z_{k,l}(y'_k)^l)^{-\check{\beta}'_+(k)}F_{\overline{\mathcal{M}}}(y'_1,\cdots,y'_n)
\end{equation}
where $z_{k,l}=1$ for $l\leq d_k$ when $d_k=1$, and $z_{k,0}=z_{k,2}=1$ with $z_{k,1}=z_k$ when $d_k=2$.
\end{thm}
\begin{proof}
 Suppose that $N=(N_1,\ldots,N_n)\in \prod_{i=1}^n\mathrm{Gr}_{e_i}(M_i)$, and let $N_{k,\In}$ and $N_{k,\out}$ be the corresponding subspace of $M_{k,\In}$ and $M_{k,\out}$, respectively (see \eqref{defn of M_in and M_out}). The condition that $N\in \mathrm{Gr}_{\mathbf{e},\mathbf{t}}(M)$ can be stated as the combination of the following three conditions:
\begin{equation}\label{subcondition 1}
    \text{$c_M(N_i)\subseteq N_j$ for any arrow $c:i\to j$ not incident to $k$ in $Q(B)$}.
\end{equation}
\begin{equation}\label{subcondition 2}
    \alpha_k(N_{k,\In})\subseteq N_k\subseteq \beta_k^{-1}(N_{k,\out})
\end{equation}
\begin{equation}\label{subcondition 3}
    \text{$N_i$ is an $H_i$-module with Jordan type $\mathbf{t}_i$ for $i\in Q_0$.}
\end{equation} 
When $d_k=1$, the setup reduces to the classical QP case, and mutation coincides with the classical construction. The proof for this case is given in \cite[Lemma 5.2]{derksen2010quivers}. We now treat the case 
$d_k=2$.
  
Recalling the definition of $\alpha_k$ and $\beta_k$ (\eqref{defn of alpha}, \eqref{defn of beta}), we can express $\alpha_k$ and $\beta_k$ as the following matrices 
\begin{align*}
  \alpha_k= \begin{pmatrix}
    a, \varepsilon_k a
\end{pmatrix}_a,\quad
\beta_k=\prescript{}{b}{\begin{pmatrix}
    b\varepsilon_k\\
    b
\end{pmatrix}}.
\end{align*}
Then 
\begin{align}\label{martix of partial_{varepsilon_k}s}
   \partial_{\varepsilon_k}\s=\alpha_k {\begin{pmatrix}
    0&\partial_{[b\varepsilon_ka]}\s \\
    0& 0\\
\end{pmatrix}}_{a,b}\beta_k, 
\end{align}
  so the image of $\partial_{\varepsilon_k}\s$ on $N_k$ lies in $\alpha_{k}\left(N_{k,\In}\right)$. Now we consider the  following two $H_k$-module morphisms:
\begin{align*}
\beta^{-1}_{k}\left(N_{k,\out}\right)\xrightarrow{\partial_{\varepsilon_k}\s}\alpha_{k}(N_{k,\In})\xrightarrow{\iota}\beta_k^{-1}(N_{k,\out}),
\end{align*}
where $\iota$ is the natural embedding. We denote by
\begin{equation}\label{construction of beta-1_1}
    \begin{aligned}
    \beta^{-1}_{k}(N_{k,\out})_1&:=\left(\beta^{-1}_{k}(N_{k,\out})\cap\ker \varepsilon_k\right)/\varepsilon_k(\beta^{-1}_{k}(N_{k,\out})),\\
    \alpha_{k}(N_{k,\In})_1&:=\left(\alpha_{k}(N_{k,\In})\cap\ker \varepsilon_k\right)/\varepsilon_k(\alpha_k(N_{k,\In})).
\end{aligned}
\end{equation}  
Then $\partial_{\varepsilon_k}\s$ and $\iota$ induce the two $K$-linear maps:
\begin{align}\label{defn of iota}
\partial_{\varepsilon_k}\s:\beta^{-1}_{k}(N_{k,\out})_1\to\alpha_{k}(N_{k,\In})_1,\quad \iota:\alpha_{k}(N_{k,\In})_1\xrightarrow{}\beta^{-1}_{k}(N_{k,\out})_1.
\end{align} 
Then $\partial_{\varepsilon_k}\s$ on $\beta^{-1}_{k}(N_{k,\out})_1$ equals the composition of the above two maps. Here to  distinguish it from $ \partial_{\varepsilon_k}\s:\beta^{-1}_{k}(N_{k,\out})_1\to\alpha_{k}(N_{k,\In})_1$, we write $\Phi$ for the former. We denote $(\phi_l)_{1\leq l\leq m}$ the Jordan type of $\Phi$. Then there is a Jordan basis $\left\{\Phi^r(\eta_{l,s})\mid1\leq l\leq m,1\leq r\leq l-1, 1\leq s\leq \phi_l\right\}$ and the Jordan blocks are determined
by the strings $S_{l,s}=\left\{\eta_{l,s},\Phi(\eta_{l,s}),\ldots,\Phi^{l-1}(\eta_{l,s})\right\}$.
We observe that
\begin{itemize}
    \item $\Coker\Phi\cong \Coker \iota\oplus \frac{\Img \iota}{\Img \Phi}$ and the set $\left\{\eta_{l,s}\mid1\leq l\leq m, 1\leq s\leq \phi_l \right\}$ forms a basis of  $\Coker\Phi$;
    \item $\ker\Phi\cong \ker\partial_{\varepsilon_k}\s\oplus \frac{\ker\Phi}{\ker\partial_{\varepsilon_k}\s}\cong \ker\partial_{\varepsilon_k}\s\oplus \ker\iota\cap \Img \partial_{\varepsilon_k}\s$ and the set $\left\{\Phi^{l-1}(\eta_{l,s})\mid1\leq l\leq m, 1\leq s\leq \phi_l\right\}$ forms a basis of $\ker \Phi$;
   \item $\Coker\Phi$ can be decomposed as follows:
\begin{center}
$\begin{aligned}
  \Coker\Phi \cong 
 \bigoplus_{1 \leq l \leq m}  \bigg( &
    \frac{\ker(\partial_{\varepsilon_k}\s \circ \Phi^{l-1}) + \Img\iota}{\ker\Phi^{l-1} + \Img\iota}
    \oplus 
    \frac{\ker\Phi^{l} + \Img\iota}{\ker(\partial_{\varepsilon_k}\s \circ \Phi^{l-1}) + \Img\iota} \oplus
   \\
  & \frac{\Img\iota \cap \ker(\partial_{\varepsilon_k}\s \circ \Phi^{l-1}) + \Img\Phi}{\Img\iota \cap \ker\Phi^{l-1} + \Img\Phi} 
  \oplus  
    \frac{\Img\iota \cap \ker\Phi^{l} + \Img\Phi}{\Img\iota \cap \ker(\partial_{\varepsilon_k}\s \circ \Phi^{l-1}) + \Img\Phi} 
  \bigg)
\end{aligned}$
\end{center}
\end{itemize}
We denote these constructed spaces by, for $1 \leq l \leq m$,
\begin{equation}\label{construction of space W}
    \begin{aligned}
         W_l^{(1,1)}  &:= \frac{\ker(\partial_{\varepsilon_k}\s \circ \Phi^{l-1}) + \Img\iota}{\ker\Phi^{l-1} + \Img\iota}, 
  &\quad 
  W_{l}^{(1,2)} &:= \frac{\ker\Phi^{l} + \Img\iota}{\ker(\partial_{\varepsilon_k}\s \circ \Phi^{l-1}) + \Img\iota}, \\
  W_l^{(2,1)}  &:= \frac{\Img\iota \cap \ker(\partial_{\varepsilon_k}\s \circ \Phi^{l-1}) + \Img\Phi}{\Img\iota \cap \ker\Phi^{l-1} + \Img\Phi}, 
  &\quad 
  W_l^{(2,2)} &:= \frac{\Img\iota \cap \ker\Phi^{l} + \Img\Phi}{\Img\iota \cap \ker(\partial_{\varepsilon_k}\s \circ \Phi^{l-1}) + \Img\Phi}.
    \end{aligned}
\end{equation}
We then immediately have the isomorphisms:
\begin{equation}\label{decomposed of iota}
   \begin{aligned}
       \Coker\iota \cong \bigoplus_{l,v} W_l^{(1,v)},\quad \ker\iota\cap \Img \partial_{\varepsilon_k}\s\stackrel{(\partial_{\varepsilon_k}\s\circ \Phi^{l-1})_l}{\cong} \bigoplus_{l,u}W_l^{(u,2)},\\
    \Img \iota\cong  \bigoplus_{1\leq f\leq l-1,u,v}\Phi^{f}(W_l^{(u,v)})\oplus \bigoplus_{v}(W_l^{(2,v)}).
   \end{aligned} 
\end{equation}
For $1 \leq u,v \leq 2$, we denote by $\phi_l^{(u,v)}$ the dimension of the $K$-linear space $W_l^{(u,v)}$, which corresponds to the number of Jordan blocks of size $l$ associated to $W_{l}^{(u,v)}$. Furthermore, let $\phi_0^{(2,2)}$ denote the dimension of $W_0^{(2,2)} := \ker\iota/(\Img\partial_{\varepsilon_k}\s \cap \ker\iota)$. We call $\phi = (\phi_l^{(u,v)})$ the \textbf{generalized Jordan type} of $\Phi$.

Let $\hat{\mathbf{e}}=(e_i)_{i\neq k}$  and $\hat{\mathbf{t}}=(\mathbf{t}_i)_{i\neq k}$ denote the integer vectors obtained from $\mathbf{e}$ and $\mathbf{t}$ by forgetting the $k$-th component of each, where $\mathbf{t}_i$ is the $i$-th row of $\mathbf{t}$. Now for such vectors $\hat{\mathbf{e}}$, $\hat{\mathbf{t}}$, $\phi=(\phi_l^{(i,j)})$, and every pair of nonnegative integer vectors $s,t\in \mathbb{N}^{d_k}$, we denote by $Z_{\hat{\mathbf{e}},\hat{\mathbf{t}};\phi,r,s}(M)$ the variety of tuples $(N_i)_{i\neq k}$ satisfying the inclusion \eqref{subcondition 1} and $\alpha_k(N_{k,\In})\subseteq \beta^{-1}_k(N_{k,\out})$, and such that $\dim N_i=e_i$ for $i\neq k$ and  the Jordan type of $N_i$ is $\mathbf{t}_i$, and the types of $\alpha_k(N_{k,\In})$ and $\beta_k^{-1}(N_{k,\out})$ are $r$ and $s$ respectively, and the generalized Jordan type of $\Phi$ is $\phi$. Let $Z'_{\mathbf{e},\mathbf{t};\phi,r,s}(M)$ denote the subset of $\mathrm{Gr}_{\mathbf{e},\mathbf{t}}(M)$ consisting of all $N=(N_1,\cdots,N_n)$ such that the tuple obtained from $N$ by forgetting $N_k$ belongs to $Z_{\hat{\mathbf{e}},\hat{\mathbf{t}};\phi,r,s}(M)$. Then $\mathrm{Gr}_{\mathbf{e},\mathbf{t}}(M)=\sqcup_{\phi,r,s} Z'_{\mathbf{e},\mathbf{t};\phi,r,s}(M)$.

In view of \eqref{subcondition 2} and \eqref{subcondition 3}, each $Z'_{\mathbf{e},\mathbf{t};\phi,r,s}(M)$ is the fiber bundle over $Z_{\hat{\mathbf{e}},\hat{\mathbf{t}};\phi,r,s}(M)$ with the fiber
\begin{align*}
   X_{e_k,\mathbf{t}_k}=\Big\{U\subseteq \beta^{-1}_{k}(N_{k,\out})\big/\alpha_{k}(N_{k,\In}) \big|& \text{ $U+\alpha_{k}(N_{k,\In})$ is an $H_k$-module}\\
   &\text{with Jordan type $t_{k,l}$ and dimension $e_k$}\Big\}. 
\end{align*}
Thus,
\begin{align}\label{expression of F over fiber}
F_{\mathcal{M}}(y_1,\ldots,y_n) = \sum_{\substack{\hat{\mathbf{e}},\hat{\mathbf{t}}\\r,s,\phi}} \sum_{e_k,\mathbf{t}_k} \chi\bigl(Z_{\hat{\mathbf{e}},\hat{\mathbf{t}};\phi,r,s}(M)\bigr) \chi(X_{e_k,\mathbf{t}_k}) \mathbf{f}^{\mathbf{t}}(\mathbf{z}) \mathbf{y}^{\mathbf{e}}.
\end{align}
Clearly,
\begin{align*}
    \beta^{-1}_{k}(N_{k,\out})&=\beta^{-1}_{k}(N_{k,\out})_1\oplus \text{some free $H_k$-modules}\\
    &=\Img \iota \oplus\Coker\iota \oplus \text{some free $H_k$-module},\\
    \alpha_{k}(N_{k,\In})&=\alpha_{k}(N_{k,\In})_1\oplus \text{some free $H_k$-modules}\\
    &=\Img \iota\oplus\ker\iota\oplus \text{some free $H_k$-modules}.
\end{align*}
Recalling the construction of $\iota$ in \eqref{defn of iota}, we see that $\ker\iota$  embeds into some free $H_k$-submodule of $\beta^{-1}_{k}(N_{k,\out})$ with dimension $\dim \ker\iota=r_1-\dim\Img\iota=r_1 - \big(\sum_{l,v} \left(l-1\right)\phi_l^{(1,v)} + l\phi_l^{(2,v)}\big).$ In view of \eqref{decomposed of iota}, we have 
\begin{align*}
     \beta^{-1}_{k}(N_{k,\out})\big/\alpha_{k}(N_{k,\In})&\cong \Coker\iota\oplus\ker\iota\oplus L\\
     &\cong \bigoplus_{l,v} W_l^{(1,v)}\oplus\bigoplus_{l,u}(\partial_{\varepsilon_k}\s\circ \Phi^{l-1})(W_l^{(u,2)})\oplus L.
\end{align*}
where $L$  is a free $H_k$-module of rank $s_2 - r_2 - r_1 + \sum_{l,v} \big((l-1)\phi_l^{(1,v)} + l\phi_l^{(2,v)}\big).$ 

We then use the following well-known property of the Euler-Poincaré characteristic: if a complex torus $T$ acts algebraically on a variety $X$, then $\chi(X) = \chi(X^T)$, where $X^T$ is the set of $T$-fixed points (see \cite{BIALYNICKIBIRULA197399}). Take $X = X_{e_k,\mathbf{t}_k}$, and consider the action of $T = (\mathbb{C}^{4m+2})^\star$ on $X$ induced by the $T$-action on $\bigoplus_{l,v} W_l^{(1,v)}\oplus\bigoplus_{l,u}(\partial_{\varepsilon_k}\s\circ \Phi^{l-1})(W_l^{(u,2)})\oplus L$ given by
\begin{align*}
    \mathbf{t} \cdot (w_{1},\ldots,w_{4m+2}) = (t_1 w_1,\ldots,t_{4m+2} w_{4m+2}).
\end{align*}
Then a point $U \in X$ is $T$-fixed if and only if $U$ admits an $H_k$-module decomposition 
\[
U = U_l^{(1,1)} \oplus U_l^{(2,2)} \oplus U_{l,1}^{(1,2)} \oplus U_{l,2}^{(1,2)} \oplus U_0,
\]
where the components satisfy the following inclusions:
\begin{align*}
U_l^{(1,1)} &\subseteq W_l^{(1,1)}, \\
U_l^{(2,2)} &\subseteq (\partial_{\varepsilon_k}\s \circ \Phi^{l-1})(W_l^{(2,2)}), \\
U_{l,1}^{(1,2)} &\subseteq W_l^{(1,2)}, \\
U_{l,2}^{(1,2)} &\subseteq (\partial_{\varepsilon_k}\s \circ \Phi^{l-1})(W_l^{(1,2)}), \\
U_0 &\subseteq L.
\end{align*} Thus, we have 
\begin{align*}
    X^{T}=&\bigsqcup_{(\theta_{l}^{(1,1)}, \theta_{l}^{(2,2)}, \theta_{1,1}^{(1,2)}, \theta_{l,2}^{(1,2)}, \Delta_l, \theta_0)\in \Omega}\prod_{1\leq l\leq m}\mathrm{Gr}_{\theta_{l}^{(1,1)}}(W_{l}^{(1,1)}) \times \prod_{0\leq l\leq m}\mathrm{Gr}_{\theta_{l}^{(2,2)}}\big(\partial_{\varepsilon_k}\s\circ\Phi^{l-1}(W_l^{(2,2)})\big)\times\\
    &\prod_{1\leq l\leq m}X_{\theta_{l,1}^{(1,2)},\theta_{l,2}^{(1,2)},\Delta_l}\times \mathrm{Gr}_{\theta_0}(L)
\end{align*}
where
\begin{align*}
X_{\theta_{l,1}^{(1,2)},\theta_{l,2}^{(1,2)},\Delta_l}= \bigg\{ &
(U_{l,1}^{(1,2)},U_{l,2}^{(1,2)}) \in \mathrm{Gr}_{\theta_{l,1}^{(1,2)}}(W_l^{(1,2)}) 
\times \mathrm{Gr}_{\theta_{l,2}^{(1,2)}}\big(\partial_{\varepsilon_k}\s \circ \Phi (W_l^{(1,2)})\big) 
 \\
&\quad\quad\Bigg|\Delta_l=\dim\Big(\big(\partial_{\varepsilon_k}\s \circ \Phi^{l-1} (U_{l,1}^{(1,2)})\big) \cap U_{l,2}^{(1,2)}\Big)\bigg\},
\end{align*}
  and the index set $\Omega$  consists of all variables $(\theta_{l}^{(1,1)}, \theta_{l}^{(2,2)}, \theta_{1,1}^{(1,2)}, \theta_{l,2}^{(1,2)}, \Delta_l, \theta_0)$ that satisfy the following conditions:
\begin{align*}
t_{k,l} =& \phi_{l}^{(2,1)} + \theta_{l}^{(1,1)} + \big(\phi_{l+1}^{(1,1)} - \theta_{l+1}^{(1,1)}\big) + \theta_{l}^{(2,2)}+ \big(\phi_{l-1}^{(2,2)} - \theta_{l-1}^{(2,2)}\big) 
  + \Delta_l + \big(\theta_{l-1,1}^{(1,2)} - \Delta_{l-1} \big)\\
  &+\big(\theta_{l+1,2}^{(1,2)} - \Delta_{l+1}\big)
 +\big(\phi_{l}^{(1,2)} - \theta_{l,1}^{(1,2)} - \theta_{1,2}^{(1,2)} + \Delta_l \big), \\
e_k =& \theta_0 + \sum_l \bigg( l\phi_{l}^{(2,1)} + l\Big(\theta_{l}^{(1,1)} + \big(\phi_{l+1}^{(1,1)} - \theta_{l+1}^{(1,1)}\big)\Big) + (l+2)\theta_{l}^{(2,2)}
  + l\big(\phi_{l+1}^{(2,2)} - \theta_{l+1}^{(2,2)}\big) + (l+2)\Delta_l\\
  &+ l\big(\theta_{l-1,1}^{(1,2)} - \Delta_{l-1}\big) + (l+2)\big(\theta_{l+1,2}^{(1,2)} - \Delta_{l+1}\big)+l\Big(\phi_{l}^{(1,2)} - \theta_{l,1}^{(1,2)} - \theta_{1,2}^{(1,2)} + \Delta_l\Big) \bigg).
\end{align*}
Hence, 
\begin{align}\label{chi of fiber}
\chi(X)=\chi(X^T)=\sum_{\Omega}\binom{\phi_l^{(1,1)}}{\theta_l^{(1,1)}}\binom{\phi_l^{(2,2)}}{\theta_l^{(2,2)}}\binom{\phi_l^{(1,2)}}{\Delta_l}\binom{\phi_l^{(1,2)}-\Delta_l}{\theta_{l,1}^{(1,2)}-\Delta_l}\binom{\phi_l^{(1,2)}-\theta_{l,1}^{(1,2)}}{\theta_{l,2}^{(1,2)}-\Delta_l}\chi\big(\mathrm{Gr}_{\theta_0}(L)\big)
\end{align}
For $l\geq 1$, we have the following notation: 
\begin{equation}
    \begin{aligned}
    g_l^{(1,1)}(y_k,z_k):=&f_{k,l-1}(z_k)y_k^{l-1}+f_{k,l}(z_k)y_k^{l};\\
    g_l^{(1,2)}(y_k,z_k):=&f_{k,l}(z_k)y_k^{l+2}+f_{k,l+1}(z_k)y_k^{l+1}+f_{k,l-1}(z_k)y_k^{l+1}+f_{k,l}(z_k)y_k^l\\
    =&f_{k,l}(z_k)y_k^{l}(1+z_ky_k+y_k^2);\\
    g_l^{(2,1)}(y_k,z_k):=&f_{k,l}(z_k)y_k^{l};\\
    g_l^{(2,2)}(y_k,z_k):=&f_{k,l}(z_k)y_k^{l+2}+f_{k,l+1}(z_k)y_k^{l+1}.
\end{aligned}
\end{equation}
Substituting \eqref{chi of fiber} into \eqref{expression of F over fiber} and performing the summation with respect to $e_k,\mathbf{t}_k$, we obtain
\begin{equation}\label{rewrite of F}
 \begin{aligned}
     F_{\mathcal{M}}(y_1,\ldots,y_n)=\sum_{\hat{\mathbf{e}},\hat{\mathbf{t}},r,s,\phi}\chi\big(Z_{\hat{\mathbf{e}},\hat{\mathbf{t}};\phi,r,s}(M)\big)y_k^{2r_2} (z_ky_k+y_k^2)^{\phi^{(2,2)}_{0}}\prod_{\substack{1\leq l,\\1\leq u,v\leq 2}}\big(g_{l}^{(u,v)}(y_k,z_k)\big)^{\phi_l^{(u,v)}}\\(1+z_ky_k+y_k^2)^{s_2-r_2-r_1+\sum_{l,v}(l-1)\phi_l^{(1,v)}+l\phi_l^{(2,v)}}\prod_{i\neq k}f_{i,l}^{t_{i,l}}(z_i)y_i^{e_i}
 \end{aligned}   
\end{equation}
The remainder of the proof of Theorem \ref{mutation of F-polynomial} is based on the following observation:
\begin{align}\label{observation}
Z_{\hat{\mathbf{e}},\hat{\mathbf{t}};\phi,r,s}(M)=Z_{\hat{\mathbf{e}},\hat{\mathbf{t}};\overline{\phi},\overline{r},\overline{s}}(\overline{M}),
\end{align}
where $\overline{r},\overline{s},\overline{\phi}$ are given by
\begin{gather}
   \overline{r}_1=s_{1},\quad \overline{r}_{2}=\sum_{i}[b_{i,k}]_+e_i+\Check{\beta}_+(k)- s_1-s_2; \label{mutation r} \\ 
     \overline{s}_1=r_{1}, \quad \overline{s}_{2}=\sum_{i}[-b_{i,k}]_+e_i+\Check{\beta}'_+(k)-r_1-r_2; \label{mutation type s} \\ 
     \overline{\phi}_l^{(1,1)}=\phi_{l-1}^{(2,2)},\quad \overline{\phi}_{l}^{(1,2)}=\phi_l^{(2,1)},\quad\overline{\phi}_l^{(2,1)}=\phi_l^{(1,2)},\quad \overline{\phi}_l^{(2,2)}=\phi_{l+1}^{(1,1)}. \label{mutation phi}
\end{gather}

In view of the symmetry between $\mathcal{M}$ and $\overline{\mathcal{M}}$, to prove \eqref{observation}, it suffices to show that every $N=(N(i))_{i\neq k}\in Z_{\hat{\mathbf{e}},\hat{\mathbf{t}};\phi,r,s}(M)$ belongs to $Z_{\hat{\mathbf{e}},\hat{\mathbf{t}};\phi,r,s}(\overline{M})$.

Suppose $N\in Z_{\hat{\mathbf{e}},\hat{\mathbf{t}};\phi,r,s}(M)$. Recall from \eqref{aftermutation} that $\Img\beta_k=\ker \overline{\alpha}_k$ and $\Img\overline{\beta}_k=\ker \alpha_k$, which implies that: 
\begin{equation}\label{iso eq in mutation of F}
    \begin{aligned}
    \text{the map $\overline{\alpha}_k$ induces an isomorphism $N_{k,\out}/(\Img\beta_k\cap N_{k,\out})\to\overline{\alpha}_k(N_k,\out)$,} \\
    \overline{\beta}_k^{-1}(N_{k,\In})=\ker \overline{\beta}_k\oplus (\Img\overline{\beta}_k\cap N_{k,\In})=\frac{\ker \gamma_k}{\Img \beta_k}\oplus V_k\oplus (\ker\alpha_k\cap N_{k,\In}).
\end{aligned}
\end{equation}
Using the exact sequences
$$0\xrightarrow{}\ker \beta_k\xrightarrow{}\beta_k^{-1}(N_{k,\out})\xrightarrow{}\Img\beta_k\cap N_{k,\out}\xrightarrow{}0,$$
 $$0\xrightarrow{}\ker\alpha_k\cap N_{k,\In}\xrightarrow{}N_{k,\In}\xrightarrow{}\alpha_k(N_{k,\In})\xrightarrow{}0,$$
we deduce that the types of  $\overline{\alpha}_k(N_k,\out)$ and $\overline{\beta}_k^{-1}(N_{k,\In})$ are exactly $\overline{r}$ and $\overline{s}$ in (\ref{mutation r}) and (\ref{mutation type s}), respectively.
 
Recalling the construction of $\overline{\beta}_k^{-1}(N_{k,\In})_1$, $\overline{\alpha}_k(N_{k,\out})_1$  (see \eqref{construction of beta-1_1}), and the definition of $N_{k,\In}$, $N_{k,\out}$ in \eqref{defn of M_in and M_out}, according to \eqref{iso eq in mutation of F}, we obtain:
\begin{itemize}
    \item the map $\overline{\alpha}_k$ induces an isomorphism  $\overline{\alpha}_k(N_k,\out)_1'\to \overline{\alpha}_k(N_k,\out)_1$ where 
    \begin{gather*}
   \overline{\alpha}_k(N_k,\out)^\prime_1:=\left\{\sum_{tb=k} b^\star\otimes y_{b}+\big(\varepsilon_k(N_{k,\out})+\Img \beta_k\cap N_{k,\out}\big)\bigg|\sum_{tb=k} \varepsilon_kb^\star\otimes y_{b}\in \Img\beta_k\cap N_{k,\out
 } \right\};
\end{gather*}
\item $\overline{\beta}_k^{-1}(N_{k,\In})_1=\left\{\sum_{ha=k}\varepsilon_k a\otimes x_{a}+\varepsilon_k(\ker\alpha_k\cap N_{k,\In})\bigg| \sum_{ha=k}\varepsilon_k a\otimes x_{a}\in \ker\alpha_k\cap N_{k,\In}\right\}.$
\end{itemize}
Then we define two $K$-linear maps $ \varphi:\overline{\beta}_k^{-1}(N_{k,\In})_1\to\alpha_k(N_{k,\In})$, $\varphi':\overline{\alpha}_k(N_k,\out)_1'\to \beta_k^{-1}(N_{k,\out})_1$ as follows:
\begin{align*}
   \varphi:\varepsilon_k a\otimes x+\varepsilon_k(\ker\alpha_k\cap N_{k,\In})_1&\longmapsto N(a)x+\varepsilon_k(\alpha_k(N_{k,\In})), \\
  \varphi': b^\star\otimes y+(\varepsilon_k(N_{k,\out})+\Img \beta_k\cap N_{k,\out})&\longmapsto \beta_k^{-1}(\varepsilon_kb^\star\otimes y +\varepsilon_k(\Img \beta_k\cap N_{k,\out})),
\end{align*}
here with some abuse of notation, we denote by  $\beta_k^{-1}$ the isomorphism from $(\Img(\beta_k)\cap N_{k,\out})_1$ to $ \beta_k^{-1}(N_{k,\out})_1$ induced by $\beta_k$. Note that $\varphi$ and $\varphi'$ are well-defined. Moreover, one can check that they are bijective.  Now we claim that the following diagram commutes:
\begin{equation}\label{iso commutive diagram}
     \begin{tikzcd}
 \overline{\beta}_k^{-1}(N_{k,\In})_1\arrow[r, "\partial_{\varepsilon_k}\widetilde{\s}"] \arrow[d,"\varphi"] & \overline{\alpha}_k(N_{k,\out})_1 \arrow[d,"\varphi^\prime\overline{\alpha}_k^{-1}"] \arrow[r,"\widetilde{\iota}"] &  \overline{\beta}_k^{-1}(N_{k,\In})_1\arrow[d,"\varphi"] \\
\alpha_k(N_{k,\In})_1 \arrow[r,"\iota"]                                     & \beta_k^{-1}(N_{k,\out})_1 \arrow[r,"\partial_{\varepsilon_k}\s"]           &   \alpha_k(N_{k,\In})_1.         
\end{tikzcd} 
\end{equation}
Remembering \eqref{mutation of s} and \eqref{martix of partial_{varepsilon_k}s}, we can express $\partial_{\varepsilon_k}\widetilde{\s}$ as the matrix $$\overline{\alpha}_k {\begin{pmatrix}
    0&ba \\
    0& 0\\
\end{pmatrix}}_{a,b}\overline{\beta}_k.$$
Then the induced $K$-linear map  $\partial_{\varepsilon_k}\widetilde{\s}:\overline{\beta}_k^{-1}(N_{k,\In})_1\to \overline{\alpha}_k(N_{k,\out})_1$ acts on elements as follows:
$$\varepsilon_k a\otimes x+\varepsilon_k(\ker\alpha_k\cap N_{k,\In}) \mapsto\overline{\alpha}_k\big(\sum_{b}b^\star \otimes N(ba)x+(\varepsilon_k(N_{k,\out})+\Img \beta_k\cap N_{k,\out})\big),$$ 
which equals  $\overline{\alpha}_k\varphi^{\prime-1}\iota \varphi$.
From the definition of $\overline{\alpha}_k$ in (\ref{defn of mutation of alpha and beta}):  
\begin{align*}
   \overline{\alpha}_k=(-\pi \rho,-\gamma,0,0)^T:M_{k,\out}\to \frac{\ker \gamma_k}{\Img \beta_k}\oplus \Img\gamma_k\oplus \frac{\ker \alpha_k}{\Img\gamma_k}\oplus V_k, 
\end{align*}
 we observe that the composition $\widetilde{\iota}\circ\overline{\alpha}_k:\overline{\alpha}_k(N_{k,\out})'_1\to\overline{\beta}_k^{-1}(N_{k,\In})_1 $ has the action: $$b^\star \otimes y+\big(\varepsilon_k(N_{k,\out})+\Img \beta_k\cap N_{k,\out}\big)\mapsto\gamma_{k}(b^\star \otimes y)+\varepsilon_k(\ker\alpha_k\cap N_{k,\In})=\sum_a \varepsilon_k a\otimes N(\partial_{[b\varepsilon_ka]}\s b )y+\varepsilon_k(\ker\alpha_k\cap N_{k,\In}).$$ Then a direct verification shows that: $\varphi \circ \widetilde{\iota}\circ\overline{\alpha}_k=\partial_{\varepsilon_k}\s\circ \varphi'$. This completes the proof of our claim.

Having established this claim, we conclude from the construction of  the $K$-linear spaces $W_l^{(u,v)}$ and $\overline{W}l^{(u,v)}$ in \eqref{construction of space W} that:
\begin{align*}
     \overline{W}_l^{(1,1)}\cong W_{l-1}^{(2,2)},\quad \overline{W}_{l}^{(1,2)}\cong W_l^{(2,1)},\quad\overline{W}_l^{(2,1)}\cong W_l^{(1,2)},\quad \overline{W}_l^{(2,2)}\cong W_{l+1}^{(1,1)},
\end{align*}
which imply the equalities in \eqref{mutation phi}.

To complete the proof of \eqref{observation}, it remains to show that $\overline{\beta}_k\overline{\alpha}_k(N_{k,\out})\subseteq N_{k,\In}$ and the Jordan type of $N_i$ is $\mathbf{t}_i$ for $i\neq k$. As an immediate consequence of  \eqref{defn of mutation of alpha and beta}, we get $\overline{\beta}_k\overline{\alpha}_k=\gamma_k$. In view of \eqref{defn of gamma}, each of the components of the map $\gamma_k$ is a linear combination of the compositions of maps of the kind $M(c)$ or $M(b\varepsilon_k^la)$ (where $a,b\in Q^\circ_1$, $c\in A_1$ are such that $ha=tb=k$, and $c$ is not incident to $k$, and $0\leq l\leq d_k$); thus, the defining conditions \eqref{subcondition 1} and \eqref{subcondition 2} imply the desired inclusion $\gamma_k(N_{k,\out})\subseteq N_{k,\In}$. In view of \eqref{mutation of s}, $\partial_{\varepsilon_i}\widetilde{\s}=\partial_{\varepsilon_i}\s$ for all $i\neq k$. Thus the Jordan type of $N_i$ is still $\mathbf{t}_i$.

The rest of the proof of (\ref{F-mutation}) is straightforward: use (\ref{rewrite of F}) and (\ref{observation}) for rewriting the right-hand side in (\ref{F-mutation}),  then substitute for \(y^{\prime}_{1},\ldots,y^{\prime}_{n}\) (resp. for \({\overline{r}}\), \({\overline{s}}\), and $\overline{\phi}$) the expressions given by (\ref{mutation for y}) (resp. by (\ref{mutation r}), (\ref{mutation type s}), (\ref{mutation phi})), simplify the resulting expression, and use (\ref{rewrite of F}) again to see that it is equal to the left-hand-side of (\ref{F-mutation}). The detailed computation proceeds as follows:
\begin{align*}
    &\left(1+z_ky_k'+(y_k')^2\right)^{\check{\beta}'_{+}(k)} F_{\overline{\mathcal{M}}}\left(y'_1,\cdots,y'_n,\mathbf{z}\right)\\
     =&\left(1+z_ky_k'+(y_k')^2\right)^{-\check{\beta}_{-}(k)} \sum_{\mathbf{e}',\mathbf{t}',\overline{r},\overline{s},\overline{\phi}}\chi\left(Z_{\hat{\mathbf{e}},\hat{\mathbf{t}};\overline{\phi},\overline{r},\overline{s}}(M)\right)\left(y'_k\right)^{2\overline{r}_2} \left(z_ky'_k+(y'_k)^2\right)^{\overline{\phi}^{(2,2)}_{0}}\prod_{\substack{1\leq l,\\1\leq u,v\leq 2}}\left(g_{l}^{(u,v)}\left(y'_k,z_k\right)\right)^{\overline{\phi}_l^{(u,v)}}\\
     &\left(1+z_ky_k'+(y'_k)^2\right)^{\overline{s}_2-\overline{r}_2-\overline{r}_1+\sum_{l,v}(l-1)\overline{\phi}_l^{(1,v)}+l\overline{\phi}_l^{(2,v)}}\prod_{i\neq k}f_{i,l}^{t_{i,l}}\left(z_i\right)\left(y'_i\right)^{e_i}\\
     =&\left(1+z_ky_k+y_k^2\right)^{-\check{\beta}_{-}(k)} y_k^{2\check{\beta}_{-}(k)}\sum_{\mathbf{e}',\mathbf{t}',r,s,\phi}\chi\left(Z_{\hat{\mathbf{e}},\hat{\mathbf{t}};\phi,r,s}(M)\right)\left(y_k\right)^{-2\left(\sum_{i}[b_{i,k}]_+e_i+\Check{\beta}_+(k)- s_1-s_2\right)} \left(z_ky_k+1\right)^{\phi^{(1,1)}_{1}}y_k^{-2\phi^{(1,1)}_{1}}\\
     &\prod_{0\leq l}\left(g_l^{(2,2)}\left(y_k,z_k\right)\right)^{\phi_l^{(2,2)}}y_k^{-2(l+1)\phi_l^{(2,2)}}
     \prod_{1\leq l}\left(g_{l}^{(2,1)}\left(y_k,z_k\right)\right)^{\phi_l^{(2,1)}}\left(1+z_ky_k+y_k^2\right)^{\phi_l^{(2,1)}}y_k^{-2(l+1)\phi_l^{(2,1)}}\\
     &\prod_{1\leq l}\left(g_{l}^{(1,2)}\left(y_k,z_k\right)\right)^{\phi_l^{(1,2)}}\left(1+z_ky_k+y_k^2\right)^{-\phi_l^{(1,2)}}y_k^{-2l\phi_l^{(1,2)}}
     \prod_{2\leq l}\left(g_l^{(1,1)}\left(y_k,z_k\right)\right)^{\phi_l^{(1,1)}}y_k^{-2l\phi_l^{(1,1)}}\\
     &\left(1+z_ky_k+y_k^2\right)^{\check{\beta}_-(k)-\check{\beta}_+(k)+\sum_{i\neq k}b_{k,i}e_i+s_2-r_1-r_2+\sum_{l}(l-1)\left(\phi_{l-1}^{(2,2)}+\phi_l^{(2,1)}\right)+l\left(\phi_l^{(1,2)}+\phi_{l+1}^{(1,1)}\right)}\\ 
     &y_k^{-2\left(\check{\beta}_-(k)-\check{\beta}_+(k)+\sum_{i\neq k}b_{k,i}e_i+s_2-r_1-r_2+\sum_{l}(l-1)\left(\phi_{l-1}^{(2,2)}+\phi_l^{(2,1)}\right)+l\left(\phi_l^{(1,2)}+\phi_{l+1}^{(1,1)}\right)\right)}\\
     &\prod_{i\neq k} f_{i,l}^{t_{i,l}}\left(z_i\right)\left(y_iy_{k}^{2\left[b_{ki}\right]_{+}}\left(1+z_ky_k+y_k^2\right)^{-b_{ki}}\right)^{e_i}\\
=&\sum_{\mathbf{e}',\mathbf{t}',r,s,\phi}\chi\left(Z_{\hat{\mathbf{e}},\hat{\mathbf{t}};\phi,r,s}(M)\right) \left(z_ky_k+y_k^2\right)^{\phi^{(2,2)}_{0}}\prod_{\substack{1\leq l,\\1\leq u,v\leq 2}}\left(g_{l}^{(u,v)}\left(y_k,z_k\right)\right)^{\phi_l^{(u,v)}}\prod_{i\neq k}f_{i,l}^{t_{i,l}}\left(z_i\right)y_i^{e_i}\\
     &y_k^{2\left(s_1+r_1+r_2-\left(\sum_l \left(2l-1\right)\phi_l^{(1,1)}+2l\phi_l^{(1,2)} +2l\phi_l^{(2,1)}+\left(2l+1\right)\phi_l^{(2,2)}\right)\right)}\left(1+z_ky_k+y_k^2\right)^{-\check{\beta}_+(k)+s_2-r_1-r_2+\sum_{l,v}(l-1)\phi_l^{(1,v)}+l\phi_l^{(2,v)}}.
\end{align*}
The final equality matches \eqref{rewrite of F} through the following verification:
\begin{equation*}
\setlength{\abovedisplayskip}{6pt}
    s_1=\sum_{l,u,v} l\phi_l^{(u,v)},\quad r_1=\sum_l (l-1)\phi_l^{(1,1)}+l\phi_l^{(1,2)}+l\phi_l^{(2,1)}+(l+1)\phi_l^{(2,2)}.  \qedhere
\end{equation*}
\end{proof}

\section{Application to generalized cluster algebras}\label{section: Application to generalized cluster algebras} 
\subsection{Generalized cluster algebras}
Recall that a matrix $B=\left(b_{i,j}\right)_{i, j=1}^{n}$ is said to be \textbf{skew-symmetrizable} if there is an $n$-tuple of positive integers $\mathbf{d}=\left(d_{1}, \ldots, d_{n}\right)$ such that $d_{i} b_{i, j}=-d_{j} b_{j ,i}$.

We start by fixing a semifield $\mathbb{P}$, whose addition is denoted by $\oplus$. Let $\mathbb{Z P}$ be the group ring of $\mathbb{P}$, and let $\mathbb{Q P}$ be the field of fractions of $\mathbb{Z} \mathbb{P}$. Let $w_{1}, \ldots, w_{n}$ be any algebraically independent variables, and let $\mathcal{F}=\mathbb{Q P}(w)$ be the field of rational functions in $w=\left(w_{1}, \ldots, w_{n}\right)$ with coefficients in $\mathbb{Q P}$.

\begin{defn}
\begin{enumerate}
    \item[(i)] A \textbf{(labeled) seed} in $\mathbb{P}$ is a triple $(\mathbf{x}, \mathbf{y}, B)$ such that
\begin{itemize}
    \item $B$ is a skew-symmetrizable matrix, called an \textbf{exchange matrix};
    \item $\mathbf{x}=\left(x_{1}, \ldots, x_{n}\right)$ is an $n$-tuple of elements in $\mathcal{F}$, called the \textbf{cluster}, and   $x_{1}, \ldots, x_{n}$ are called \textbf{cluster variables} or \textbf{$x$-variables};
    \item $\mathbf{y}=\left(y_{1}, \ldots, y_{n}\right)$ is an $n$-tuple of elements in $\mathbb{P}$, where $y_{1}, \ldots, y_{n}$ are called \textbf{coefficients} or \textbf{$y$-variables}.
\end{itemize}
\item [(ii)] A \textbf{(labeled) mutation pair} in $\mathbb{P}$ is a pair $(\mathbf{d}, \mathbf{z})$, where
\begin{itemize}
    \item $\mathbf{d}=\left(d_{1}, \ldots, d_{n}\right)$ is an $n$-tuple of positive integers, and we call these integers the \textbf{mutation degrees};
    \item $\mathbf{z}=\left(z_{i, s}\right)_{i=1, \ldots, n ; s=1, \ldots, d_{i}-1}$ is a family of elements in $\mathbb{P}$ satisfying the reciprocity condition
\begin{equation}\label{reciprocity condition}
    z_{i, s}=z_{i, d_{i}-s} \quad \text{for } 1\leq i\leq n \text{ and } s=1, \ldots, d_{i}-1.
\end{equation}
We call them the \textbf{frozen coefficients}, since they are not \enquote{mutated}, or simply the \textbf{$z$-variables}. We additionally set
\begin{equation*}
    z_{i, 0}=z_{i, d_{i}}=1
\end{equation*}
for $i=1,\ldots,n$.
\end{itemize}
\end{enumerate}
\end{defn}
\begin{defn}
 Let $(\mathbf{d}, \mathbf{z})$ be a mutation pair. For any seed $(\mathbf{x}, \mathbf{y}, B)$ in $\mathbb{P}$ and $k\in \{1, \ldots, n\}$, the \textbf{$(\mathbf{d}, \mathbf{z})$-mutation} of $(\mathbf{x}, \mathbf{y}, B)$ at $k$ is another seed $\left(\mathbf{x}^{\prime}, \mathbf{y}^{\prime}, B^{\prime}\right)=$ $\mu_{k}(\mathbf{x}, \mathbf{y}, B)$ in $\mathbb{P}$ defined by the following rule:
 \begin{align}
 & x_{i}^{\prime}= \begin{cases}\label{mutation of x}
x_{k}^{-1}\left(\prod_{j=1}^{n} x_{j}^{\left[- b_{j, k}\right]_{+}}\right)^{d_{k}} \frac{\sum_{s=0}^{d_{k}} z_{k, s} \hat{y}_{k}^{s}}{\bigoplus_{s=0}^{d_{k}} z_{k, s} y_{k}^{s}} & i=k; \\
 x_i & i \neq k,\end{cases}\\
 & y_{i}^{\prime}= \begin{cases}\label{mutation of y} y_{k}^{-1} & i=k; \\
y_{i}\left(y_{k}^{\left[b_{k, i}\right]_{+}}\right)^{d_{k}}\left(\bigoplus_{s=0}^{d_{k}} z_{k, s} y_{k}^{ s}\right)^{-b_{k ,i}} & i \neq k,\end{cases} \\
 & \text{where } \hat{y}_{k}=y_{k} \prod_{j=1}^{n} x_{j}^{b_{j ,k}},\\ \label{defn of hat-y}
     & b_{i, j}^{\prime}= \begin{cases}-b_{i ,j} & i=k \text { or } j=k; \\
b_{i, j}+d_{k}\left(\left[-b_{i, k}\right]_{+} b_{k, j}+b_{i, k}\left[b_{k, j}\right]_{+}\right) & i, j \neq k.\end{cases} 
 \end{align}
 When the data $(\mathbf{d}, \mathbf{z})$ is clearly assumed, we may drop the prefix and simply call it the \textbf{generalized mutation}. 
\end{defn}
\begin{pro}[{\cite[Proposition 3.2]{nakanishi2015}}]
    The $(\mathbf{d}, \mathbf{z})$-mutation $\mu_{k}$ is  an involution, that is, $\mu_{k}\left(\mu_{k}(\mathbf{x}, \mathbf{y}, B)\right)=(\mathbf{x}, \mathbf{y}, B)$.
\end{pro}
\begin{pro}[{\cite[Proposition 2.5]{Nakanishi_2015}}]
    Under the mutation $\mu_k$, the  $\hat{y}$-variables mutate in the same way as the $y$-variables, namely, 
\begin{equation}\label{mutation of hat-y}
    \hat{y}_{i}^{\prime}= \begin{cases}\hat{y}_{k}^{-1} & i=k \\ \hat{y}_{i}\left(\hat{y}_{k}^{\left[ b_{k, i}\right]_{+}}\right)^{d_{k}}\left(\sum_{s=0}^{d_{k}} z_{k, s} \hat{y}_{k}^{s}\right)^{-b_{k, i}} & i \neq k\end{cases}
\end{equation}
\end{pro}

Let $\mathbb{T}_n$ be the $n$-regular tree with edges labeled by the set $\{1,\ldots,n\}$ so that the $n$ edges emanating from each vertex receive different labels. We write $t \overset{k}{\rule[3pt]{6mm}{0.05em}} t'$ if the vertices $t$ and $t'$ of $\mathbb{T}_n$ are connected by an edge labeled with $k$.
\begin{defn}
    \begin{enumerate}
        \item [(i)] A \textbf{$(\mathbf{d},\mathbf{z})$-cluster pattern} $\Sigma$ is an assignment of a seed $(B_t,\mathbf{x}_t,\mathbf{y}_t)$ to every vertex $t$ of the $n$-regular tree $\mathbb{T}_n$ such that $$(\mathbf{x}_{t'},\mathbf{y}_{t'},B_{t'})=\mu_k(\mathbf{x}_t,\mathbf{y}_t,B_t)$$ for any edge $t\overset{k}{\rule[3pt]{6mm}{0.05em}}t'$, where $\mu_k$ is the $(d,\mathbf{z})$-mutation at $k$. For the seed  $(B_t,\mathbf{x}_t,\mathbf{y}_t)$, we always write 
\begin{align*}
    \mathbf{x}_t=(x_{1;t},\ldots,x_{n;t}),\quad  \mathbf{y}_t=(y_{1;t},\ldots,y_{n;t}),\quad  B_t=(b_{i,j}^t).
\end{align*}
        \item [(ii)] The \textbf{$(\mathbf{d},\mathbf{z})$-cluster algebra} $\mathcal{A}(\mathbf{x},\mathbf{y},B)=\mathcal{A}(\mathbf{x}, \mathbf{y}, B;\mathbf{d},\mathbf{z})$ (also called a \textbf{generalized cluster algebra}) is the $\mathbb{ZP}$-subalgebra of $\mathcal{F}$ generated by all cluster variables appearing in $\Sigma$, more precisely
        $$\mathcal{A}(\mathbf{x}, \mathbf{y}, B;\mathbf{d},\mathbf{z})=\mathbb{ZP}[x_{i;t}:t\in \mathbb{T},1\leq i\leq n]\subseteq \mathcal{F}.$$
    \end{enumerate}
\end{defn}
\begin{rmk}
    If $\mathbf{d}=(1, \ldots, 1)$ and $\mathbf{z}$ is empty, the generalized cluster algebra reduces to the classical cluster algebra by Fomin and  Zelevinsky \cite{fomin2001}. 
\end{rmk}
As with classical cluster algebras, generalized cluster algebras also exhibit the Laurent phenomenon.
\begin{thm}[{\cite[Theorem 2.8]{Nakanishi_2015}}]\label{Laurent}
Let $\mathcal{A}=\mathcal{A}(\mathbf{x},\mathbf{y},B;\mathbf{d},\mathbf{z})$ be a generalized cluster algebra. Then any element of $\mathcal{A}$ can be expressed as a Laurent polynomial in the cluster $\mathbf{x}$ with coefficients in $\mathbb{Z}\mathbb{P}$.
\end{thm}
Recall that for every finite family of indeterminates $u_1,\ldots,u_n$, one can associate two semifields: the universal semifield $\mathbb{Q}_{\text{sf}}(u_1,\ldots,u_n)$ and the tropical semifield $\text{Trop}(u_1,\ldots,u_n)$ (cf. \cite[Definitions 2.1, 2.2]{fomin2006cluster}) where $\mathbb{Q}_{\text{sf}}(u_1,\ldots,u_n)$ consists of all rational functions in $u_1,\ldots,u_n$ that can be written as subtraction-free rational expressions and $\operatorname{Trop}(u_1,\ldots,u_n)$ is the multiplicative group of Laurent monomials $\prod_{j=1}^n u_j^{l_j}$ with the addition $\oplus$ defined by 
$$\prod_{j}u_j^{l_j}\oplus\prod_{j}u_j^{s_j}=\prod_{j}u_j^{\min(l_j,s_j)}.$$

Now, for each initial $(\mathbf{x},\mathbf{y},B)$, we
\begin{itemize}
    \item replace the $x$-, $y$-, and $z$-variables by formal indeterminates (denoted by the same symbols by abuse of notation);
    \item replace the semifield $\mathbb{P}$ by the tropical semifield $\operatorname{Trop}(\mathbf{y},\mathbf{z})$;
    \item replace $\mathcal{F}$ by $\mathbb{Q}(\mathbf{x}, \mathbf{y},\mathbf{z})$ and choose to perform all calculations here.
\end{itemize}
Since no subtraction occurs in the recursions \eqref{mutation of x}, we obtain in this way \textbf{$X$-functions} $X_{i;t} \in \mathbb{Q}_{\mathrm{sf}}(\mathbf{x}, \mathbf{y},\mathbf{z})$. 
Alternatively, performing the $y$-mutations \eqref{mutation of y} inside $\mathbb{Q}_{\mathrm{sf}}(y,\mathbf{z})$, we obtain \textbf{$Y$-functions} $Y_{j;t} \in \mathbb{Q}_{\mathrm{sf}}(\mathbf{\mathbf{y},\mathbf{z}})$. 
By the universality of the semifield $\mathbb{Q}_{\mathrm{sf}}(\mathbf{y},\mathbf{z})$, we can recover the original coefficient $y_{j;t}$ through the specialization $Y_{j;t}\big|\mathbb{P}$. Taking this specialization where $\mathbb{P} = \operatorname{Trop}(\mathbf{y},\mathbf{z})$, we get monomials $Y_{j;t}\big|_{\operatorname{Trop}(\mathbf{y},\mathbf{z})} = \prod_{i=1}^n y_i^{c_{i,j}^t},$ where we write $C_t$ for the resulting matrix whose columns $\mathbf{c}_{j;t} \in \mathbb{Z}^n$ are called \textbf{$c$-vectors}. 
Note that the $c$-vectors depend only on the initial exchange matrix $B$ and the mutation degrees $d_i$,  but not on the choice of initial cluster $\mathbf{x}$.
\begin{pro}[{\cite[Proposition 3.3]{Nakanishi_2015}}]
    Each $X$-function $X_{i;t}$ is contained in $\mathbb{Z}[\mathbf{x}^{\pm1},\mathbf{y},\mathbf{z}]$.
\end{pro}
Analogous to classical cluster algebras, the structure of generalized cluster algebras is controlled by a family of integer vectors called \textbf{$\mathbf{g}$-vectors}, and a family of integer polynomials called \textbf{$F$-polynomials}. Now, according to \cite[Proposition 3.15]{Nakanishi_2015}, the vectors $\mathbf{g}_{l;t}=\mathbf{g}_{l;t}^{B;t_0}$ can be defined by the initial condition 
\begin{equation}\label{defn of g in gca}
    \mathbf{g}_{l;t_0}=\mathbf{e}_l,\quad \text{for } 0\leq l\leq n;
\end{equation}
together with the recurrence relations 
\begin{align}
    \mathbf{g}_{l;t'}=\begin{cases}
    \mathbf{g}_{l;t} \quad\text{ $l\neq k$};\\
    -\mathbf{g}_{k;t}+\sum_{i=1}^n [-b_{i,k}^td_k]_+\mathbf{g}_{i;t}-\sum_{i=1}^n \mathbf{b}_i[-c_{i,k}^td_k]_+ \quad\text{ $l=k$},
    \end{cases}
\end{align}
for every edge $t\overset{k}{\rule[3pt]{6mm}{0.05em}}t'$ in $\mathbb{T}_n$.
Here $\mathbf{e}_1,\ldots,\mathbf{e}_n$ are the unit vectors in $\mathbb{Z}^n$, $\mathbf{b}_1,\ldots,\mathbf{b}_n$ are the columns of $B$.  

Similarly, by  \cite[Proposition 3.12]{Nakanishi_2015}, the polynomial $F_{l;t}=F_{l;t}^{B,t_0}(\mathbf{y},\mathbf{z})$  can be defined by the initial condition 
\begin{equation}
    F_{l;t_0}=1,\quad \text{for }  0\leq l\leq n;
\end{equation}
together with the recurrence relations  
\begin{align}\label{mutation of F}
    F_{l;t'}=\begin{cases}
    F_{l;t} \quad\text{ $l\neq k$};\\
    F_{k;t}^{-1}\left(\prod_{i=1}^n y_i^{[-c_{i,k}^t]_+}F_{i;t}^{[-b_{i,k}^t]_+}\right)^{d_k}\sum_{s=0}^{d_k}z_{k,s}\left(\prod_{i=1}^n u_i^{-c_{i,k}^t}F_{i;t}^{-b_{i,k}^t}\right)^s\quad\text{ $l=k$},
    \end{cases}
\end{align} 
for each edge $t\overset{k}{\rule[3pt]{6mm}{0.05em}}t'  \in \mathbb{T}_n$. Since \eqref{mutation of F} does not involve subtraction, every $F$-polynomial $F_{l;t}^{B,t_0}$ belongs to $\mathbb{Q}_{\mathrm{sf}}(\mathbf{\mathbf{y},\mathbf{z}})$. Note that every subtraction-free expression $F(\mathbf{y},\mathbf{z})$ can be evaluated in arbitrary semifield $\mathbb{P}$ for arbitrary  elements $\mathbf{y},\mathbf{z}$.

Now we will see the $\mathbf{g}$-vectors and $F$-polynomials play a crucial role in the theory of generalized cluster algebras. For each $l$ and $t$ as above, they determine a cluster variable $x_{l;t}\in\mathcal{F}=\mathbb{QP}(\mathbf{x})$ through (see \cite[Theorem 3.23]{Nakanishi_2015})
\begin{equation}\label{express x in term of F and g}
    x_{l;t}=\prod_{i=1}^nx_i^{g(i)}\frac{F_{l;t}^{B,t_0}|_{\mathcal{F}}(\hat{\mathbf{y}},\mathbf{z})}{F_{l;t}^{B,t_0}|_{\mathbb{P}}(\mathbf{y},\mathbf{z})}.
\end{equation}
where $(g(1),\cdots,g(n))=\mathbf{g}_{l;t}^{B,t_0}$.

We now fix $l$ and $t$, and discuss the dependence of $\mathbf{g}_{l;t}^{B,t_0}$ and $F_{l;t}^{B,t_0}$ on the initial vertex $t_0$ and the initial exchange matrix $B$. More precisely, for any $k\in[1,n]$ with $t_0\overset{k}{\rule[3pt]{6mm}{0.05em}}t_1$ and $B_1=\mu_k(B)$, we establish relations between the vectors $\mathbf{g}_{l,t}^{B,t_0}$ and $\mathbf{g}_{l,t}^{B_1,t_1}$, and the polynomials $F_{l,t}^{B,t_0}$ and $F_{l,t}^{B_1,t_1}$. To do this, we first denote by $\mathbf{h}_{l;t}^{B;t_0}=(h(1),\ldots,h(n))$ the integer vector given by
\begin{align}
    y_1^{h(1)}\cdots y_n^{h(n)}=F_{l;t}^{B,t_0}|_{\text{Trop}(\mathbf{y},\mathbf{z})}(y_i\xrightarrow{}y_i^{-1}\prod_{j\neq i}y_j^{d_j[-b_{j,i}]_+}).
\end{align}

\begin{thm}\label{mutation rule of g and F in GQP}
Suppose $t_0\overset{k}{\rule[3pt]{6mm}{0.05em}}t_1$, $B_1=\mu_k(B)$, and $\mathbf{y}'\in \mathbb{Q}_{\mathrm{sf}}(\mathbf{\mathbf{y},\mathbf{z}})$ is obtained from $\mathbf{y}$ by the mutation rule \eqref{mutation of hat-y}. Let $h(k)$ (resp. $h'(k)$) be the $k$-th component of $\mathbf{h}_{l;t}^{B;t_0}$ (resp. $\mathbf{h}_{l;t}^{B_1;t_1}$). Then the $\mathbf{g}$-vectors $\mathbf{g}_{l;t}^{B,t_0}=(g(1),\cdots,g(n))$ and $\mathbf{g}_{l;t}^{B_1,t_1}=(g'(1),\cdots,g'(n))$ are related as follows:
\begin{align}
    g'(i)=\begin{cases}\label{mutation rule of g in gca}
    -g(k) &\text{if $i=k$}\\
    g(i)+d_k[-b_{i,k}]_+g(k)-b_{i,k}h(k)& \text{if $i\neq k$}.
    \end{cases}
\end{align}
We also have 
\begin{equation}\label{relate of h in gca}
    d_kg(k)=h(k)-h'(k)
\end{equation}
and 
\begin{equation}\label{mutation of F in gca}
    (\sum_{s=0}^{d_k}z_{k,s}y_k^s)^{\frac{h(k)}{d_k}}F_{l;t}^{B,t_0}(y_1,\ldots, y_n)= (\sum_{s=0}^{d_k}z_{k,s}{y'_k}^s)^{\frac{h'(k)}{d_k}}F_{l;t}^{B_1,t_1}(y'_1,\ldots, y'_n).
\end{equation}
\end{thm}
\begin{proof}
This is proved using the technique in \cite[Proposition 6.8 and Proposition 6.9]{fomin2006cluster}.
\end{proof}
\subsection{Upper generalized cluster algebras}

Since the generalized cluster algebra $\mathcal{A}(\mathbf{x},\mathbf{y},B)$ is generated by cluster variables from the seeds mutation equivalent to the initial seed $(\mathbf{x},\mathbf{y},B)$, Theorem \ref{Laurent} can be rephrased as 
$$\mathcal{A}(\mathbf{x},\mathbf{y},B)\subseteq \bigcap_{(\mathbf{x},\mathbf{y},B)\in \Sigma}\mathcal{L}_{\mathbf{x}}$$
where $\mathcal{L}_{\mathbf{x}}=\mathbb{ZP}[\mathbf{x}^{\pm1}].$
\begin{defn}
The \textbf{upper generalized cluster algebra} with seed $(\mathbf{x},\mathbf{y},B;\mathbf{d},\mathbf{z})$ is $$\overline{\mathcal{A}}(\mathbf{x},\mathbf{y},B)=\overline{\mathcal{A}}(\mathbf{x},\mathbf{y},B;\mathbf{d},\mathbf{z}):=\bigcap_{(\mathbf{x},\mathbf{y},B)\in \Sigma}\mathcal{L}_{\mathbf{x}}.$$
\end{defn}
In general, there may be  infinitely many seeds mutation-equivalent to $(\mathbf{x},\mathbf{y},B)$. So the above definition is not very useful to test the membership in an upper generalized cluster algebra. However, the following theorem reduces this to a finite verification.
\begin{thm}\label{one step mutation}
Let $\mathbf{x}_k$ be the adjacent cluster obtained from $\mathbf{x}$ by applying the mutation at $k$. Define the \textbf{upper bounds}:
$$\mathcal{U}(\mathbf{x},\mathbf{y},B):=\bigcap_{1\leq k\leq n}\mathcal{L}_{\mathbf{x}_k}.$$
Suppose that $B(Q)$ has full rank, and $(\mathbf{x}',\mathbf{y}',B')=\mu_k(\mathbf{x},\mathbf{y},B)$ for some $1\leq k\leq n$, then $\mathcal{U}(\mathbf{x}',\mathbf{y}',B')=\mathcal{U}(\mathbf{x},\mathbf{y},B)$. In particular, $\mathcal{U}(\mathbf{x},\mathbf{y},B)=\overline{\mathcal{A}}(\mathbf{x},\mathbf{y},B).$
\end{thm}
\begin{proof}
This is proved using the technique in  {\cite[Corollary 1.9]{berenstein2005cluster}}.
\end{proof}
\subsection{\texorpdfstring{$H$}{}-based QP-interpretation of \texorpdfstring{$\mathbf{g}$}{}-vectors and \texorpdfstring{$F$}{}-polynomials}
In the rest of this section, we assume that $B$ is skew-symmetric and $d_k \leq 2$ for all $1 \leq k \leq n$. Then $(B,\mathbf{d})$ is associated to a pair $(Q,\mathbf{d})=(Q(B),\mathbf{d})$ defined in Section \ref{section: Quivers and Complete Path Algebras} by setting
$$\text{ $Q_0=\{1\ldots,n\}$, and for any two vertices $i\neq j$, there are $[b_{i,j}]_+$ arrows from $j$ to $i$ in $Q^\circ$}.$$ Assume there is a potential $\s$ such that $(Q,\mathbf{d},\s)$ is a nondegenerate, locally free and Jacobian-finite $H$-based QP. To a representation $\mathcal{M}=(M,V)$ of  $(Q,\mathbf{d},\s)$, we associate both the $\check{g}$-vector $\check{\mathbf{g}}_{\mathcal{M}}=(\check{g}(1),\ldots,\check{g}(n))$ given by Definition \ref{defn of g_M}, and the $F$-polynomial $F_{\mathcal{M}}=F_M$ given by Definition \ref{defn of F-pplynomial}. 

Now we define a family of $H$-based QP-representations $\mathcal{M}_{l;t}^{B;t_0}$ as follows: let
\begin{align*}
    t_0\overset{k_1}{\rule[3pt]{6mm}{0.05em}}t_1\overset{k_2}{\rule[3pt]{6mm}{0.05em}}\cdots\overset{k_p}{\rule[3pt]{6mm}{0.05em}}t_p=t
\end{align*}
be the (unique) path joining $t_0$ and $t_1$ in $\mathbb{T}_n$. We set 
\begin{align*}
    (Q(t),\mathbf{d},\s(t))=\mu_{k_p}\cdots\mu_{k_1}(Q,\mathbf{d},\s),
\end{align*}
which is well-defined since $(Q,\mathbf{d},\s)$ is nondegenerate. Let $E_l^{-}(Q(t),\mathbf{d},\s(t))$ be the negative generalized simple representation of $(Q(t),\mathbf{d},\s(t))$ at $l$. We observe that $E_l^{-}(Q(t),\mathbf{d},\s(t))$ is in general position. Then we set
\begin{align}\label{construction of family M}
\mathcal{M}_{l;t}^{B;t_0}=\mu_{k_p}\cdots\mu_{k_1}\left(E_l^{-}(Q(t),\mathbf{d},\s(t))\right)
\end{align}
which is well-defined since $(Q,\mathbf{d},\s)$ is locally free and Jacobian-finite. In view of Theorem \ref{mutation of gqp is involution }, Proposition \ref{mutation of rep is involution}, and Proposition \ref{general is mutation invarints}, we can assume that it is a general representation of $(Q,\mathbf{d},\s)$.
\begin{thm}\label{thm: interpretation of g and F}
    We have 
    \begin{align}\label{interpretation of g and F}
\mathbf{g}_{l;t}^{B;t_0}=\check{\mathbf{g}}_{\mathcal{M}_{l;t}^{B;t_0}},\quad F_{l;t}^{B;t_0}=F_{\mathcal{M}_{l;t}^{B;t_0}}.
    \end{align}
\end{thm}
\begin{proof}
 Let
\begin{align}
    0\xrightarrow{}\mathcal{M}_{l;t}^{B;t_0} \xrightarrow{} I(\check{\beta}_{+,\mathcal{M}_{l;t}^{B;t_0}})\xrightarrow{}I(\check{\beta}_{-,\mathcal{M}_{l;t}^{B;t_0}})\xrightarrow{}\cdots
\end{align}
be the minimal injective presentation of $\mathcal{M}_{l;t}^{B;t_0}$ in $\mathcal{P}(Q,\mathbf{d},\s)$. 
 We denote by $\mathbf{h}_{\mathcal{M}_{l;t}^{B;t_0}}=(h(1),\ldots,h(n))$ the integer vector given by 
\begin{align}\label{defn of h in GQP}
    h(k)=-d_k\check{\beta}_{-,\mathcal{M}_{l;t}^{B;t_0}}(k),
\end{align}
We prove \eqref{interpretation of g and F} together with the equality
\begin{align}\label{interpretation of h}
    \mathbf{h}_{\mathcal{M}_{l;t}^{B;t_0}} = \mathbf{h}^{B;t_0}_{l;t}
\end{align}
by induction on the distance between $t_0$ and $t$ in the tree $\mathbb{T}_n$. When $t = t_0$, the statement is clear. Now assume that \eqref{interpretation of g and F} and \eqref{interpretation of h} hold for some $l$ and $t$, and let $t_0 \overset{k}{\rule[3pt]{6mm}{0.05em}} t_1$ in $\mathbb{T}_n$. To finish the proof, it suffices to show that
\begin{enumerate}
    \item[(1)] $\mathbf{g}_{\mathcal{M}_{l;t}^{B_1;t_1}} = \mathbf{g}_{l; t}^{B_1; t_1}$,
    \item[(2)] $F_{\mathcal{M}_{l;t}^{B_1;t_1}} = F_{l; t}^{B_1; t_1}$,
    \item[(3)] $\mathbf{h}_{\mathcal{M}_{l;t}^{B_1;t_1}} = \mathbf{h}_{l; t}^{B_1; t_1}$.
\end{enumerate}
First, observe that $\E_{\mathcal{P}(Q(t),\mathbf{d},\s(t))}\left(E^-_l\left(Q(t),\mathbf{d},\s(t)\right)\right)=0$. Thus, $\E_{\mathcal{P}(Q,\mathbf{d},\s)}(\mathcal{M})=\E_{\mathcal{P}(Q(t),\mathbf{d},\s(t))}(E^-_l)=0$ via \eqref{cor1} which implies that $\mathcal{M}$ is in general position. Therefore, as a direct result of \eqref{gbeta}, \eqref{chechgg'} and \eqref{F-mutation}, the $\check{g}$-vector $\check{\mathbf{g}}_{\mathcal{M}_{l;t}^{B;t_0}}$ satisfies \eqref{relate of h in gca} and is related to $\check{\mathbf{g}}_{\mathcal{M}_{l;t}^{B_1;t_1}}$ through the same rule \eqref{mutation rule of g in gca} that expresses $\mathbf{g}_{l;t}^{B_1;t_1}$ in terms of $\mathbf{g}_{l;t}^{B;t_0}$. Moreover, the $F$-polynomials $F_{\mathcal{M}_{l;t}^{B;t_0}}$ and $F_{\mathcal{M}_{l;t}^{B_1;t_1}}$ obey \eqref{mutation of F in gca}. Thus, we obtain the proof of (1) and (2). Finally, to prove (3), we apply the following lemma to $\mathcal{M}_{l;t}^{B_1;t_1}$, where the proof technique parallels \cite[Proposition 3.3]{derksen2010quivers}.
\begin{lem}
Under the assumption that $F_M\in \mathbb{Q}_{\mathrm{sf}}(\mathbf{y},\mathbf{z})$, we have 
$$ y_1^{h(1)}\cdots y_n^{h(n)}=F_M|_{\mathrm{Trop}(\mathbf{y},\mathbf{z})}(y_i\xrightarrow{}y_i^{-1}\prod_{j\neq i}y_j^{d_j[-b_{j,i}]_+}).$$ 
\end{lem}
\end{proof}
\begin{cor}\label{apply 1}
    Each polynomial $F_{l;t}^{B,t_0}(\mathbf{y},\mathbf{z})$ has constant term $1$ and contains a unique maximal-degree monomial in the indeterminates $y_1,\ldots,y_n$ with coefficient 1.
\end{cor}
\begin{cor}\label{apply 2}
For fixed $t, B, t_0$, and for all $l$, the $k$-th components of $\mathbf{g}_{l;t}^{B,t_0}$ are either all nonnegative, or all nonpositive.
\end{cor}
\begin{proof}
Let $\mathcal{M}=\bigoplus _l\mathcal{M}_{l;t}^{B,t_0}$. Lemma \ref{general is g-coherent} implies that $\mathcal{M}$ is $\check{g}$-coherent, i.e., $$\min(\check{\beta}_{+,\mathcal{M}}(k),\check{\beta}_{-,\mathcal{M}}(k))=0.$$ 
Suppose that $\check{\beta}_{+,\mathcal{M}}(k)=0$,which means $\ker \beta_{\mathcal{M},k}=0$, where the map $\beta_{\mathcal{M},k}$ is given by \eqref{defn of beta}. Let $\beta_{k,l}$ denote the restriction of $\beta_{\mathcal{M},k}$ on $\mathcal{M}_l:=\mathcal{M}_{l;t}^{B,t_0}$; then $\ker\beta_{\mathcal{M},k}=0$ implies that $\check{\beta}_{+,\mathcal{M}_l}
=0$.  This shows that $\mathbf{g}_{\mathcal{M}_l}=\check{\beta}_{-,\mathcal{M}_l}$ is nonnegative since each $\mathcal{M}_l$ is  $\check{g}$-coherent. If $\check{\beta}_{-,\mathcal{M}}(k)=0$, then the same argument shows that the $k$-th coordinates of all these vectors are nonpositive.
\end{proof}
\begin{cor}\label{apply 3}
For every $t\in \mathbb{T}_n$, $\{\mathbf{g}_{l;t}^{B,t_0}\}_l$ is a $\mathbb{Z}$-basis of the lattice $\mathbb{Z}^n$.
\end{cor}
\begin{proof}
 We prove this by induction on the distance between $t_0$ and $t$ in $\mathbb{T}_n$. The base case where $t=t_0$ is clear. Let $\mathbf{g}_{l,t}^{B_1,t_1}=(g'_l(i))$. Suppose that  there exist $a_l\in \mathbb{Z}$ for $1\leq l\leq n$ such that $\sum a_l\mathbf{g}_{l,t}^{B_1,t_1}=0$. By Theorem \ref{mutation rule of g} and Theorem \ref{interpretation of g and F}, we see that
 \begin{align}
     \sum a_lg'_l(i)=0 \quad \text{iff}\quad \begin{cases}
     \sum a_lg_l(k)=0;\\
     \sum a_l\left(g_l(i)+d_k[b_{i,k}]_+[g_l(k)]_+-d_k[-b_{i,k}]_+[-g_l(k)]_+\right)=0.
     \end{cases}
 \end{align}
 By Proposition \ref{apply 2}, we may assume that $[-g_l(k)]_+=0$ for all $1\leq l\leq n$. Then it is straightforward to check that $\sum a_l\mathbf{g}_{l,t}^{B_1,t_1}$  implies $a_l=0$ for all $l$.
\end{proof}

\subsection{Generic character}
Throughout this subsection, we assume that $(Q,\mathbf{d},\s)$ is a nondegenerate and locally free $H$-based QP with $d_k\leq 2$ for all vertices $k$. Let $\mathcal{A}(Q,\mathbf{d})$ be the corresponding generalized cluster algebra. We assume that the coefficient semifield of the generalized cluster algebra $\mathcal{A}(Q,\mathbf{d})$ is the tropical semifield $\mathrm{Trop}(\mathbf{z})$ and all $y$-variables are set to 1.  Then $\mathbb{ZP}$ is the Laurent polynomial ring $\mathbb{Z}[\mathbf{z}^{\pm1}]$.
\begin{defn}
    We define the \textbf{cluster character}    $C:\Rep(Q,\mathbf{d},\mathcal{S})\to \mathbb{Z}[\mathbf{x}^{\pm 1},\mathbf{z}^{\pm 1}]$ by 
\begin{equation}\label{defn of cluster character}
C(\mathcal{M})=\mathbf{x}^{\check{g}(\mathcal{M})}F_{\mathcal{M}}(\hat{\mathbf{y}}, \mathbf{z}),
\end{equation}
where $\hat{\mathbf{y}}$ is given by \eqref{defn of hat-y}.
\end{defn}
Using  Lemma \ref{g-vector of M}, we can rewrite $\check{g}(i)$ as 
    $$d_i\check{g}(i)=\dim M_{i,\out}-\dim M_i-\rank \gamma_i=\sum_j[-b_{i,j}]_+d_ir_j -d_ir_i -\rank \gamma_i,$$
    where $r_i$ is the rank of the free $H_i$-module $M_i$.
Applying  \eqref{defn of hat-y}, we can reinterpret \eqref{defn of cluster character} as
\begin{align*}
        C(\mathcal{M})=\prod_{i=1}^nx_i^{-r_i}\sum_{\mathbf{e},\mathbf{t}} \chi(\mathrm{Gr}_{\mathbf{e},\mathbf{t}}(M)) \mathbf{f}^\mathbf{t}(\mathbf{z})\prod_{i=1}^nx_i^{-d_i^{-1}\rank \gamma_i+\sum_j[b_{i,j}]_+e_j-[-b_{ij}]_+(e_j-r_j)}.
    \end{align*} 
\begin{lem}\label{commute with C}
 The mutations commute with the cluster character $C$ for a generic representation $\mathcal{M}$. So $C(\mathcal{M})$ is an element in the upper cluster algebra $\overline{\mathcal{A}}(Q,\mathbf{d})$. 
\end{lem}
\begin{proof}
The commuting property $\mu_{k}(C(\mathcal{M}))=C\left(\mu_{k}(\mathcal{M})\right)$ is equivalent to that
$$
\mathbf{x}^{\check{g}} F_{M}(\hat{\mathbf{y}},\mathbf{z})=\left(\mathbf{x}^{\prime}\right)^{\check{g}^{\prime}} F_{M^{\prime}}\left(\hat{\mathbf{y}}^{\prime},\mathbf{z}\right).
$$
Comparing with (\ref{F-mutation}), we see that it suffices to show that
$$
\left(\sum_{s=0}^{d_k}z_{k,s} \hat{y}_k^s\right)^{\check{\beta}_{+}(k)} \mathbf{x}^{\check{g}}=\left(\sum_{s=0}^{d_k} z_{k,s}\left(\hat{y}_k^\prime\right)^s\right)^{\check{\beta}^\prime_{+}(k)}\left(\mathbf{x}^{\prime}\right)^{\check{g}^{\prime}} .
$$
The following equalities are all equivalent to the above.
\begin{align*}
    \left(\sum_{s=0}^{d_k}z_{k,s} \hat{y}_k^s\right)^{\check
    {\beta}_{+}(k)} \mathbf{x}^{\check{g}}&=\left(\sum_{s=0}^{d_k}z_{k,s} \hat{y}_k^{-s}\right)^{\check{\beta}_+^\prime(k)}\left(x_k^{-1}\mathbf{x}^{d_k[-\mathbf{b}_k]_+}\left(\sum_{s=0}^{d_k} z_{k,s}\hat{y}_k^s\right)\right)^{\check{g}'(k)}\mathbf{x}^{\check{g}'-\check{g}'(k)\mathrm{e}_k} \tag{\eqref{mutation of x}, \eqref{mutation of hat-y}}\\
    \mathbf{x}^{\check{g}}&= \mathbf{x}^{-d_k\check{\beta}_+^\prime(k)\mathbf{b}_k+d_k\check{\beta}'(k)[-\mathbf{b}_k]_+ +\check{g}'-2\check{g}'(k)\mathrm{e}_k}\tag{\eqref{gbeta},\eqref{reciprocity condition}}
\end{align*}
The last equality can be easily verified using $\mathbf{b}_k=[\mathbf{b}_k]_+ - [-\mathbf{b}_k]_+$ for which $\mathbf{b}_k$ is the $k$-th column of the matrix $B(Q)$, and the mutation rule of $\check{g}$-vector in \eqref{chechgg'}. 

To show $C(\mathcal{M}) \in \overline{\mathcal{A}}(Q,\mathbf{d})$, we use the same argument as in that of \cite[Lemma 5.3]{fei2017cluster}. Let $\mathbf{x}_{k}$ be the cluster obtained from $\mathbf{x}$ by applying the mutation at $k$. By Theorem \ref{one step mutation} it suffices to show that $C(\mathcal{M}) \in \mathcal{L}_{\mathbf{x}_{k}}$ for any mutable vertex $k$. The expression of $C(\mathcal{M})$ with respect to the cluster $\mathbf{x}_{k}$ is given by $\mu_{k}(C(\mathcal{M}))$, which is $C\left(\mu_{k}(\mathcal{M})\right) \in \mathcal{L}_{\mathbf{x}_{k}}$. 
\end{proof}
\begin{defn}
    A \textbf{cluster monomial} in a generalized cluster algebra $\mathcal{A}(Q,\mathbf{d})$ is a monomial in cluster variables appearing in the cluster pattern $\Sigma$.  
\end{defn}
Suppose that $\prod_l(x_{l;t}^{B(Q),t_0})^{m_l}$ is a  generalized cluster monomial in $\mathcal{A}(Q,\mathbf{d})$ and $t_0\overset{\mathbf{k}}{\rule[3pt]{6mm}{0.05em}}t$ is the path joining $t_0$ and $t$ in $\mathbb{T}_n$.  Let $\mathcal{M}$ be the representation mutated from the negative representation $\bigoplus_lm_lE_l^-$ of $(Q',\mathbf{d},\s')$ via the sequence of mutations $\mu_\mathbf{k}$. Using \eqref{express x in term of F and g} and Theorem \ref{thm: interpretation of g and F}, we obtain $C(\mathcal{M})=\prod_l(x_{l;t}^{B(Q),t_0})^{m_l}$.

Suppose that an element $\omega \in \overline{\mathcal{A}}(Q,\mathbf{d})$ can be written as
$$
\omega=\mathrm{x}^{g(\omega)} F(\hat{\mathbf{y}},\mathbf{z}),
$$
where $F$ is a primitive rational polynomial in $\mathbb{Q}[\mathbf{\hat{y}},\mathbf{z}]$ , and $g(\omega) \in \mathbb{Z}^{n}$. By definition, a \textbf{primitive rational polynomial} is a ratio of two polynomials not divisible by any $\hat{y}_i
$.  
\begin{lem}
  If we assume that the matrix $B(Q)$ has full rank, then the elements $\hat{y}_{1}, \hat{y}_{2}, \ldots, \hat{y}_{p}$ are algebraically independent so that the vector $g(\omega)$ is uniquely determined.
\end{lem}
\begin{proof}
    Use the technique in \cite[Proposition 7.8]{fomin2006cluster}.
\end{proof}
 We call the vector $g(\omega)$ the $g$-vector of $\omega$. Note that the $g$-vector of $C(\mathcal{M})$ as in Lemma \ref{commute with C}  is $\check{g}(\mathcal{M})$. Definition implies at once that for two such elements $\omega_{1}, \omega_{2}$ we have that $g\left(\omega_{1} \omega_{2}\right)=g\left(\omega_{1}\right)+g\left(\omega_{2}\right)$. So the set $\mathbb{Z}^n$ of all $g$-vectors in $\overline{\mathcal{A}}(Q,\mathbf{d})$ forms a semigroup.
\begin{lem}\label{bases}
Assume that the matrix $B(Q)$ has full rank. Let $\Omega=\left\{\omega_{1}, \omega_{2}, \ldots, \omega_{s}\right\}$ be a subset of $\overline{\mathcal{A}}(Q,\mathbf{d})$ with well-defined $g$-vectors. If $g\left(\omega_{i}\right)$'s are all distinct, then $\Omega$ is linearly independent over $\mathbb{Z}$.
\end{lem}
\begin{proof}
The proof is the same as \cite[Lemma 5.5]{fei2017cluster}.
\end{proof}
It follows from Lemma \ref{commute with C} and Lemma \ref{bases}  that
\begin{thm}\label{base by cluster character}
Suppose that $(Q, \mathbf{d},\mathcal{S})$ is a nondegenerate and locally free $H$-based QP with $d_k\leq 2$ for all vertices $k$, and $B(Q)$ has full rank. Let $\mathcal{R}$ be a set of general decorated representations with all distinct $\check{g}$-vectors, then $C$ maps $\mathcal{R}$ (bijectively) to a set of $\mathbb{ZP}$-linearly independent elements in the upper generalized cluster algebra $\overline{\mathcal{A}}(Q,\mathbf{d})$.
\end{thm}
\begin{defn}
Define the \textbf{generic character} $C_{\text{gen}}:\mathbb{Z}^n\to \mathbb{Z}[\mathbf{x}^{\pm 1},\mathbf{z}^{\pm 1}]$ by
\begin{equation}
    C_{\text{gen}}(\check{g})=\mathbf{x}^{\check{g}}F_{\ker(\check{g})}(\mathbf{y},\mathbf{z}).
\end{equation}
\end{defn}
\begin{pro}
The generic character is well-defined.
\end{pro}
\begin{proof}
Building on the discussion immediately following \eqref{incidence variety}, we observe that for any open set $O\subseteq PC(\check{g})$, $p_1p_2^{-1}(O)$ is constructible and dense, and thus contains an open subset of $\IHom(\check{g})$. To show that $C_{\text{gen}}$ is well-defined, it suffices to show that $\chi(\mathrm{Gr}_{\mathbf{e},\mathbf{t}}(M))$ is constant on an open dense subset of $PC(\check{g})\subseteq \Rep_{\alpha}(Q)$.

We consider the  function:
\begin{align*}
\mu_{\mathbf{e},\mathbf{t}}:PC(\check{g})&\longrightarrow \mathbb{Z}\\
    M &\longmapsto \chi(\mathrm{Gr}_{\mathbf{e},\mathbf{t}}(M)).
\end{align*}
Using the fact that a constructible function on an irreducible variety admits a generic value, we need to show that the function $\mu_{\mathbf{e},\mathbf{t}}$ is constructible. To prove this, we need the following proposition:
\begin{pro}[{\cite[Proposition 4.1.31]{dimca2004sheaves}}] \label{construntive map}
For an algebraic variety $X$, the ring of constructible functions from $X$ to $\mathbb{Z}$ is denoted by $CF(X)$. Any morphism $f: X \to Y$ of complex algebraic varieties induces a well-defined push-forward homomorphism $CF(f): CF(X)\to CF(Y)$, determined by the property
$$CF(f)(1_Z)(y)=\chi(f^{-1}(y)\cap Z)$$
for any closed subvariety $Z$ in $X$ and any point $y\in Y$.
\end{pro}

Next, for any dimension vector $\mathbf{e}$ and Jordan type   $\mathbf{t}$, we consider the quasi-projective variety: 
\begin{align*}
     Z(\alpha,\mathbf{e},\mathbf{t})=\Big\{(\mathcal{M},\mathcal{N})\in PC(\check{g})\times  \prod_{i=1}^n \mathrm{Gr}\binom{\alpha_i}{e_i}\Big|& M(a)(N_{ta})\subseteq N_{ha}\text{ for all $a\in Q_1$},\\ 
     & \text{and the Jordan type of $N_i$ is $\mathbf{t}_i$}\Big\}.
\end{align*}
We have two natural projections:
\begin{center}
    \begin{tikzcd}
      & {Z(\alpha,\mathbf{e},\mathbf{t})} \arrow[ld, "p"'] \arrow[rd, "q"] &                                             \\
PC(\check{g}) &                                                                    & {\mathrm{Gr}\binom{\alpha}{\mathbf{e}}}
\end{tikzcd}
\end{center}
The projection $p$ is the composition of the inclusion $Z(\alpha,\mathbf{e},\mathbf{t})\to PC(\check{g})$ and the projection $PC(\check{g})\times \mathrm{Gr}\binom{\alpha_i}{e_i}\to PC(\check{g})$. Thus, $p$ is an algebraic variety morphism.

Applying Proposition \ref{construntive map} to the first projection $p:Z(\alpha,\mathbf{e},\mathbf{t})\to PC(\check{g})$, we conclude that $ \mu_{\mathbf{e},\mathbf{t}}=CF(p)(1_{Z(\alpha,\mathbf{e},\mathbf{t})})$.
\end{proof}
It follows from Theorem \ref{base by cluster character} that
\begin{thm}\label{bases in gca}
Suppose that  $(Q, \mathbf{d},\mathcal{S})$ is a nondegenerate and locally free $H$-based QP with $d_k\leq 2$ for all vertices $k$, and $B(Q)$ has full rank. Then the generic character $C_{\text{gen}}$ maps $\mathbb{Z}^n$ (bijectively) to a set of linearly independent elements in $\overline{\mathcal{A}}(Q,\mathbf{d})$ containing all cluster monomials.
\end{thm}

\section*{Acknowledgement}
This article is part of my PhD thesis, under the supervision of Professor Jiarui Fei. I would like to express my deepest gratitude to Professor Fei for his exceptional guidance, unwavering support, and great patience.

\appendix
 \section{Proofs by Adaptation to \texorpdfstring{$H$}{H}-based Setting}
Our proof adapts the stratification technique from \cite{derksen2008quivers} to the $H$-based setting.
\begin{proof}[Proof of Theorem \ref{splitting thm}]\label{app:thm splitting}
Let us first prove the existence of  \eqref{existence of split thm}.
There is nothing to prove if $(H, A, \s)$ is reduced, so let us assume that $\s^{\circ (2)} \neq 0$. Using Proposition \ref{triv} and replacing $\s$ if necessary by a cyclically equivalent potential, we can assume that $\s$ is of the form
\begin{align}\label{split thm 1}
    \s=\sum_{k=1}^{N}\left(a_{k} b_{k}+ a_k u_{k}+v_{k} b_{k}\right)+\s^{\prime}
\end{align}
where each $a_{k} b_{k}$ is a cyclic 2-path, the arrows $a_{1}, b_{1}, \cdots, a_{N}, b_{N}$ in $Q_1^\circ$ form a basis of $A^{\circ}_{\text{triv}}$, the elements $u_{k}$ and $v_{k}$ belong to $\mathfrak{m}^\circ(A)^{2}+H^\circ\mathfrak{m}^\circ(A)+\mathfrak{m}^\circ(A)H^\circ\subseteq \mathfrak{m}(A)^2$, and the potential $\s^{\prime}\in\mathfrak{m}^\circ(A)^3+\mathfrak{m}^\circ(A)^2H^\circ\subseteq \mathfrak{m}(A)^3$  is a linear combination of cyclic paths containing none of the arrows $a_{k}$ or $b_{k}$.

By Lemma \ref{two topologies concides}, the  $m^\circ$-adic and $m$-adic topologies coincide on $\overline{T_H(A)}$.
 We preserve all technical constructions of the automorphisms in the following results in exact accordance with \cite[Lemma 4.7, Proposition 4.10]{derksen2008quivers}, where the original proof relies on the $m$-adic topology of the completed path algebra. Thus, Lemma \ref{lem in split thm} and Proposition \ref{lem in reduce in split thm} follow directly from the same reasoning under the $m$-adic topology.
\begin{lem}[{\cite[Lemma 4.7]{derksen2008quivers}}]\label{lem in split thm}
For every potential $\s$ of the form (\ref{split thm 1}), there exists a quasi-unitriangular automorphism $\varphi$ of $\overline{T_H(A)}$ such that $\varphi(\s)$ is cyclically equivalent to a potential of the form (\ref{split thm 1}) with $u_{k}=v_{k}=0$ for all $k$.
\end{lem}
\begin{pro}[{\cite[Proposition 4.10]{derksen2008quivers}}]\label{lem in reduce in split thm}
Let $(H,A,\s)$ and $\left(H,A, \s^{\prime}\right)$ be reduced $H$-based QPs such that $\s^{\prime}-\s \in J(\s)^{2}$. Then we have:
\begin{enumerate}
    \item[(1)] $J\left(\s^{\prime}\right)=J(\s)$.
    \item[(2)] $(H,A,\s)$ is right-equivalent to $\left(H,A,\s^{\prime}\right)$. More precisely, there exists an algebra automorphism $\varphi$ of $\overline{T_H(A)}$ such that $\varphi(\s)$ is cyclically equivalent to $\s^{\prime}$, and $\varphi(u)-u \in J(\s)$ for all $u \in \overline{T_H(A)}$.
\end{enumerate}
\end{pro}

The above argument makes it clear that the right-equivalence class of $(H,A_{\text {triv }}, \s_{\text {triv }})$ is determined by the right-equivalence class of $(H,A,\s)$. To prove Theorem \ref{splitting thm}, it remains to show that the same is true for $\left(H,A_{\text {red }}, \s_{\text {red }}\right)$. Changing notation a little bit, we need to prove the following.
\begin{pro}\label{reduce in split thm}
Let $(H,A,\s)$ and $\left(H,A, \s^{\prime}\right)$ be reduced $H$-based QPs, and $(H,C,\mathcal{T})$ a trivial $H$-based QP. If $(H,A \oplus C, \s+\mathcal{T})$ is right-equivalent to $\left(H,A \oplus C, \s^{\prime}+\mathcal{T}\right)$ then $(H,A,\s)$ is right-equivalent to $\left(H,A,\s^{\prime}\right)$.
\end{pro}
 \begin{proof}
We abbreviate $J=J(\s)$ and $J^{\prime}=J\left(\s^{\prime}\right)$ (understood as the Jacobian ideals of $S$ and $\s^{\prime}$ in $\overline{T_H(A)}$). As in Proposition \ref{triv}, let $L$ denote the closure of the two-sided ideal in $\overline{T_H(A\oplus C)}$ generated by $C$. Then we have
\begin{align}\label{split 1}
    \overline{T_H(A\oplus C)}=\overline{T_H(A)} \oplus L, \quad J(\s+\mathcal{T})=J \oplus L, \quad J\left(\s^{\prime}+\mathcal{T}\right)=J^{\prime} \oplus L.
\end{align}
Let $\varphi$ be an automorphism of $\overline{T_H(A\oplus C)}$, such that $\varphi(\s+\mathcal{T})$ is cyclically equivalent to $\s^{\prime}+\mathcal{T}$. In view of (\ref{split 1}) and Proposition \ref{iso of jacobian algebra}, we have
\begin{align}
    \varphi(J \oplus L)=J^{\prime} \oplus L
\end{align}
Let $\psi: \overline{T_H(A)}\rightarrow \overline{T_H(A)}$ denote the restriction to $\overline{T_H(A)}$ of the composition $p \varphi$, where $p$ is the projection of $\overline{T_H(A\oplus C)}$ onto $\overline{T_H(A)}$ along $L$. In view of Proposition \ref{lem in reduce in split thm}, it suffices to show the following:
\begin{align}\label{split 2}
\psi \text { is an automorphism of } \overline{T_H(A)} \text { such that } 
\end{align}
$$
    \s^{\prime}-\psi(\s) \text { is cyclically equivalent to an element of } \psi\left(J^{2}\right)
$$
(indeed, assuming (\ref{split 2}) and using Proposition \ref{iso of jacobian algebra}, we see that $\psi\left(J^{2}\right)=J(\psi(\s))^{2}$, hence one can apply Proposition \ref{lem in reduce in split thm} to potentials $S^{\prime}$ and $\psi(S)$ ).

Clearly, $\psi$ is an algebra homomorphism, so can be represented by a pair $\left(\psi^{(1)}, \psi^{(2)}\right)$ as in Proposition \ref{determined morphism}. To show that $\psi$ is an automorphism of $\overline{T_H(A)}$, it suffices to show that $\psi^{(1)}$ is an $H$-bimodule automorphism of $A$. Precisely, it suffices to show that $\psi^{(1,1)}$ is an $R$-bimodule automorphism of $A^\circ$. By the definition, if we write the $R$-bimodule automorphism $\varphi^{(1,1)}$ of $A^\circ \oplus C^\circ$ as a matrix
$$\begin{pmatrix}
    \varphi_{A^\circ A^\circ} & \varphi_{A^\circ C^\circ} \\
\varphi_{C^\circ A^\circ} & \varphi_{C^\circ C^\circ}
\end{pmatrix}$$
then $\psi^{(1,1)}=\varphi_{A^\circ A^\circ}$. Since
$$
\varphi(C^\circ) \subset \varphi(J \oplus L)=J^{\prime} \oplus L \subseteq (\mathfrak{m}^\circ(A)^{2}+\mathfrak{m}^\circ H^\circ+H^\circ\mathfrak{m}^\circ) \oplus L
$$
it follows that $\varphi_{A^\circ C^\circ}=0$, implying that $\psi^{(1,1)}=\varphi_{A^\circ A^\circ}$ is an $R$-bimodule automorphism of $A^\circ$, and that $\psi$ is an automorphism of $\overline{T_H(A)}$.

Since $\s^{\prime}+T$ is cyclically equivalent to $\varphi(\s+\mathcal{T})$, the same is true for the potentials obtained from them by applying the projection $p$; it follows that $\s^{\prime}-\psi(\s)$ is cyclically equivalent to $p \varphi(\mathcal{T})$. Since $\mathcal{T} \in C^{2}$, the claim that $\s^{\prime}-\psi(\s)$ is cyclically equivalent to an element of $\psi\left(J^{2}\right)$ follows from the fact that $p \varphi(L) \subseteq \psi(J)$, or, equivalently, that $\varphi(L) \subseteq \varphi(J)+L$. Applying the inverse automorphism $\varphi^{-1}$ to both sides, it suffices to show that $L \subseteq J+\varphi^{-1}(L)$. Using the obvious symmetry between $J$ and $J^{\prime}$, it is enough to show the inclusion $L \subseteq J^{\prime}+\varphi(L)$.

Let us abbreviate $M=\mathfrak{m}^\circ(A \oplus C)$, and $I=J^{\prime}+\varphi(L)$. Since $\varphi(J) \subseteq J^{\prime} \oplus L$, and $J \subseteq (\mathfrak{m}^\circ(A)+H^\circ)^{2}$, it follows that $\varphi(J) \subseteq J^{\prime} \oplus\left(L \cap (M+H^\circ)^{2}\right)\subseteq J^{\prime}+(M+H^\circ)L+L(M+H^\circ)$. Therefore, we have
$$
L \subseteq J^{\prime}+L=\varphi(J)+\varphi(L) \subseteq I+(M+H^\circ) L+L (M+H^\circ)
$$
Substituting this upper bound for $L$ into its right hand side, we deduce the inclusion
$$
L \subseteq I+(M+H^\circ)^{2} L+(M+H^\circ) L (M+H^\circ)+L (M+H^\circ)^{2}
$$
Continuing in the same way, for every $n>0$, we have the inclusion
$$
L \subseteq I+\sum_{k=0}^{n} (M+H^\circ)^{k} L (M+H^\circ)^{n-k}\subseteq I+(M+H^\circ)^n
$$
In view of (\ref{toplogy in complete algebra}) and Lemma \ref{two topologies concides}, it follows that $L$ is contained in $\overline{I}$, the closure of $I$ in $\overline{T_H(A\oplus C)}$. However, it is easy to see that $I=J^{\prime}+\varphi(L)$ is closed in $\overline{T_H(A\oplus C)}$ (indeed, the closedness of $I$ is equivalent to that of $\varphi^{-1}(I)=\varphi^{-1}\left(J^{\prime}\right)+L$, and so, by symmetry, it is enough to show that $\varphi(J)+L$ is closed; but this is clear since $\varphi(J)+L=p^{-1}(\psi(J))$ is the inverse image of the closed ideal $\psi(J)$ of $\overline{T_H(A)}$). This completes the proofs of Proposition \ref{reduce in split thm} and Theorem \ref{splitting thm}.
 \end{proof}
\end{proof}
\begin{proof}[Proof of Theorem \ref{mutation uniquele right-equivalence class of the GQP}]\label{app:mutation uniquele right-equivalence class of the GQP}
   Let $\widehat{A}$ be the finite-dimensional $H$-bimodule given by
\begin{align}
    \widehat{A}=A \oplus\left(e_{k} A\right)^{\star} \oplus\left(A e_{k}\right)^{\star} .
\end{align}
The natural embedding $A \rightarrow \widehat{A}$ identifies $\overline{T_H(A)}$ with a closed subalgebra in $\overline{T_H(\widehat{A})}$. We also have a natural embedding $\widetilde{A} \rightarrow \overline{T_H(\widehat{A})}$ (sending each arrow $[b \varepsilon_k^l a]$ to the product $b \varepsilon_k^l a$). This allows us to identify $\overline{T_H(\widetilde{A})}$ with another closed subalgebra in $\overline{T_H(\widehat{A})}$, namely, with the closure of the linear span of the paths $\varepsilon_{ha_1}^{l_1}\widehat{a}_{1} \cdots \widehat{a}_{d}\varepsilon_{ta_d}^{l_d+1}$ such that $\widehat{a}_{1} \notin e_{k} A$ and $\widehat{a}_{d} \notin A e_{k}$. Under this identification, the potential $\widetilde{\s}$ given by (\ref{mutation of s}) and viewed as an element of $\overline{T_H(\widehat{A})}$ is cyclically equivalent to the potential
$$
\s+\sum_{l=0}^{d_k-1}\left(\left(\sum_{b \in Q^\circ_{1}  \cap A^{\circ} e_{k}} \varepsilon_k^{d_k-l-1} b^{\star} b\right)\left(\sum_{a \in Q^\circ_{1} \cap e_{k} A} \varepsilon_k^l a a^{\star}\right)\right).
$$

Taking this into account, we see that Theorem \ref{mutation uniquele right-equivalence class of the GQP} becomes a consequence of the following lemma.
\begin{lem}\label{lem of mutation uniquele right-equivalence class of the GQP}
Every automorphism $\varphi$ of $\overline{T_H(A)}$ can be extended to an automorphism $\widehat{\varphi}$ of $\overline{T_H(\widehat{A})}$ satisfying
\begin{align}\label{in lem1 of mutation uniquele right-equivalence class of the GQP}
    \widehat{\varphi}\left(\overline{T_H(\widetilde{A})}\right)=\overline{T_H(\widetilde{A})}
\end{align}
and
\begin{align}\label{in lem2 of  mutation uniquele right-equivalence class of the GQP}
    \widehat{\varphi}\left(\sum_{a \in Q_{1} \cap e_{k} A} a a^{\star}\right)=\sum_{a \in Q_{1} \cap e_{k} A} a a^{\star}, \quad \widehat{\varphi}\left(\sum_{b \in Q_{1} \cap A e_{k}} b^{\star} b\right)=\sum_{b \in Q_{1} \cap A e_{k}} b^{\star} b.
\end{align}
\end{lem}
   In order to extend $\varphi$ to an automorphism $\widehat{\varphi}$ of $\overline{T_H(\widetilde{A})}$, we need only to define the elements $\widehat{\varphi}\left(a^{\star}\right)$ and $\widehat{\varphi}\left(b^{\star}\right)$ for all arrows $a \in Q_{1}^\circ \cap e_{k} A$ and $b \in Q_{1}^\circ \cap A e_{k}$.

We first deal with $\widehat{\varphi}\left(a^{\star}\right)$. Let $Q^\circ_{1} \cap e_{k} A=\left\{a_{1}, \ldots, a_{s}\right\}$. In view of Proposition \ref{determined morphism}, the action of $\varphi$ on these arrows is given by
\begin{equation}\label{action of varphi on a}
\setlength{\arraycolsep}{1.5pt}
\begin{aligned}
    \left(\begin{array}{llll}
\varphi\left(a_{1}\right), & \varphi\left(a_{2}\right), & \cdots, & \varphi\left(a_{s}\right)
\end{array}\right)&=\left(\begin{array}{llll}
a_{1}, & a_{2}, & \cdots, & a_{s}
\end{array}\right)\left(C_{0}+C_{1}\right)+\varepsilon_k(\begin{array}{llll}
a_{1}, & a_{2}, & \cdots ,& a_{s}
\end{array})C_2\\&+\cdots+\varepsilon_k^{d_k-1}(\begin{array}{llll}
a_{1}, & a_{2}, & \cdots, & a_{s}
\end{array})C_{d_k},
\end{aligned}
\end{equation}
where:
\begin{itemize}
    \item $C_{0}$ is an invertible $s \times s$ matrix with entries in $K$ such that its $(p, q)$-entry is 0 unless $t\left(a_{p}\right)=t\left(a_{q}\right)$;
    \item $C_{1}$ is a $s \times s$ matrix whose $(p, q)$-entry belongs to $\left(\mathfrak{m}^\circ(A)+H^\circ\right)_{t\left(a_{p}\right), t\left(a_{q}\right)}$;
    \item  $C_{l}$ is a $s \times s$ matrix whose $(p, q)$-entry belongs to $\left(\mathfrak{m}^\circ(A)+H\right)_{t\left(a_{p}\right), t\left(a_{q}\right)}$ for $2\leq l\leq d_k$.
\end{itemize}
Note that $C_{0}+C_{1}$ is invertible, and its inverse is of the same form: indeed, we have
$$
\left(C_{0}+C_{1}\right)^{-1}=\left(I+C_{0}^{-1} C_{1}\right)^{-1} C_{0}^{-1}=\left(I+\sum_{n=1}^{\infty}(-1)^{n}\left(C_{0}^{-1} C_{1}\right)^{n}\right)C_{0}^{-1}.
$$

Now we define the elements $\widehat{\varphi}\left(a_{p}^{\star}\right)$ by setting
\begin{align*}
    \left(\begin{array}{c}
\widehat{\varphi}\left(a_{1}^{\star}\right) \\
\widehat{\varphi}\left(a_{2}^{\star}\right) \\
\vdots \\
\widehat{\varphi}\left(a_{s}^{\star}\right)
\end{array}\right)=\left(C_{0}+C_{1}\right)^{-1}\left(\left(\begin{array}{c}
a_{1}^{\star} \\
a_{2}^{\star} \\
\vdots \\
a_{s}^{\star}
\end{array}\right)+D_1\left(\begin{array}{c}
a_{1}^{\star} \\
a_{2}^{\star} \\
\vdots \\
a_{s}^{\star}
\end{array}\right)\varepsilon_k+\cdots+D_{d_k-1}\left(\begin{array}{c}
a_{1}^{\star} \\
a_{2}^{\star} \\
\vdots \\
a_{s}^{\star}
\end{array}\right)\varepsilon_k^{d_k-1}\right)
\end{align*}
where
$$D_l=\sum_{(i_1,\ldots,i_m)\in\Omega_l}(-1)^m C_{i_1+1}(C_0+C_1)^{-1}C_{i_2+1}(C_0+C_1)^{-1}\cdots C_{i_m+1}(C_0+C_1)^{-1},$$
and
$$\Omega_l=\{(i_1,\ldots,i_m)\mid 1\leq i_t\leq l \text{ and } i_1+\cdots i_m=l \text{ for some $m\in \mathbb{N}_+$}\}.$$
It follows that
$$
\setlength{\arraycolsep}{1.5pt}
\begin{aligned}
\widehat{\varphi}\left(\sum_{p} a_{p} a_{p}^{\star}\right) & =\left(\begin{array}{llll}
\widehat{\varphi}\left(a_{1}\right), & \widehat{\varphi}\left(a_{2}\right), & \cdots ,& \widehat{\varphi}\left(a_{s}\right)
\end{array}\right)\left(\begin{array}{c}
\widehat{\varphi}\left(a_{1}^{\star}\right) \\
\widehat{\varphi}\left(a_{2}^{\star}\right) \\
\vdots \\
\widehat{\varphi}\left(a_{s}^{\star}\right)
\end{array}\right)
\end{aligned}
$$
is cyclically equivalent to
\begin{align*}
\setlength{\arraycolsep}{1.5pt}
    \sum_{p} a_pa_p^\star+(\begin{array}{llll}
        a_1,&a_2,&\cdots,&a_s
    \end{array})\left(\sum_{l=1}^{d_k-1}(D_l+\sum_{i=1}^lC_{i+1}(C_0+C_1)^{-1}D_{l-i})\right)\left(\begin{array}{c}
          a_1^\star \\
          a_2^\star \\
          \vdots \\
          a_s^\star
    \end{array}\right)
\end{align*}
with the convention $D_0=\text{Id}$. We check that $(D_l+\sum_{i=1}^lC_{i+1}(C_0+C_1)^{-1}D_{l-i})=0$ by induction on $l$. It is easy to see that $D_1=-C_2(C_0+C_1)^{-1}D_0$, so  we finish the argument for $l=1$. Now assume that $l\geq 2$, and that our assertion holds for any integers less than $l$. Then we have that 
\begin{align*}
    D_l&=\sum_{(i_1,\ldots,i_m)\in\Omega_l}(-1)^m C_{i_1+1}(C_0+C_1)^{-1}C_{i_2+1}(C_0+C_1)^{-1}\cdots C_{i_m+1}(C_0+C_1)^{-1}\\
    &=-\sum_{i_1=1}^l C_{i_1+1}(C_0+C_1)\sum_{(i_2,\ldots,i_m)\in \Omega_{l-i_1}}(-1)^{m-1}C_{i_2+1}(C_0+C_1)^{-1}\cdots C_{i_m+1}(C_0+C_1)^{-1}\\
    &=-\sum_{i_1=1}^lC_{i_1+1}(C_0+C_1)^{-1}D_{l-i_1},
\end{align*}completing our inductive step.

For $b \in Q_{1} \cap A e_{k}$, we define $\widehat{\varphi}\left(b^{\star}\right)$ in a similar way. Namely, let $Q^{\circ}_{1} \cap A e_{k}=$ $\left\{b_{1}, \ldots, b_{t}\right\}$. As above, the action of $\varphi$ on these arrows is given by
\begin{equation}\label{action of varphi on b}
\begin{aligned}
    \left(\begin{array}{c}
\varphi\left(b_{1}\right) \\
\varphi\left(b_{2}\right) \\
\vdots \\
\varphi\left(b_{t}\right)
\end{array}\right)=\left(F_{0}+F_{1}\right)\left(\begin{array}{c}
b_{1} \\
b_{2} \\
\vdots \\
b_{t}
\end{array}\right)+F_2\left(\begin{array}{c}
b_{1} \\
b_{2}\\
\vdots \\
b_{t}
\end{array}\right)\varepsilon_k+\cdots+F_{d_k}\left(\begin{array}{c}
b_{1} \\
b_{2}\\
\vdots \\
b_{t}
\end{array}\right)\varepsilon_k^{d_k-1},
\end{aligned}
\end{equation}
where:
\begin{itemize}
    \item $F_{0}$ is an invertible $t \times t$ matrix with entries in $K$ such that its $(p, q)$-entry is 0 unless $h\left(b_{p}\right)=h\left(b_{q}\right)$;
    \item $F_{1}$ is a $t \times t$ matrix whose $(p, q)$-entry belongs to $\left(\mathfrak{m}^\circ(A)+H^\circ\right)_{h\left(b_{p}\right), h\left(b_{q}\right)}$; 
    \item $F_{l}$ is a $t \times t$ matrix whose $(p, q)$-entry belongs to $\left(\mathfrak{m}^\circ(A)+H\right)_{h\left(b_{p}\right), h\left(b_{q}\right)}$ for $1\leq l\leq d_k$.
\end{itemize}
As above, we see that $D_{0}+D_{1}$ is invertible, and its inverse is of the same form. Now we define the elements $\widehat{\varphi}\left(b_{q}^{\star}\right)$ by setting
\setlength{\arraycolsep}{1.5pt}
\begin{align*}
    &\left(\begin{array}{llll}
\widehat{\varphi}\left(b_{1}^{\star}\right), & \widehat{\varphi}\left(b_{2}^{\star}\right), & \cdots, & \widehat{\varphi}\left(b_{t}^{\star}\right)
\end{array}\right)\\&=\left(\left(\begin{array}{llll}
b_{1}^{\star}, & b_{2}^{\star}, & \cdots, & b_{t}^{\star}
\end{array}\right)+\left(\begin{array}{llll}
b_{1}^{\star}, & b_{2}^{\star}, & \cdots, & b_{t}^{\star}
\end{array}\right)\varepsilon_k G_1+ \cdots+\left(\begin{array}{llll}
b_{1}^{\star} & b_{2}^{\star} & \cdots & b_{t}^{\star}
\end{array}\right)\varepsilon_k^{d_k-1} G_{d_k-1}\right)\left(F_{0}+F_{1}\right)^{-1}
\end{align*}
where
$$G_l=\sum_{(i_1,\ldots,i_m)\in\Omega_l}(-1)^m(F_0+F_1)^{-1} G_{i_1+1}(F_0+F_1)^{-1}G_{i_2+1}\cdots (C_0+C_1)^{-1}G_{i_n+1}.$$
It follows that
$$
\setlength{\arraycolsep}{1.5pt}
\begin{aligned}
\widehat{\varphi}\left(\sum_{q} b_{q}^{\star} b_{q}\right) & =\left(\begin{array}{llll}
\widehat{\varphi}\left(b_{1}^\star\right), & \widehat{\varphi}\left(b_{2}^{\star}\right), & \cdots, & \widehat{\varphi}\left(b_{t}^{\star}\right)
\end{array}\right)\left(\begin{array}{c}
\widehat{\varphi}\left(b_{1}\right) \\
\widehat{\varphi}\left(b_{2}\right) \\
\vdots \\
\widehat{\varphi}\left(b_{t}\right)
\end{array}\right) 
\end{aligned}
$$ is cyclically equivalent to 
\setlength{\arraycolsep}{1.5pt}
\begin{align*}
    \sum_q b^\star b_q+(\begin{array}{llll}
        b_1^\star,&b_2^\star,&\cdots,&b_t^\star
    \end{array})\left(\sum_{l=1}^{d_k-1}(G_l+\sum_{j=1}^lG_{l-j}(C_0+C_1)^{-1}F_{l+1})\right)\left(\begin{array}{c}
          b_1 \\
          b_2 \\
          \vdots \\
          b_t
    \end{array}\right).
\end{align*}
One can check that $G_l+\sum_{j=1}^lG_{l-j}(C_0+C_1)^{-1}F_{l+1}=0$ and so $\widehat{\varphi}\left(\sum_{q} b_{q}^{\star} b_{q}\right)=\sum_q b_q^\star b_q$.
The condition (\ref{in lem2 of  mutation uniquele right-equivalence class of the GQP}) is then clearly satisfied; the construction also makes clear that the automorphism $\widehat{\varphi}$ of $\overline{T_H(\widehat{A})}$ preserves the subalgebra $\overline{T_H(\widetilde{A})}$. As a consequence
of Proposition \ref{determined morphism}, $\widehat{\varphi}$ restricts to an automorphism of $\overline{T_H(\widetilde{A})}$, verifying (\ref{in lem1 of mutation uniquele right-equivalence class of the GQP}) and completing the proofs of Lemma \ref{lem of mutation uniquele right-equivalence class of the GQP}.
 \end{proof}
\begin{proof}[Proof of Theorem \ref{mutation of gqp is involution }]\label{app: mutation of gqp is involution}
     Let $(H,A,\s)$ be a reduced $H$-based  QP satisfying (\ref{condition on mutation of GQP 1}) and (\ref{condition on nutation of GQP 2}). Let $\widetilde{\mu}_{k}(H,A,\s)=(H,\widetilde{A}, \widetilde{\s})$ and $\widetilde{\mu}_{k}^{2}(H,A, \s)=\widetilde{\mu}_{k}(H,\widetilde{A}, \widetilde{\s})=(H,\widetilde{\widetilde{A}}, \widetilde{\widetilde{\s}})$. In view of Theorem \ref{splitting thm} and Proposition \ref{trival part of jacobian}, it is enough to show that
    \begin{align}
        \text{$(H,\widetilde{\widetilde{A}}, \widetilde{\widetilde{\s}})$ is right-equivalent to $(H,A,\s) \oplus(H,C,\mathcal{T})$, where $(H,C,\mathcal{T})$ is a trivial $H$-based QP.}
    \end{align}
    Using (\ref{mutation of A}) twice, and identifying $\left(e_{k} A\right)^{\star}$ with $A^{\star} e_{k}$, and $\left(A e_{k}\right)^{\star}$ with $e_{k} A^{\star}$, where $A^{\star}$ is the dual $H$-bimodule of $A$, we conclude that
    \begin{align}\label{1 in mutation of gqp is involution}
        \widetilde{\widetilde{A}}=A \oplus A e_{k} A \oplus A^{\star} e_{k} A^{\star} .
    \end{align}
    Furthermore, the basis of arrows in $\widetilde{\widetilde{A}}$ consists of the original set of arrows $Q_{1}^\circ$ in $A$ together with the arrows $[b \varepsilon_k^l a] \in A e_{k} A$ and $\left[a^{\star} \varepsilon_k^l b^{\star}\right] \in A^{\star} e_{k} A^{\star}$ for $a \in Q_{1}^\circ \cap e_{k} A$ and $b \in Q_{1}^\circ \cap A e_{k}$, $0\leq l\leq d_k-1$.

Then we see that the potential $\widetilde{\widetilde{\s}}$ is given by
\begin{align}
    \widetilde{\widetilde{\s}}=[[\s]]+\left[\Delta_{k}(A)\right]+\Delta_{k}(\widetilde{A})=[\s]+\sum_{\substack{a, b \in Q^\circ_{1}: h(a)=t(b)=k\\0\leq l\leq d_k-1-l}}\left([b \varepsilon_k^la]\left[a^{\star}\varepsilon_k^{d_k-1-l} b^{\star}\right]+\left[a^{\star} \varepsilon_k^{d_k-1-l} b^{\star}\right] b \varepsilon_k^l a\right),
\end{align}
hence it is cyclically equivalent to
\begin{align}\label{2 in mutation of gqp is involution}
    \s_{1}=[\s]+\sum_{\substack{a, b \in Q_{1}^\circ: h(a)=t(b)=k\\0\leq l\leq d_k-1-l}}([b \varepsilon_k^la]+b \varepsilon_k^l a)\left[a^{\star} \varepsilon_k^{d_k-1-l} b^{\star}\right].
\end{align}

Let us abbreviate
$$
(H,C,\mathcal{T})=\left(H,A e_{k} A \oplus A^{\star} e_{k} A^{\star}, \sum_{\substack{a, b \in Q_{1}^\circ: h(a)=t(b)=k\\0\leq l\leq d_k-1-l}}[b \varepsilon_k^la]\left[a^{\star} \varepsilon_k^{d_k-1-l} b^{\star}\right]\right).
$$
This is a trivial $H$-based QP (cf. Proposition \ref{triv}); therefore to prove Theorem \ref{mutation of gqp is involution } it suffices to show that the $H$-based QP $\left(H,\widetilde{\widetilde{A}}, \s_{1}\right)$ given by (\ref{1 in mutation of gqp is involution}) and (\ref{2 in mutation of gqp is involution}) is right-equivalent to $(H,A,\s) \oplus$ $(H,C,\mathcal{T})$. We proceed in several steps.

\textbf{Step 1}: Let $\varphi_{1}$ be the change of arrows automorphism of $\overline{T_H(\widetilde{\widetilde{A}})}$ multiplying each arrow $b \in Q_{1}^\circ  \cap A e_{k}$ by $-1$ , and fixing the rest of the arrows in $\widetilde{\widetilde{A}}$. Then the potential $\s_{2}=\varphi_{1}\left(\s_{1}\right)$ is given by
$$
\s_{2}=[\s]+\sum_{\substack{a, b \in Q_{1}: h(a)=t(b)=k\\0\leq l\leq d_k-1-l}}([b \varepsilon_k^la]-b \varepsilon_k^l a)\left[a^{\star} \varepsilon_k^{d_k-1-l} b^{\star}\right].
$$

\textbf{Step 2}: Let $\varphi_{2}$ be the unitriangular automorphism of $\overline{T_H(\widetilde{\widetilde{A}})}$ sending each arrow $[b \varepsilon_k^l a] \in A e_{k} A$ to $[b \varepsilon_k^l a]+b \varepsilon_k^l a$, and fixing the rest of the arrows in $\widetilde{\widetilde{A}}$. Remembering the definition of $[\s]$, we can write $[\s]^{\circ (2)}$ as $\sum_{a,b,l}c_{a,b,l}[b\varepsilon_k^l a]$ where $c_{a,b,l}\in \bar{e}_kA^\circ \bar{e}_k$. It is easy to see that the potential $\varphi_{2}\left(\s_{2}\right)$ is cyclically equivalent to a potential of the form
$$
S_{3}=\s+\sum_{\substack{a, b \in Q_{1}: h(a)=t(b)=k\\0\leq l\leq d_k-1-l}}[b \varepsilon_k^la](\left[a^{\star} \varepsilon_k^{d_k-1-l} b^{\star}\right]+c_{a,b,l}+f(a, b,l))
$$
for some elements $f(a, b,l) \in \mathfrak{m}^\circ\left(A \oplus A e_{k} A\right)^2+H^\circ\mathfrak{m}^\circ\left(A \oplus A e_{k} A\right)+\mathfrak{m}^\circ(A\oplus Ae_kA)H^\circ$.

\textbf{Step 3}: Let $\varphi_{3}$ be the  automorphism of $\overline{T_H(\widetilde{\widetilde{A}})}$ sending each arrow $\left[a^{\star}\varepsilon_k^{d_k-1-l} b^{\star}\right] \in A^{\star} e_{k} A^{\star}$ to $\left[a^{\star} \varepsilon_k^{d_k-1-l} b^{\star}\right]-c_{a,b,l}-f(a, b,l)$, and fixing the rest of the arrows in $\widetilde{\widetilde{A}}$. Then we have $\varphi_{3}\left(\s_{3}\right)=\s+\mathcal{T}$.

Combining these three steps, we conclude that the $H$-based QP $\left(H,\widetilde{\widetilde{A}}, \s_{1}\right)$ is right-equivalent to $(H,\widetilde{\widetilde{A}}, \s+\mathcal{T})=(H,A,\s) \oplus(H,C,\mathcal{T})$, finishing the proof of Theorem \ref{mutation of gqp is involution }.
 \end{proof}

\begin{proof}[Proof of Lemma \ref{lemma 2 in jacion algebra invaraints}]\label{app: lemma 2 in jacion algebra invaraints}
We first show that the morphism constructed in \ref{lemma 2 in jacion algebra invaraints} is an epimorphism. It is enough to prove the following two facts:
    \begin{align}\label{1 in lem in jacobian algebra mutation invariants}
       \overline{T_H(\widetilde{A})}_{\hat{k}, \hat{k}}=\overline{T_H(\widetilde{A}_{\hat{k}, \hat{k}})}+J(\widetilde{S})_{\hat{k}, \hat{k}}
      \end{align} 
      \begin{align}\label{2 in lem in jacobian algebra mutation invariants}
      \left[J(S)_{\hat{k}, \hat{k}}\right] \subseteq \overline{T_H(\widetilde{A}_{\hat{k}, \hat{k}})} \cap J(\widetilde{S})_{\hat{k}, \hat{k}} 
    \end{align}
    
To show (\ref{1 in lem in jacobian algebra mutation invariants}), we note that if a path  $\varepsilon_{ha_1}^{l_1}\widetilde{a}_{1} \cdots \widetilde{a}_{d}\varepsilon_{ta_d}^{l_{d+1}}\in \overline{T_H(\widetilde{A})}_{\hat{k}, \hat{k}}$ does not belong to $\overline{T_H(\widetilde{A}_{\hat{k}, \hat{k}})}$ then it must contain one or more factors of the form $a^{\star}\varepsilon_k^l b^{\star}$ with $h(a)=t(b)=k$. In view of (\ref{mutation of s}), we have
\begin{align}\label{3 in lem in jacobian algebra mutation invariants}
    a^{\star}\varepsilon_k^l b^{\star}=\partial_{[b \varepsilon_k^{d_k-1-l}a]} \widetilde{\s}-\partial_{[b \varepsilon_k^{d_k-1-l} a]}[\s] .
\end{align}
Substituting this expression for each factor $a^{\star} \varepsilon_k^l b^{\star}$, we see that $\varepsilon_{ha_1}^{l_1}\widetilde{a}_{1} \cdots \widetilde{a}_{d}\varepsilon_{ta_d}^{l_{d+1}} \in \overline{T_H(\widetilde{A}_{\hat{k}, \hat{k}})}+J(\widetilde{S})_{\hat{k}, \hat{k}}$, as desired.

To show (\ref{2 in lem in jacobian algebra mutation invariants}), we note that $J(S)_{\hat{k}, \hat{k}}$ is easily seen to be the closure of the ideal in $\overline{T_H(A)}_{\hat{k}, \hat{k}}$ generated by the elements $\partial_{c} S$ for all arrows $c \in Q^\circ_{1}$ with $t(c) \neq k$ and $h(c) \neq k$, together with the elements $\left(\partial_{a} S\right) \varepsilon_k^la^{\prime}$ for $a, a^{\prime} \in Q^\circ_{1} \cap e_{k} A$, and $b^{\prime}\varepsilon_k^l \left(\partial_{b} S\right)$ for $b, b^{\prime} \in Q^\circ_{1} \cap A e_{k}$. Let us apply the map $u \mapsto[u]$ to these generators. First, we have: $\left[\partial_{c} \s\right]=\partial_{c}\s$. 
With a little bit more work (using (\ref{3 in lem in jacobian algebra mutation invariants})), we obtain 
\begin{align*}
    \left[\left(\partial_{a} \s\right)\varepsilon_k^l a^{\prime}\right]
    & =\sum_{\substack{t(b)=k\\0\leq s\leq d_k-1}}\left(\partial_{[b\varepsilon_k^s a]}[\s]\right)\left[b\varepsilon_k^{l+s} a^{\prime}\right] \\
    & =\sum_{\substack{t(b)=k\\0\leq s\leq d_k-1}}\left(\partial_{[b \varepsilon_k^s a]} \widetilde{\s}-a^{\star}\varepsilon_k^{d_k-1-s} b^{\star}\right)\left[b \varepsilon_k^{l+s} a^{\prime}\right] \\
    &=\sum_{\substack{t(b)=k\\0\leq s\leq d_k-1}}\partial_{[b \varepsilon_k^s a]} \widetilde{\s}\left[b \varepsilon_k^{l+s} a^{\prime}\right]-\sum_{\substack{t(b)=k\\ l\leq s\leq d_k-1}} a^{\star}\varepsilon_k^{d_k-1-s+l}b^{\star}\left[b \varepsilon_k^s a'\right]\\
    & =\sum_{\substack{t(b)=k\\ 0\leq s\leq d_k-1}}\partial_{[b \varepsilon_k^s a]} \widetilde{\s}\left[b \varepsilon_k^{l+s} a^{\prime}\right] -a^{\star}\varepsilon_k^{l}\sum_{\substack{t(b)=k\\0\leq s\leq d_k-1}}\varepsilon_k^{d_k-1-s}b^{\star}\left[b \varepsilon_k^s a'\right]\\
    & =\sum_{\substack{t(b)=k\\0\leq s\leq d_k-1}}\partial_{[b \varepsilon_k^s a]} \widetilde{\s}\left[b \varepsilon_k^{l+s} a^{\prime}\right] -a^{\star}\varepsilon_k^l\partial_{a^{\prime \star}}\widetilde{\s}, 
\end{align*}and 
\begin{align*}
    \left[b^{\prime}\varepsilon_k^l \left(\partial_{b} \s\right)\right]  
    & =\sum_{\substack{h(a)=k\\0\leq s\leq d_k-1}}\left[b^{\prime} \varepsilon_k^{l+s}a\right]\left(\partial_{[b\varepsilon_k^s a]}[\s]\right) \\
    & =\sum_{\substack{h(a)=k\\0\leq s\leq d_k-1}}\left[b^{\prime} \varepsilon_k^{l+s} a\right]\left(\partial_{[b\varepsilon_k^s a]} \widetilde{\s}-a^{\star} \varepsilon_k^{d_k-1-s}b^{\star}\right) \\
   & =\sum_{\substack{h(a)=k\\0\leq s\leq d_k-1}}\left[b^{\prime}\varepsilon_k^{l+s} a\right]\left(\partial_{[b\varepsilon_k^s a]} \widetilde{S}\right)-\left(\partial_{b^{\prime} \star} \widetilde{S}\right)\varepsilon_k^l b^{\star}.
\end{align*}
This implies the desired inclusion in (\ref{2 in lem in jacobian algebra mutation invariants}). We have now shown that the morphism (let us denote it by $\alpha$ ) in Lemma \ref{lem in jacobian algebra mutation invariants} is an epimorphism. We next prove it is actually an isomorphism. To do this, we construct the left inverse algebra homomorphism $\beta: \mathcal{P}(H,\widetilde{A}, \widetilde{\s})_{\hat{k}, \hat{k}} \rightarrow \mathcal{P}(H,A,\s)_{\hat{k}, \hat{k}}$ (so that $\beta \alpha$ is the identity map on $\mathcal{P}(H,A,\s)_{\hat{k}, \hat{k}}$). We define $\beta$ as the composition of three maps. First, we apply the epimorphism $\mathcal{P}(H,\widetilde{A}, \widetilde{\s})_{\hat{k}, \hat{k}} \rightarrow \mathcal{P}(H,\widetilde{\widetilde{A}}, \widetilde{\widetilde{\s}})_{\hat{k}, \hat{k}}$ defined in the same way as $\alpha$. Remembering the proof of Theorem \ref{mutation of gqp is involution } and using the notation introduced there, we then apply the isomorphism $\mathcal{P}(H,\widetilde{\widetilde{A}}, \widetilde{\widetilde{\s}})_{\hat{k}, \hat{k}} \rightarrow \mathcal{P}(H,A \oplus C, \s+\mathcal{T})_{\hat{k}, \hat{k}}$ induced by the automorphism $\varphi_{3} \varphi_{2} \varphi_{1}$ of $\overline{T_H(A \oplus C)}$. Finally, we apply the isomorphism $\mathcal{P}(H,A \oplus C, \s+\mathcal{T})_{\hat{k}, \hat{k}} \rightarrow \mathcal{P}(H,A, \s)_{\hat{k}, \hat{k}}$ given in Proposition \ref{trival part of jacobian}.
Since all the maps involved are algebra homomorphisms, it is enough to check that $\beta \alpha$ fixes the generators $p(c)$ and $p(b \varepsilon_k^l a)$ of $\mathcal{P}(H,A,\s)_{\hat{k}, \hat{k}}$, where $p$ is the projection $\overline{T_H(A)} \rightarrow \mathcal{P}(H,A,\s)$, and $a, b, c$ have the same meaning as above. This is done by direct tracing of the definitions.
\end{proof}
\begin{proof}[Proof of Proposition \ref{finite dim is invariants}]\label{app:finite dim is invariants}
    We start by showing that finite dimensionality of $\mathcal{P}(H,A,\s)$ follows from a seemingly weaker condition.
    \begin{lem}\label{lem in finite dim is invariants}
        Let $J \subseteq \mathfrak{m}^\circ(A)$ be a closed ideal in $\overline{T_H(A)}$. Then the quotient algebra $\overline{T_H(A)}/ J$ is finite dimensional provided the subalgebra $\overline{T_H(A)}_{\hat{k}, \hat{k}} / J_{\hat{k}, \hat{k}}$ is finite dimensional. In particular, the Jacobian algebra $\mathcal{P}(H,A,\s)$ is finite-dimensional if and only if so is the subalgebra $\mathcal{P}(H,A, \s)_{\hat{k}, \hat{k}}$.
    \end{lem}
    \begin{proof}
    For an $H$-bimodule $B$, we denote
$$
B_{k, \hat{k}}=e_{k} B \bar{e}_{k}=\bigoplus_{j \neq k} B_{k, j}, \quad B_{\hat{k}, k}=\bar{e}_{k} B e_{k}=\bigoplus_{i \neq k} B_{i, k}.
$$
We need to show that if $\overline{T_H(A)}_{\hat{k}, \hat{k}} / J_{\hat{k}, \hat{k}}$ is finite dimensional then so is each of the spaces $\overline{T_H(A)}_{k, \hat{k}} / J_{k, \hat{k}}, \overline{T_H(A)}_{\hat{k}, k} / J_{\hat{k}, k}$ and $\overline{T_H(A)}_{k, k} / J_{k, k}$. Let us treat $\overline{T_H(A)}_{k, k} / J_{k, k}$; the other two cases are done similarly (and a little simpler).

Let
$$
Q^\circ_{1} \cap A_{k, \hat{k}}=\left\{a_{1}, \ldots, a_{s}\right\}, \quad Q^\circ_{1} \cap A_{\hat{k}, k}=\left\{b_{1}, \ldots, b_{t}\right\}
$$
We have
$$
\overline{T_H(A)}_{k, k}= H_k\bigoplus\left( \bigoplus_{l, m,f,h} \varepsilon_k^f a_{l} \overline{T_H(A)}_{\hat{k}, \hat{k}} b_{m}\varepsilon_k^h\right) .
$$
It follows that there is a surjective map $\alpha: K^{d_k} \times \operatorname{Mat}_{d_ks \times d_kt}\left(\overline{T_H(A)}_{\hat{k}, \hat{k}}\right) \rightarrow \overline{T_H(A)}_{k, k} / J_{k, k}$ given by
$$
\setlength{\arraycolsep}{1.5pt}
\alpha(\sum c_l, C)=p(\sum c_l \varepsilon_{k}^l+\left(\begin{array}{lllllllll}
a_{1} & a_{2} & \cdots & a_{s} &\cdots&  \varepsilon_k^{d_k-1}a_{1} & \varepsilon_k^{d_k-1}a_{2} & \cdots & \varepsilon_k^{d_k-1}a_{s} 
\end{array}\right) C\left(\begin{array}{c}
b_{1} \\
b_{2} \\
\vdots \\
b_{t}\\
\vdots\\
b_{1} \varepsilon_k^{d_k-1}\\
b_{2} \varepsilon_k^{d_k-1}\\
\vdots \\
b_{t}\varepsilon_k^{d_k-1}
\end{array}\right))
$$
where $K^{d_k}$ is $d_k$ copies of $K$, $\operatorname{Mat}_{d_ks \times d_kt}(B)$ stands for the space of $d_ks \times d_kt$ matrices with entries in $B$, and $p$ is the projection $\overline{T_H(A)} \rightarrow \overline{T_H(A)} / J$. The kernel of $\alpha$ contains the space $\operatorname{Mat}_{d_ks \times d_kt}\left(J_{\hat{k}, \hat{k}}\right)$, hence $\overline{T_H(A)}_{k, k} / J_{k, k}$ is isomorphic to a quotient of the finite-dimensional space $K^{d_k} \times\operatorname{Mat}_{s \times t}\left(\overline{T_H(A)} _{\hat{k}, \hat{k}} / J_{\hat{k}, \hat{k}}\right)$. Thus, $\overline{T_H(A)}_{k, k} / J_{k, k}$ is finite dimensional, as desired.
    \end{proof}
To finish the proof of Proposition \ref{finite dim is invariants}, suppose that $\mathcal{P}(H,A, \s)$ is finite dimensional. Then $\mathcal{P}(H,\widetilde{A}, \widetilde{\s})_{\hat{k}, \hat{k}}$ is finite dimensional by Proposition \ref{jacobian algebra mutation invariants}. Applying Lemma \ref{lem in finite dim is invariants} to the $H$-based QP $(H,\widetilde{A}, \widetilde{\s})$, we conclude that $\mathcal{P}(H,\widetilde{A}, \widetilde{\s})$ is finite dimensional, as desired.
\end{proof}
\begin{proof}[Proof of Proposition \ref{prop of eixstence of nondenerate GQP}]\label{app:prop of eixstence of nondenerate GQP}
We proceed by induction on $l$. First let us deal with the case $l=1$, that is, with a single mutation $\mu_{k}$. Recall that $\mu_{k}(H,A,\s)=(H,\overline{A}, \overline{\s})$ is the reduced part of the $H$-based QP $\widetilde{\mu}_{k}(H,A,\s)=(H,\widetilde{A}, \widetilde{\s})$ given by (\ref{mutation of A}) and (\ref{mutation of s}). It is clear from the definition that $\widetilde{\s}=\widetilde{G}(\s)$ for a polynomial map $\widetilde{G}: K^{\mathcal{C}(A)} \rightarrow K^{\mathcal{C}(\widetilde{A})}$. Now let us apply Proposition \ref{condition of 2-cyclic} to the quiver with the arrow span $\widetilde{A}$. We see that there exists a polynomial function of the form $D_{c_{1}, \ldots, c_{N}}^{d_{1}, \ldots, d_{N}}$ on $K^{\mathcal{C}(\widetilde{A})}$ (see (\ref{det}), where we have changed the notation for the arrows to avoid the notation conflict with Section \ref{section:mutation of GQP}) such that the reduced part $(H,\overline{A}, \overline{\s})$ of an $H$-based QP $(H,\widetilde{A}, \widetilde{\s})$ is 2-acyclic whenever $\widetilde{\s} \in U\left(D_{c_{1}, \ldots, c_{N}}^{d_{1}, \ldots, d_{N}}\right)$. Furthermore, for $\widetilde{S} \in U\left(D_{c_{1}, \ldots, c_{N}}^{d_{1}, \ldots, d_{N}}\right)$, the $H$-based QP $(H,\overline{A}, \overline{\s})$ is right-equivalent to $\left(H,A^{\prime}, E(\widetilde{\s})\right)$ for some regular map $E: U\left(D_{c_{1}, \ldots, c_{N}}^{d_{1}, \ldots, d_{N}}\right) \rightarrow K^{\mathcal{C}\left(A^{\prime}\right)}$, where $A^{\prime}=\mu_{k}(A)$. We now define a polynomial function $F: K^{\mathcal{C}(A)} \rightarrow K$ and a regular map $G: U(F) \rightarrow K^{\mathcal{C}\left(A^{\prime}\right)}$ by setting
\begin{align}\label{1 in prop of eixstence of nondenerate GQP}
    F=D_{c_{1}, \ldots, c_{N}}^{d_{1}, \ldots, d_{N}} \circ \widetilde{G}, \quad G=E \circ \widetilde{G}
\end{align}
To finish the argument for $l=1$, it remains to show that $F$ is not identically equal to zero. But this is clear from the definitions (\ref{det}) and (\ref{mutation of s}), since the oriented 2-cycles in $\widetilde{A}^\circ$ (up to cyclical equivalence) are of the form $c[b \varepsilon_k^s a]$ and so are in a bijection with the oriented 3-cycles $c b \varepsilon_k^s a$ in $A$ that pass through $k$.

Now assume that $l \geq 2$, and that our assertion holds if we replace $l$ by $l-1$. Let $A_{1}=\mu_{k_{1}}(A)$, so $A^{\prime}=\mu_{k_{l}} \cdots \mu_{k_{2}}\left(A_{1}\right)$. By the inductive assumption, there exist a nonzero polynomial function $F^{\prime}: K^{\mathcal{C}\left(A_{1}\right)} \rightarrow K$ and a regular map $G^{\prime}: U\left(F^{\prime}\right) \rightarrow K^{\mathcal{C}\left(A^{\prime}\right)}$ such that, for any $H$-based QP $\left(H,A_{1},\s_{1}\right)$ with $\s_{1} \in U\left(F^{\prime}\right)$, the $H$-based QP $\mu_{k_{l}} \cdots \mu_{k_{2}}\left(H,A_{1}, \s_{1}\right)$ is right-equivalent to $\left(H,A^{\prime}, G^{\prime}\left(\s_{1}\right)\right)$. Also by the already established case $l=1$, there exists a non-zero polynomial function $F^{\prime \prime}: K^{\mathcal{C}\left(A_{1}\right)} \rightarrow K$ such that, for any $H$-based QP $\left(H,A_{1}, \s_{1}\right)$ with $\s_{1} \in U\left(F^{\prime \prime}\right)$, the $H$-based QP $\mu_{k_{1}}\left(H,A_{1},\s_{1}\right)$ is 2-acyclic, hence is right-equivalent to some $H$-based QP on $A$. Since the based field $K$ is assumed to be infinite, we have $U\left(F^{\prime}\right) \cap U\left(F^{\prime \prime}\right) \neq \emptyset$. Choose $\s_{1}^{(0)} \in U\left(F^{\prime}\right) \cap U\left(F^{\prime \prime}\right)$, and let $\left(H,A, \s_{0}\right)=\mu_{k}\left(H,A_{1},\s_{1}^{(0)}\right)$. By Theorem \ref{mutation of gqp is involution }, we have $\mu_{k}\left(H,A,\s_{0}\right)=\left(H,A_{1}, \s_{1}^{(0)}\right)$. By the above argument for $l=1$, there exist a nonzero polynomial function $F_{1}: K^{\mathcal{C}(A)} \rightarrow K$ and a regular map $G_{1}: U\left(F_{1}\right) \rightarrow K^{\mathcal{C}\left(A_{1}\right)}$ (of the type (\ref{1 in prop of eixstence of nondenerate GQP})) such that $\mu_{k}(H,A,\s)=\left(H,A_{1}, G_{1}(\s)\right)$ for $\s\in U\left(F_{1}\right)$. In particular, we have $G_{1}\left(\s_{0}\right)=\s_{1}^{(0)}$ implying that $F^{\prime} \circ G_{1}$ is a nonzero polynomial function on $K^{\mathcal{C}(A)}$. It follows that the nonzero polynomial function $F(\s)=F_{1}(\s) F^{\prime}\left(G_{1}(\s)\right)$ and the regular map $G=G^{\prime} \circ G_{1}: U(F) \rightarrow K^{\mathcal{C}\left(A^{\prime}\right)}$ are well-defined and satisfy all the required conditions. This completes the proof of Proposition \ref{prop of eixstence of nondenerate GQP}.
\end{proof}

\begin{proof}[Proof of Proposition \ref{mutation of right-equivalence class of the representation}]\label{app: mutation of right-equivalence class of the representation}
Let $\varphi$ be an automorphism of $\overline{T_H(A)}$, and let $\mathcal{M}^{\prime}=\left(H,A, \varphi(\s), M^{\prime}, V^{\prime}\right)$ be the $H$-based QP-representation defined as follows: $V^{\prime}=V$ and $M^{\prime}=M$ as $H$-modules, while the $\overline{T_H(A)}$-actions in $M$ and $M^{\prime}$ are related by
     \begin{align}\label{eq in mutation of right-equivalence class of the representation}
        M(u)=M'(\varphi(u)) \quad(u \in \overline{T_H(A)}) 
     \end{align}
(note that we indeed define a representation of $(H,A,\varphi(\s))$ in view of Proposition \ref{iso of jacobian algebra}). To prove Proposition \ref{mutation of right-equivalence class of the representation}, it suffices to show that the representations $\widetilde{\mu}_{k}(\mathcal{M})$ and $\widetilde{\mu}_{k}\left(\mathcal{M}^{\prime}\right)$ are right-equivalent.

We retain all the above notation related to $\mathcal{M}$ and $\widetilde{\mu}_{k}(\mathcal{M})$; in particular, $\alpha_k, \beta_k$ and $\gamma_k$ stand for the $H_k$-linear maps in the triangle (\ref{triangle}). Let $\alpha^{\prime}_k, \beta^{\prime}_k$ and $\gamma^{\prime}_k$ denote the corresponding maps for the representation $\mathcal{M}^{\prime}$. Our first order of business is to relate these maps to $\alpha_k, \beta_k$ and $\gamma_k$.

Recall (\ref{action of varphi on a}) and (\ref{action of varphi on b}), after a change of notation we can write the action of $\varphi$ on the arrows $a_{1}, \ldots, a_{s}$ as follows:
\begin{equation}\label{varphi on a}
\setlength{\arraycolsep}{1.5pt}
\begin{aligned}
    \left(\begin{array}{llll}
\varphi\left(a_{1}\right), & \varphi\left(a_{2}\right), & \cdots, & \varphi\left(a_{s}\right)
\end{array}\right)&=\sum_{0\leq l\leq d_k-1}\varepsilon_k^l\left(\begin{array}{llll}
a_{1}, & a_{2}, & \cdots, & a_{s}
\end{array}\right)C_{l}
\end{aligned}
\end{equation}
where $C_0$ is an invertible $s\times s$ matrix. Similarly, the action of $\varphi$ on the arrows $b_{1}, \ldots, b_{s}$ can be written as
\begin{equation}\label{varphi on b}
\begin{aligned}
    \left(\begin{array}{c}
\varphi\left(b_{1}\right) \\
\varphi\left(b_{2}\right) \\
\vdots \\
\varphi\left(b_{t}\right)
\end{array}\right)=\sum_{0\leq l\leq d_k-1}F_l\left(\begin{array}{c}
b_{1} \\
b_{2} \\
\vdots \\
b_{t}
\end{array}\right)\varepsilon_k^{l},
\end{aligned}
\end{equation}
where $F_0$ is an invertible $s\times s$ matrix. Therefore, recalling (\ref{defn of alpha}) and using the language of linear algebra, we express  $\alpha_k$ as a matrix 
\begin{equation}\label{relate of alpah and alpha'}
\setlength{\arraycolsep}{1.5pt}
    \begin{aligned}
    \alpha_k&=M\left(\begin{array}{lllllllll}
(a_{1}, & a_{2}, & \cdots, & a_{s}), &\cdots, & \varepsilon_k^{d_k-1}(a_{1}, & a_{2}, & \cdots, & a_{s})
\end{array}\right)\\
&=M'\left(\begin{array}{lllllllll}
(\varphi\left(a_{1}\right), & \varphi\left(a_{2}\right), & \cdots, & \varphi\left(a_{s}\right)),&\cdots,&\varepsilon_k^{d_k-1}(\varphi(a_{1}), &\varphi(a_{2}),& \cdots, & \varphi(a_{s}))
\end{array}\right)\\
&=\alpha'_k M'(C)
\end{aligned}
\end{equation}
where $$C=\left(\begin{array}{llll}
    C_0 &0&\cdots&0 \\
    C_1 & C_0&\cdots& 0\\
    \vdots&\vdots &\vdots&\vdots\\
    C_{d_k-1}&C_{d_k-2}&\cdots&C_0
\end{array}\right).$$
Similarly,
\begin{align}\label{relate of beta and beta'}
   \beta_k=M\left(\begin{array}{ccc}
        \varepsilon_k^{d_k-1}\left(\begin{array}{cc}
            b_1 \\
        b_2\\
        \vdots\\
        b_t
        \end{array}\right)\\
        \vdots\\
        \left(\begin{array}{cc}
            b_1 \\
        b_2\\
        \vdots\\
        b_t
        \end{array}\right)
   \end{array}\right)=M'(F)\beta'_k, 
\end{align}
where 
$$F=\left(\begin{array}{llll}
    F_0 &0&\cdots&0 \\
    F_1 & F_0&\cdots& 0\\
    \vdots&\vdots &\ddots&\vdots\\
    F_{d_k-1}&F_{d_k-2}&\cdots&F_0
\end{array}\right).$$
Here $M'(C)$ and $M'(F)$ are understood as an $H_k$-automorphism of $M'_{k,\In}=M_{k,\In}$ and $M'_{k,\out}=M_{k,\out}$, respectively. 

Turning to the maps $\gamma_k$ and $\gamma^{\prime}_k$, we claim that they are related by
\begin{align}\label{relate of gamma and gamma'}
    \gamma^{\prime}_k=M'(C) \gamma_k M'(F).
\end{align}
 Recalling (\ref{defn of gamma}) and using the language of linear algebra, we can express $\gamma_k$ as a matrix  
 \begin{align*}
     \gamma_k=\left(\begin{array}{llll}
         \gamma^0 & 0 & \cdots & 0  \\
         \gamma^1 & \gamma^0 &\cdots & 0\\
         \vdots&\vdots&\ddots&\vdots\\
         \gamma^{d_k-1} &\gamma^{d_k-2} &\cdots &\gamma^0
     \end{array}\right)
 \end{align*}
where $\gamma^l$ is an $s\times t$ matrix whose component $\gamma^l_{p,q}:M_{hb_q}\to M_{ta_p}$ is given by
\begin{align}
\gamma^l_{p,q}=M(\partial_{[b_q\varepsilon_k^la_p]}\s).
\end{align}
We use Lemma \ref{chain rule} to write
\begin{align}\label{defn of gamma(p,q,l)}
   (\gamma_{p, q}^{l})'
   &=M'\left(\partial_{\left[b_{q} \varepsilon_k^l a_{p}\right]} \varphi(S)\right) \\
   &=M'\left(\sum_{c} \Delta_{\left[b_{q} \varepsilon_k^l a_{p}\right]}(\varphi(c)) \square \varphi\left(\partial_{c} S\right)\right) 
\end{align}
where the sum is over all arrows $c$ in $(\widetilde{A})^\circ_{\hat{k}, \hat{k}}$. If $c$ is one of the arrows in $A^\circ$, then by (\ref{eq in mutation of right-equivalence class of the representation}) we have
\begin{align}
    M'(\varphi\left(\partial_{c} S\right))=M\left(\partial_{c} S\right)=0;
\end{align}
remembering the definition (\ref{defn of delta}), we see that $c$ does not contribute to (\ref{defn of gamma(p,q,l)}). Thus, we have
\begin{align}\label{gamma(p,q,l)prime}
    (\gamma_{p, q}^{l})'=M'\left(\sum_{p^{\prime}, q^{\prime},f} \Delta_{\left[b_{q} \varepsilon_k^la_{p}\right]}\left(\varphi\left(b_{q^{\prime}}\varepsilon_k^f a_{p^{\prime}}\right)\right) \square \varphi\left(\partial_{\left[b_{q^{\prime}}\varepsilon_k^f a_{p^{\prime}}\right]} S\right)\right).
\end{align}
Remembering (\ref{defn of  square}), and using (\ref{varphi on a}) and (\ref{varphi on b}), we see that the summand with $\left(p^{\prime}, q^{\prime}\right)=(p, q)$ in (\ref{gamma(p,q,l)prime}) contains among its terms the $(p, q)$-entry of the matrix
$$
M'\left(\sum_{f+g+h=l}C_f \varphi\left(\partial_{\left[b_{q} \varepsilon_k^g a_{p}\right]} \s\right) F_h\right)=\sum_{f+g+h=l} M'(C_f) \gamma^g_{p,q} M'(F_h).
$$
Thus, to prove (\ref{relate of gamma and gamma'}), it remains to show that the rest of the terms in (\ref{gamma(p,q,l)prime}) add up to 0. Again using the definitions (\ref{defn of delta}) and (\ref{defn of  square}), we can rewrite the rest of the sum in (\ref{gamma(p,q,l)prime}) as $S_{1}+S_{2}$, where
$$
\begin{aligned}
& S_{1}=M'\left(\sum_{p^{\prime}, q^{\prime},f} \Delta_{\left[b_{q} \varepsilon_k^la_{p}\right]}\left(\varphi\left(b_{q^{\prime}}\right)\right) \square \varphi\left(\varepsilon_k^fa_{p^{\prime}} \cdot \partial_{\left[b_{q^{\prime}} \varepsilon_k^fa_{p^{\prime}}\right]} S\right)\right),\\
& S_{2}=M'\left(\sum_{p^{\prime}, q^{\prime},f} \Delta_{\left[b_{q} \varepsilon_k^la_{p}\right]}\left(\varphi\left(a_{p^{\prime}}\right)\right) \square \varphi\left(\partial_{\left[b_{q^{\prime}} \varepsilon_k^fa_{p^{\prime}}\right]} S \cdot b_{q^{\prime}}\varepsilon_k^f\right)\right).
\end{aligned}
$$
It remains to observe that
$$
S_{1}=M'\left( \sum_{q^{\prime}} \Delta_{\left[b_{q} \varepsilon_k^la_{p}\right]}\left(\varphi\left(b_{q^{\prime}}\right)\right) \square \varphi\left(  \partial_{b_{q^{\prime}}} S\right)\right)=0
$$
since $M'(\varphi\left(\partial_{b_{q^{\prime}}} S\right))=M\left(\partial_{b_{q^{\prime}}} S\right)=0$; and similarly,
$$
S_{2}=M'\left(\sum_{p^{\prime}} \Delta_{\left[b_{q} \varepsilon_k^la_{p}\right]}\left(\varphi\left(a_{p^{\prime}}\right)\right) \square \varphi\left(\partial_{a_{p^{\prime}}} S\right)\right)=0.
$$
In view of (\ref{relate of alpah and alpha'}), (\ref{relate of beta and beta'}) and (\ref{relate of gamma and gamma'}), we have 
\begin{equation}\label{relate of ker alpha}
    \begin{aligned}
    \setlength{\arraycolsep}{5pt}
    \begin{array}{ll}
        \ker \alpha_k= M'(C^{-1})(\ker \alpha_k^\prime),& \Img \alpha_k=\Img \alpha_k^\prime,\\
        \ker \beta_k=\ker \beta_k^\prime,& \Img \beta_k=M'(F)(\Img\beta_k^\prime)\\
        \ker \gamma_k=M'(F)(\ker \gamma_k^\prime),& \Img \gamma_k=M'(C^{-1})(\Img \gamma_k^\prime).
    \end{array}
    \end{aligned}
\end{equation}
Recall that the spaces $\overline{M}$ and $\overline{V}$ in the decorated representation $\widetilde{\mu}_{k}(\mathcal{M})=(\overline{M}, \overline{V})$ of $(H,\widetilde{A}, \widetilde{\s})$ are given by (\ref{mutation of Mk 1}) and (\ref{mutation of Mk}). We express the decorated representation $\widetilde{\mu}_{k}\left(\mathcal{M}^{\prime}\right)=\left(\overline{M^{\prime}}, \overline{V^{\prime}}\right)$ of $(H,\widetilde{A}, \widetilde{\varphi(\s)})$ in the same way, with the maps $\alpha_k, \beta_k$ and $\gamma_k$ replaced by $\alpha^{\prime}_k, \beta^{\prime}_k$ and $\gamma^{\prime}_k$. In particular, we have $\overline{V^{\prime}}=\overline{V}$, and $\overline{M^{\prime}}{ }_{i}=\overline{M}_{i}=M_{i}$ for $i \neq k$. To specify the actions of $\overline{T_H(A)}$ in $\overline{M}$ and $\overline{M^{\prime}}$, we need to choose the splitting data $(\rho, \sigma)$ and $\left(\rho^{\prime}, \sigma^{\prime}\right)$ as in (\ref{splitting data rho}) and (\ref{soplitting data sigma}). Note that, in view of (\ref{relate of ker alpha}), we can choose
\begin{align}\label{splitting data of alpha'}
    \rho^{\prime}=M'(F^{-1}) \rho M'(F), \quad \sigma^{\prime}=M'(C) \sigma M(C^{-1});
\end{align}
here with some abuse of notation we use the same notation $M'(C^{-1})$ for the isomorphism  $\ker\alpha_k^{\prime} \rightarrow\ker\alpha_k$ and the induced isomorphism $\ker \alpha_k^{\prime} / \Img \gamma_k^{\prime} \rightarrow \ker \alpha_k / \Img_k \gamma_k$.

Everything is now in place for defining the desired right-equivalence $(\widehat{\varphi}, \psi, \eta)$ between $\widetilde{\mu}_{k}(\mathcal{M})$ and $\widetilde{\mu}_{k}\left(\mathcal{M}^{\prime}\right)$. First of all, we define $\widehat{\varphi}: \overline{T_H(\widetilde{A})} \rightarrow\overline{T_H(\widetilde{A})}$ as the right-equivalence between $(H,\widetilde{A}, \widetilde{\s})$ and $(H,\widetilde{A}, \widetilde{\varphi(\s)})$ constructed in the proof of Lemma \ref{lem of mutation uniquele right-equivalence class of the GQP}. In particular, we have
\begin{align*}
    \left(\begin{array}{c}
        \varepsilon_k^{d_k-1}\left(\begin{array}{c}
\widehat{\varphi}\left(a_{1}^{\star}\right) \\
\widehat{\varphi}\left(a_{2}^{\star}\right) \\
\vdots \\
\widehat{\varphi}\left(a_{s}^{\star}\right)
\end{array}\right)\\
\vdots\\
\left(\begin{array}{c}
\widehat{\varphi}\left(a_{1}^{\star}\right) \\
\widehat{\varphi}\left(a_{2}^{\star}\right) \\
\vdots \\
\widehat{\varphi}\left(a_{s}^{\star}\right)
\end{array}\right)
    \end{array}\right)=C^{-1}\left(\begin{array}{c}
        \varepsilon_k^{d_k-1}\left(\begin{array}{c}
            a_1^\star \\
        a_2^\star\\
        \vdots\\
        a_s^\star
        \end{array}\right)\\
        \vdots\\
        \left(\begin{array}{c}
           a_1^\star \\
        a_2^\star\\
        \vdots\\
        a_s^\star
        \end{array}\right)
   \end{array}\right)
\end{align*}
\begin{align*}
    &\left(\begin{array}{lllllllll}
(\widehat{\varphi}(b_{1}^\star), & \widehat{\varphi}(b_{2}^\star), & \cdots, & \widehat{\varphi}(b_{t}^\star)), &\cdots, & \varepsilon_k^{d_k-1}(\widehat{\varphi}(b_{1}^\star), & \widehat{\varphi}(b_{2}^\star), & \cdots, & \widehat{\varphi}(b_{t}^\star))
\end{array}\right)\\
&=\left(\begin{array}{lllllllll}
(b_1^\star,&b_2^\star,&\cdots,&b_t^\star),&\cdots,&\varepsilon_k^{d_k-1}(b_1^\star,&b_2^\star,&\cdots,&b_t^\star)
\end{array}\right)F^{-1}
\end{align*}
Next we define $\psi: \overline{M} \rightarrow \overline{M^{\prime}}$ as the identity map on $\oplus_{i \neq k} \overline{M}_{i}=\oplus_{i \neq k} M_{i}=\oplus_{i \neq k} \overline{M^{\prime}}{ }_{i}$, and the restriction $\left.\psi\right|_{\overline{M}_{k}}: \overline{M}_{k} \rightarrow{\overline{M^{\prime}}}_{k}$ given by the block-diagonal matrix
\begin{align}\label{defn of psi}
\psi|_{\overline{M}_{k}}=\left(\begin{array}{cccc}
M'(F^{-1}) & 0 & 0 & 0 \\
0 & M'(C) & 0 & 0 \\
0 & 0 & M'(C) & 0 \\
0 & 0 & 0 & I
\end{array}\right)  
\end{align}
(this is well-defined in view of (\ref{relate of ker alpha})). Finally, we define $\eta: \overline{V} \rightarrow \overline{V^{\prime}}$ simply as the identity map. 

The only thing to check is the equality $\psi \circ \overline{M}(c)=\overline{M'}(\widehat{\varphi}(c)) \circ \psi$ for all $c \in A$. And the only case that may require some consideration is when $c$ is one of the arrows $a_{p}^{\star}$ or $b_{q}^{\star}$. Unraveling the definitions, it suffices to show that
$$
\overline{\beta}_k=M'(C^{-1}) \overline{\beta^{\prime}}_k \circ \psi|_{\overline{M}_{k}},\quad \psi|_{\overline{M}_{k}} \circ \overline{\alpha}_k=\overline{\alpha^{\prime}}_k M'(F^{-1})
$$
But this is an immediate consequence of the definitions (\ref{defn of psi}) and (\ref{defn of mutation of alpha and beta}) (we also need an analogue of (\ref{defn of mutation of alpha and beta}) for the maps $\overline{\beta^{\prime}}$ and $\overline{\alpha^{\prime}}$, using the splitting data (\ref{splitting data of alpha'})). This completes the proof of Proposition \ref{mutation of right-equivalence class of the representation}.
\end{proof}

\bibliography{Ref}

\providecommand{\bysame}{\leavevmode\hbox to3em{\hrulefill}\thinspace}
\providecommand{\MR}{\relax\ifhmode\unskip\space\fi MR }
\providecommand{\MRhref}[2]{%
  \href{http://www.ams.org/mathscinet-getitem?mr=#1}{#2}
}
\providecommand{\href}[2]{#2}
\begin{thebibliography}{10}

\bibitem{elements}
Ibrahim Assem, Daniel Simson, and Andrzej Skowro{\'n}ski, \emph{Elements of the
  representation theory of associative algebras: Techniques of representation
  theory}, London Mathematical Society Student Texts, vol.~65, Cambridge
  University Press, Cambridge, 2006.

\bibitem{berenstein2005cluster}
Arkady Berenstein, Sergey Fomin, and Andrei Zelevinsky, \emph{Cluster algebras
  {III}: Upper bounds and double {B}ruhat cells}, Duke Math. J. \textbf{126}
  (2005), no.~1, 1 -- 52.

\bibitem{BIALYNICKIBIRULA197399}
A.~Bialynicki-Birula, \emph{On fixed point schemes of actions of multiplicative
  and additive groups}, Topology \textbf{12} (1973), no.~1, 99--103.

\bibitem{cao2019}
Peigen Cao and Fang Li, \emph{On some combinatorial properties of generalized
  cluster algebras}, J. Pure Appl. Algebra \textbf{225} (2021), no.~8, 106650.

\bibitem{Chekhov2013}
Leonid Chekhov and Michael Shapiro, \emph{Teichmüller spaces of riemann
  surfaces with orbifold points of arbitrary order and cluster variables}, Int.
  Math. Res. Not. \textbf{2014} (2013), no.~10, 2746–2772.

\bibitem{cheung2022}
Man-Wai Cheung, Elizabeth Kelley, and Gregg Musiker, \emph{Cluster scattering
  diagrams and theta functions for reciprocal generalized cluster algebras},
  Ann. Comb. \textbf{27} (2023), no.~3, 615--691.

\bibitem{derksen2015general}
Harm Derksen and Jiarui Fei, \emph{General presentations of algebras}, Adv.
  Math. \textbf{278} (2015), 210--237.

\bibitem{derksen2008quivers}
Harm Derksen, Jerzy Weyman, and Andrei Zelevinsky, \emph{Quivers with
  potentials and their representations {I}: Mutations}, Sel. Math. \textbf{14}
  (2008), no.~1, 59--119.

\bibitem{derksen2010quivers}
\bysame, \emph{Quivers with potentials and their representations {II}:
  Applications to cluster algebras}, J. Amer. Math. Soc. \textbf{23} (2010),
  no.~3, 749--790.

\bibitem{dimca2004sheaves}
Alexandru Dimca, \emph{Sheaves in topology}, Universitext, Springer, Berlin
  Heidelberg, 2004.

\bibitem{fei2017cluster}
Jiarui Fei, \emph{Cluster algebras and semi-invariant rings {I}. triple flags},
  Proc. Lond. Math. Soc. \textbf{115} (2017), no.~1, 1--32.

\bibitem{fei2024crystalstructureuppercluster}
\bysame, \emph{Crystal structure of upper cluster algebras}, arXiv preprint
  arXiv:2309.08326 (2023).

\bibitem{fei2023general}
\bysame, \emph{On the general ranks of {QP} representations}, Algebr.
  Represent. Theory \textbf{28} (2025), no.~1, 47--79.

\bibitem{fomin2001}
Sergey Fomin and Andrei Zelevinsky, \emph{Cluster algebras {I}: foundations},
  J. Amer. Math. Soc. \textbf{15} (2002), no.~2, 497--529.

\bibitem{fomin2006cluster}
Sergey Fomin and Andrei Zelevinsky, \emph{Cluster algebras {IV}: Coefficients},
  Compos. Math. \textbf{143} (2007), no.~1, 112--164.

\bibitem{Geiss_2016}
Christof Geiss, Bernard Leclerc, and Jan Schr{\"o}er, \emph{Quivers with
  relations for symmetrizable {C}artan matrices {I}: Foundations}, Invent.
  Math. \textbf{209} (2016), no.~1, 61--158.

\bibitem{Gei2011}
Christof Geiß, Bernard Leclerc, and Jan Schröer, \emph{Generic bases for
  cluster algebras and the chamber ansatz}, J. Amer. Math. Soc. \textbf{25}
  (2011), no.~1, 21–76.

\bibitem{Gekhtman2017}
Misha Gekhtman, Michael Shapiro, and Alek Vainshtein, \emph{Drinfeld double of
  $gl_n$ and generalized cluster structures}, Proc. Lond. Math. Soc.
  \textbf{116} (2017), no.~3, 429–484.

\bibitem{gleitz2014}
Anne-Sophie Gleitz, \emph{Quantum affine algebras at roots of unity and
  generalised cluster algebras}, arXiv preprint arXiv:1410.2446 (2014).

\bibitem{Iwaki2015}
Kohei Iwaki and Tomoki Nakanishi, \emph{Exact wkb analysis and cluster algebras
  ii: Simple poles, orbifold points, and generalized cluster algebras}, Int.
  Math. Res. Not. \textbf{2016} (2015), no.~14, 4375–4417.

\bibitem{Mou2024}
Lang Mou, \emph{Scattering diagrams for generalized cluster algebras}, Algebra
  Number Theory \textbf{18} (2024), no.~12, 2179–2246.

\bibitem{Nakanishi_2015}
Tomoki Nakanishi, \emph{Structure of seeds in generalized cluster algebras},
  Pacific J. Math. \textbf{277} (2015), no.~1, 201--218.

\bibitem{nakanishi2015}
Tomoki Nakanishi and Dylan Rupel, \emph{Companion cluster algebras to a
  generalized cluster algebra}, arXiv preprint arXiv:1504.06758 (2015).

\bibitem{plamondon2012}
Pierre-Guy Plamondon, \emph{Generic bases for cluster algebras from the cluster
  category}, Int. Math. Res. Not. \textbf{2013} (2013), no.~10, 2368--2420.

\bibitem{rupel2013}
Dylan Rupel, \emph{Greedy bases in rank 2 generalized cluster algebras}, arXiv
  preprint arXiv:1309.2567 (2013).

\end{thebibliography}
\bibliographystyle{amsplain}
\end{document}